\documentclass[a4paper]{article}
\usepackage[margin=25mm]{geometry}
\pdfoutput=1

\usepackage{hyperref}
\hypersetup{colorlinks}
\usepackage{graphicx}
\usepackage{amsmath}
\usepackage{amsthm}
\usepackage{amssymb}
\usepackage{mathtools}
\usepackage{tikz}
\usepackage{comment}
\usepackage{booktabs}
\usepackage{placeins} % for floatbarrier
\usepackage{ifthen}
\usepackage{pgfplots}
\usepackage{subfig}
\usepackage{array}
\usepackage{doi}
\pgfplotsset{compat=1.18}
\allowdisplaybreaks % allow equations to break over leaf

\numberwithin{equation}{section} % number equations within sections 

\newtheorem{theorem}{Theorem}[section]

\newtheorem{lemma}[theorem]{Lemma}
\newtheorem{remark}[theorem]{Remark}%
\newtheorem{assumption}[theorem]{Assumption}

\newtheorem{definition}[theorem]{Definition}

\newcommand{\half}{\frac{1}{2}}
\newcommand{\R}{\mathbb{R}}

\newcommand{\prob}{\mathbb{P}}

\newcommand{\cP}{\mathcal{P}}
\newcommand{\cT}{\mathcal{T}}

\newcommand{\polOrder}{\ell}
\newcommand{\localVEMSpace}{V_{h,\polOrder}^K}
\newcommand{\vemSpace}{V_{h,\polOrder}}
\newcommand{\enlargedVemSpace}{\widetilde{V}_{h,\polOrder}}
\newcommand{\localenlargedVemSpace}{\widetilde{V}_{h,\polOrder}^K}
\newcommand{\HTwoNCSpace}{H^{2,nc}_{\polOrder}(\cT_h)}

\newcommand{\valueProj}{\Pi^K_0}
\newcommand{\gradProj}{\Pi^K_1}
\newcommand{\hessProj}{\Pi^K_2}
\newcommand{\edgeProj}{\Pi^e_0}
\newcommand{\edgeNormalProj}{\Pi^{e}_1}

\newcommand{\edgeValueDofs}{d_0^e}

\newcommand{\innerDofs}{d_0^i}

\newcommand{\globalInterpolation}{{I}_h}
\newcommand{\Confinterpolation}{{I}^1_h}
\newcommand{\stabilisation}{S^K}
\newcommand{\vertiii}[1]{{\left\vert\kern-0.25ex\left\vert\kern-0.25ex\left\vert #1 
    \right\vert\kern-0.25ex\right\vert\kern-0.25ex\right\vert}}

\newcommand{\red}[1]{\textcolor{red}{#1}}

\newcommand{\ellipticProj}{P_h}
\newcommand{\ltwo}{\cP^K}

\providecommand{\keywords}[1]
{
  \small	
  \textbf{Keywords:} #1
}

\newcommand{\bh}{b_h}
\newcommand{\ahlower}{m_h}
\newcommand{\ahlowerLocal}{m_{h}^K}
\newcommand{\ahhessian}{a_h}
\newcommand{\ahhessianLocal}{a_{h}^K}
\newcommand{\rhLocal}{r_{h}^K}
\newcommand{\rh}{r_h}

\newcommand{\mesh}{\cT_h}
\newcommand{\eps}{\varepsilon}
\newcommand{\dx}{\mathrm{d}x}
\newcommand{\ds}{\mathrm{d}s}
\newcommand{\nonconformity}{\mathcal{N}}
\newcommand{\element}{K}

\newboolean{thesis}
\newboolean{thesiscorrections}
\newboolean{JSCcorrections}
\newboolean{arxiv}

\providecommand{\keywords}[1]
{
  \small	
  \textbf{Keywords:} #1
}

\title{A higher order nonconforming virtual element method for the Cahn-Hilliard equation}

\author{Andreas Dedner\thanks{Department of Mathematics, University of Warwick, Coventry, CV4 7AL, UK. Email: \href{mailto:a.s.dedner@warwick.ac.uk}{a.s.dedner@warwick.ac.uk}} \
and Alice Hodson\thanks{Corresponding author. Charles University, Faculty of Mathematics and Physics, Sokolovsk\'{a} 83, 186 75, Praha, Czech Republic. Email: \href{mailto:hodson@karlin.mff.cuni.cz}{hodson@karlin.mff.cuni.cz}} 
}

\date{}

\begin{document}

\maketitle

\begin{abstract}
    In this paper we develop a fully nonconforming virtual element method (VEM) of arbitrary approximation order for the two dimensional Cahn-Hilliard equation. 
    We carry out the error analysis for the semidiscrete (continuous-in-time) scheme and verify the theoretical convergence result via numerical experiments. 
    We present a fully discrete scheme which uses a convex splitting Runge-Kutta method to discretize in the temporal variable alongside the virtual element spatial discretization.
\end{abstract}

\hspace{10pt}

\keywords{virtual element method; Cahn-Hilliard equation; nonlinear; fourth-order problems; nonconforming; DUNE}
% \subclass{65M12 \and 65M60}

% set booleans
\setboolean{thesis}{false}
\setboolean{thesiscorrections}{false}
\setboolean{JSCcorrections}{false}
\setboolean{arxiv}{true}

\ifthenelse{\boolean{thesiscorrections}}{
        \newcommand{\corrections}[1]{\textcolor{red}{#1}}}
    {
        \newcommand{\corrections}[1]{\textcolor{black}{#1}}
    }

\ifthenelse{\boolean{JSCcorrections}}{
        \newcommand{\jsccorrections}[1]{\textcolor{red}{#1}}}
    {
        \newcommand{\jsccorrections}[1]{\textcolor{black}{#1}}
    }

\section{Introduction}\label{sec: intro}
Let $\Omega \subset \R^2$ denote a polygonal domain, with boundary $\partial \Omega$ and outward pointing normal $n$. Originally introduced by Cahn and Hilliard in \cite{cahn1961spinodal,cahn1958free} to model the phase separation of a binary alloy, we consider the following two dimensional Cahn-Hilliard problem: find $u(x,t) : \Omega \times [0,T] \rightarrow \R$ such that 
\begin{subequations}\label{eqn: CH eqn}
    \begin{alignat}{3}
        &\partial_t u - \Delta ( \phi(u) - \eps^2 \Delta u ) = 0 \quad &&\text{ in } \Omega \times (0,T],&& 
        \\
        &u(\cdot,0) = u_0(\cdot) \quad &&\text{ in } \Omega,&&
        \\
        &\partial_n u  =  0, \, \partial_n (\phi(u)- \eps^2 \Delta u) = 0 \qquad &&\text{ on } \partial \Omega \times (0,T],&&
    \end{alignat}
\end{subequations} 
for time $T>0$. We use the notation $\partial_n$ for denoting the normal derivative and $\eps>0$ to represent the interface parameter. We define $\phi (x) = \psi^{\prime} (x)$, where the free energy $\psi : \R \rightarrow \R$ is defined as
\begin{align}
    \psi(x) := \frac{1}{4} (1-x^2)^2.
\end{align}

As described in \cite{cahn1961spinodal,cahn1958free,elliott1989cahn}, phase separation is a physical phenomenon occurring when a high temperature mixture is cooled down quickly and the two or more components in the mixture separate into regions of each one.
As well as being used as a model for this type of phenomena, Cahn-Hilliard type equations have been used in a wide range of problems such as image processing \cite{dolcetta2002area},  for example.
Due to the numerous applications of the Cahn-Hilliard equation, there has been a lot of attention and research dedicated to numerical methods for \corrections{this} problem.

Many \corrections{classical} methods used to solve the Cahn-Hilliard equation  have been finite element (FE) based methods.
These approaches can be split into two types.
In the first, the equations are reformulated in mixed form, resulting in a system of second order problems which can be solved using classical methods suitable for elliptic problems, e.g., \cite{elliott1989second,kay2009discontinuous}. 
The second approach involves solving the equations directly in their weak form, requiring the use of second order derivatives of the finite element functions. 
Solving the variational formulation of the fourth-order problem directly using FE methods is not straightforward due to the higher regularity requirements which have to be imposed on the finite element basis functions. 
Fully conforming methods \cite{elliott1987numerical,elliott_cahn-hilliard_1986} require a large number of degrees of freedom even for the lowest order approximation or are based on sub-triangulation. 
An easier approach is based on using suitable nonconforming spaces and to possibly include stabilization terms to achieve stability \cite{choo2005discontinuous,wells2006discontinuous}.
A few nonconforming spaces have been suggested which have sufficient regularity to be stable without extra penalty terms 
\cite{elliott1989nonconforming,zhang2010nonconforming}.
However, higher order versions of these spaces are not easily obtained. 
Consequently, there are few methods for fourth-order problems readily available in software packages, with most only providing the lowest order Morley element for these types of problems. 

More recently, we have seen a handful of virtual element methods to discretize the Cahn-Hilliard equation \cite{antonietti_$c^1$_2016,liu2019virtual,liu2020fully}. The virtual element method is an extension and generalization of both finite element and mimetic difference methods. 
First introduced for second order elliptic problems in \cite{beirao_da_veiga_basic_2013}, virtual elements are highly desirable due to the straightforward way in which they extend to general polygonal meshes. 
The virtual element method is incredibly versatile and as such has been applied to a wide range of problems; for example, the development of higher order continuity spaces \cite{antonietti2021review,da2014virtual}, even in arbitrary dimensions \cite{huang2021h}, suitable for e.g. the approximation of polyharmonic problems \cite{antonietti_conforming_2018} as well as the construction of pointwise divergence-free spaces for the Stokes problem \cite{da_veiga_divergence_2015}.
Other methods which have been considered for the discretization of problem \eqref{eqn: CH eqn} include the hybrid high-order (HHO) method (see \cite{di2015hybrid} where the method was first introduced for a linear elasticity problem) as well as isogeometric analysis \cite{gomez2008isogeometric,kastner2016isogeometric}. 
In \cite{chave2016hybrid} an HHO approximation of the Cahn-Hilliard equation in mixed form is considered and, like the VEM approach, extends easily to general polygonal meshes. 

A VEM discretization for problem \eqref{eqn: CH eqn} is considered in \cite{antonietti_$c^1$_2016} where a $C^1$-conforming method is presented with only 3 degrees of freedom (dof) per vertex. 
Another conforming approach is considered in \cite{liu2019virtual} however the problem is formulated in mixed form. 
Both works present a semidiscrete (continuous-in-time) convergence result for the lowest order VEM space. 
\corrections{The only other applications of the virtual element method to the Cahn-Hilliard equation are seen in the following two works. In \cite{liu2020fully} a fully discrete scheme is presented and is shown to satisfy both a discrete energy law and a mass conservation law.
A numerical exploration of Cahn-Hilliard type equations is presented in \cite{antonietti2021c}.}
We also note that a $C^1$-conforming virtual element method for the high-order phase-field model is considered in \cite{adak2024ac}, however, only a linear fourth-order PDE is studied. In contrast to the virtual element discretizations developed in \cite{antonietti_$c^1$_2016,liu2019virtual,liu2020fully}, we present the first analysis of a higher order VEM method and achieve optimal order error estimates for the semidiscrete scheme. 
A clear advantage of using higher order methods is the ability to use considerably lower grid resolutions which in turn can avoid the use of strongly adapted grids.   

% The method used in this paper for defining the projection operators is identical to the approach taken in \cite{10.1093/imanum/drab003} but unlike the approach in \cite{antonietti_$c^1$_2016} does not require any special considerations for the nonlinear term. 
% This is due to the fact that our method does not rely on defining projections based on the bilinear forms, we instead use an approach based on solving a constraint least squares problem. 
% Therefore we can construct a semilinear form using the projections directly, avoiding the averaging technique used in \cite{antonietti_$c^1$_2016}.  

The aim of this paper is to present a new nonconforming virtual element method for the discretization of the Cahn-Hilliard equation. 
Our approach for constructing the VEM space follows \cite{10.1093/imanum/drab003}.
We show that by defining the projection operators without using the underlying variational problem, we can directly apply our method to nonlinear fourth-order problems. 
Consequently, our approach does not require any special treatment of the nonlinearity as in \cite{antonietti_$c^1$_2016}. 
Our method is shown to converge with optimal order also in the higher order setting.   
% We present the first fully nonconforming virtual element /method with arbitrary approximation order.
This projection approach has been implemented within the DUNE software framework \cite{dunegridpaperII,dedner2010generic} requiring little change to the existing code base.
To the best of our knowledge, this is the first analysis of a nonconforming virtual element method for the Cahn-Hilliard equation as well as the only higher order method without using a mixed formulation of \eqref{eqn: CH eqn}. 

This paper is organized in the following way. 
In Section~\ref{sec: cts problem} we introduce the weak form of the continuous fourth-order Cahn-Hilliard problem, followed by the virtual element method discretization in Section~\ref{sec: discrete problem}. 
In Section~\ref{sec: error anal} we carry out the error analysis of the continuous-in-time scheme before presenting numerical experiments to verify the theoretical results in Section~\ref{sec: numerical testing}.
Finally, we give proofs of some technical lemmas in Appendix~\ref{sec: appendix}.
\section{Problem formulation}\label{sec: cts problem}
We begin by introducing the variational formulation of problem \eqref{eqn: CH eqn} before introducing notation and some technicalities needed for the rest of the paper. 

\subsection{The continuous problem}
First, we introduce the following space 
\begin{align*}
    V = \{ v \in H^2(\Omega) : \partial_n v = 0 \text{ on } \partial \Omega \}.
\end{align*}
Then, the variational form for \eqref{eqn: CH eqn} is described as follows: \corrections{for a.e. $t \in (0,T)$} find $u(\cdot,t) \in V$ such that 
\begin{align}\label{eqn: cts weak form}
    \begin{split}
        &(\partial_t u,v) + \eps^2 a(u,v) + r(u;u,v) = 0 \quad \forall \,v \in V \\
        &u(\cdot,0) = u_0(\cdot) \in V
    \end{split}
\end{align}
where the bilinear form $a(\cdot,\cdot)$ is the standard hessian form arising in the study of fourth-order problems
\begin{align*}
    a(v,w) = (D^2 v,D^2w) 
    &= \int_{\Omega} (D^2 v) : (D^2 w) \, \mathrm{d} x 
    = \int_{\Omega} \sum_{i,j=1}^2 \frac{\partial^2 v}{\partial x_i \partial x_j} \frac{\partial^2 w}{\partial x_i \partial x_j} \, \mathrm{d} x,
    % \quad \forall \, v,w \in V. 
    \intertext{and the semilinear form $r(\cdot;\cdot,\cdot)$ is defined as}
    r(z;v,w) &= \int_{\Omega} \phi^{\prime} (z) \, D v \cdot D w \, \dx
\end{align*}
for all $z,v,w \in V$.
Existence and uniqueness of a solution $u$ to problem \eqref{eqn: cts weak form} where $u \in L^{\infty} (0,T;V) \cap L^2(0,T;H^4(\Omega))$, and $u_t \in L^2(0,T;L^2(\Omega))$ can be found in e.g. \cite{elliott1989nonconforming}. 
\corrections{Further, higher regularity of $u_t$ can be derived, provided that $-\Delta^2 u_0 + \Delta \phi(u_0) \in V$ (see e.g. \cite{elliott_cahn-hilliard_1986}).}

Note that we can view the Cahn-Hilliard equation as the $H^{-1}$ gradient flow of the following free energy functional
\begin{align}\label{eqn: energy functional}
    E(u) = \int_{\Omega} \Big( 
        \psi(u) + \frac{\eps^2}{2} |\nabla u|^2 
        \Big) \, \dx,
\end{align}
and notice that the total energy $E(u)$ decreases with time, $\frac{d}{dt} E(u(t)) \leq 0$.

\subsection{Basic spaces and notation}
Throughout this paper we use the notation $a \lesssim b$ to mean that $a \leq Cb$ for some constant $C$, which does not depend on $h$.

We denote the space of polynomials of degree less than or equal to $\polOrder$ on a set $\element \subseteq \R^2$ by $\prob_{\polOrder}(\element)$. 
We denote a decomposition of the space $\Omega$ by $\cT_h$ and let $h_K$ denote the diameter of a polygon $\element \in \cT_h$ where $\text{diam}(\element) = \max_{x,y \in \element} \| x-y\|$. 
We will denote the edges of a polygon $\element \in \cT_h$ by $e \subset \partial \element$ and denote the set of all edges in $\cT_h$ by $\mathcal{E}_h = \mathcal{E}_h^{\text{int}} \cup \mathcal{E}_h^{\text{bdry}}$, split into boundary and internal edges respectively.
Similarly, denote the set of vertices in $\cT_h$ by $\mathcal{V}_h = \mathcal{V}_h^{\text{int}} \cup \mathcal{V}_h^{\text{bdry}}$.

For an integer $s>0$, define the \emph{broken Sobolev space} $H^s(\mathcal{T}_h)$ by 
\begin{align*}
    H^s(\mathcal{T}_h) := \{  v \in L^2(\Omega) : v|_K \in H^s(\element), \ \forall \, \element \in \mathcal{T}_h  \},
\end{align*} 
and on this space define the 
% mesh dependent 
inner product $(v_h,w_h)_{s,h} := \sum_{\element \in \mesh} (D^s v_h,D^s w_h)_K$, with the broken $H^s$ seminorm
\begin{align*}
    | v_h |^2_{s,h} := \sum_{\element \in \cT_h} |v_h|^2_{s,\element}.
\end{align*}     

For a function $v \in H^2(\cT_h)$ we define the jump operator $[ \cdot ]$ across an edge $e \in \mathcal{E}_h$ as follows. For an internal edge, $e \in \mathcal{E}_h^{\text{int}}$, define $[ v ] := v^+ - v^- $ where $v^{\pm}$ denotes the trace of $v|_{\element^{\pm}}$ where $e \subset \partial \element^{+} \cap \partial \element^{-}$. For boundary edges, $e \in \mathcal{E}_h^{\text{bdry}}$, let $[v] := v|_e$.

\begin{definition}\label{defn: H2 nonconforming space}
    We define the \emph{$H^2$-nonconforming space} as follows.
    \begin{align*}
        \HTwoNCSpace := \{ v \in H^2(\cT_h) : &\ v \text{ is continuous at the vertices, } 
        % v(v^i) = 0 \ \forall v^i \in \mathcal{V}_h^{\text{bdry}},
        \\
        &\int_e [v]p \, \ds = 0\ \forall \, p \in \prob_{\polOrder-3}(e), \ \forall e \in \mathcal{E}_h^{\text{int}},
        \\ 
        &\int_e [\partial_n v]p \, \ds = 0 \ \forall \, p \in \prob_{\polOrder-2}(e),\ \forall e \in \mathcal{E}_h \}.
    \end{align*}
    % \red{Remove setting of boundary vertcies to zero due to the non dirichlet bcs check \cite{brenner2004poincare}}
    % As shown in e.g. \cite{antonietti_fully_2018, 10.1093/imanum/drab003} it follows that the broken seminorm $|\cdot|_{2,h}$ is a norm on $\HTwoNCSpace$.
\end{definition}     

We also \jsccorrections{assume} the following regularity conditions on the mesh $\cT_h$. 
Note that the following assumption is standard in the VEM framework (see e.g. \cite{beirao_da_veiga_basic_2013}).
\begin{assumption}[Mesh assumptions]\label{assumption: mesh regularity} Assume there exists some $\rho >0$ such that the following holds. 
    % Let $\element \in \mesh$.
    \begin{itemize}
        \item For every element $\element \in \mesh$ and every $e \subset \partial \element$, $h_{e} \geq \rho h_{\element}$ where $h_{e}=|{e}|$.
        \label{assump: mesh 1}
        \item \corrections{Each element $\element \in \mesh$} is star shaped with respect to a ball of radius $\rho h_K$. \label{assump: mesh 2}
        % \item There exists an interior point $x_K$ to $\element$ such that the sub triangle formed by connecting $x_K$  to the vertices of $\element$ is made of shape regular triangles. 
        % 
     \end{itemize}
\end{assumption} 

\begin{remark}\label{assump: star shaped wrt a ball}
    We note that from Assumption~\ref{assumption: mesh regularity} we can show that there exists an interior point $x_K$ to $\element$ such that the sub-triangle formed by connecting $x_K$  to the vertices of $\element$ is made of shape regular triangles. 
    This property, considered also in \cite{antonietti_fully_2018}, is necessary for the proof of Lemma \ref{lemma: boundary term bound}.
\end{remark}

For any $\element \in \mesh$ we denote the orthogonal $L^2(\element)$ projection onto the space $\prob_{\polOrder}(\element)$ by $\ltwo_{\polOrder} : L^2(\element) \rightarrow \prob_{\polOrder}(\element)$.
We also recall the following approximation results for the $L^2$ projection. A proof of the following can be obtained using for example the theory in \cite{brenner_mathematical_2008}.
\begin{theorem}\label{thm: interpolation estimates}
    Under Assumption \ref{assumption: mesh regularity}, for $\polOrder\geq 0$ and for any $w \in H^m (\element)$ with $1 \leq m \leq \polOrder +1$, it follows that
\begin{align*}
    | w - \ltwo_{\polOrder} w |_{s,\element} \lesssim h_K^{m-s} |w|_{m,\element}
\end{align*}
for $s = 0,1,2$ with $s\leq m$.
Further, for any edge $e$ shared by $\element^{+}$,$\element^{-} \in \cT_h$ 
and for any $w \in H^m(\element^{+} \cup \element^{-})$, with 
$1 \leq m \leq \polOrder+1$, it follows that
\begin{align*}
    |w-\cP_{\polOrder}^e w|_{s,e} 
    \lesssim
    h_e^{m-s-\frac{1}{2}} | w |_{m,\element^{+}\cup \element^{-}}
\end{align*}
for $s=0,1,2$ with $s\leq m$.
\end{theorem}

Finally, we use the following notation for the local bilinear form $a^\element(\cdot,\cdot)$, that is, for $\element \in \mesh$, define $a^\element$ as 
$$a^\element(v,w) = \int_K D^2 v : D^2w \, \dx \quad \forall \, v,w \in V.$$
\section{Virtual element discretization}\label{sec: discrete problem}
In this section we construct the VEM discretization. This involves building the discrete virtual element space, $\vemSpace$, the projection operators, and the discrete forms. 
The first step in constructing the VEM space $\vemSpace$ is to introduce an enlarged space  $\localenlargedVemSpace$ and an extended set of dofs for this space. 
We then introduce a suitable set of dofs for $\localVEMSpace$ and use these dofs to construct the \emph{dof compatible} projection operators $\valueProj, \gradProj, \hessProj$ which will be used to define the local VEM space as well as the local discrete forms. 
The projections have the advantage of being fully computable using the dofs. 
We aim to keep this section brief as the construction can be found in \cite{10.1093/imanum/drab003}. 
Throughout this section and the remainder of the paper we assume that $\polOrder \geq 2$.

\subsection{Local VEM space}
Following the standard VEM enhancement technique \cite{ahmad_equivalent_2013,cangiani_conforming_2015} we first give the definition of the local enlarged VEM space $\localenlargedVemSpace$ for $K \in \mesh$.

\begin{definition}[Local enlarged space]\label{defn: local enlarged space}
    Given an element $K \in \mesh$ the \emph{local enlarged space} $\localenlargedVemSpace$ is defined as 
    \begin{align*}
        % \label{eqn: local enlarged space}
        \localenlargedVemSpace = \left\{ v_h \in H^2(K) : \Delta^2 v_h \in \prob_{\polOrder}(K), \, 
        v_h|_{e} \in \prob_{\polOrder}(e),
        \,
        \Delta v_h |_{e} \in \prob_{\polOrder-2}(e), \, 
        \forall e \subset \partial K \right\}.
    \end{align*}   
\end{definition}
    We adopt the \emph{dof tuple} notation from \cite{10.1093/imanum/drab003}, and note that $\localenlargedVemSpace$ is characterized by the extended set of dofs $\widetilde{\Lambda}^{\element}$ described by the dof tuple $(0,-1,\polOrder-2,\polOrder-2,\polOrder)$, i.e.,
    we fix function values of $v_h$ at the vertices of $K$ and we fix edge moments up to $\polOrder-2$. 
    We then define the edge normal moments up to order $\polOrder-2$ and define inner moments up to order $\polOrder$. 
    This set of dofs is unisolvent over the space $\localenlargedVemSpace$ \cite{10.1093/imanum/drab003}.
The local VEM space $\localVEMSpace \subset \localenlargedVemSpace$ is characterized by the set of dofs $\Lambda^{\element}$ described by the dof tuple $(0,-1,\polOrder-3,\polOrder-2,\polOrder-4)$.
Note that this set of dofs is identical to the set used in \cite{antonietti_fully_2018} for the biharmonic problem and they are visualized on triangles in Figure \ref{figure: C^1 non-conforming dofs for l=2,3,4,5 on triangles}.

\ifthenelse{\boolean{arxiv}}
{
    \begin{figure}[!ht]
        \centering
        % \resizebox{\textwidth}{!}{%
        \begin{tikzpicture}
        %triangle 1:
            \draw (0,0) -- (2,0);
            \draw (0,0) -- (1,1.73);
            \draw (1,1.73) -- (2,0);
            \fill[blue!40!white] (0,0) circle (0.075cm);
            % \fill[blue!40!white] (0,0) rectangle (0.1,0.1);
            \fill[blue!40!white] (1,1.73) circle (0.075cm);
            \fill[blue!40!white] (2,0) circle (0.075cm);
            \draw[arrows=->,line width=.4pt] (1.5,0.87) -- (1.93,1.12);
            \draw[arrows=->,line width=.4pt] (1,0) -- (1,-0.5);
            \draw[arrows=->,line width=.4pt] (0.5,0.87) -- (0.067,1.12);
            \node at (1, -1) {$\polOrder = 2$};
        
        %triangle 2:
            \draw (4,0) -- (6,0);
            \draw (4,0) -- (5,1.73);
            \draw (5,1.73) -- (6,0);
            \fill[blue!40!white] (6,0) circle (0.075cm);
            \fill[blue!40!white] (5,1.73) circle (0.075cm);
            \fill[blue!40!white] (4,0) circle (0.075cm);
            \draw[arrows=->,line width=.4pt] (4.8,0) -- (4.8,-0.5);
            \draw[arrows=->,line width=.4pt] (5.2,0) -- (5.2,-0.5);
            \draw[arrows=->,line width=.4pt] (5.367,1.046) -- (5.81,1.302);
            \draw[arrows=->,line width=.4pt] (5.61,0.674) -- (6.06,0.93);
            \draw[arrows=->,line width=.4pt] (4.595,1.03) -- (4.152,1.28);
            \draw[arrows=->,line width=.4pt] (4.39,0.68) -- (3.941,0.94);
            \fill[red!40!white] (4.5,0.87) circle (0.075cm);
            \fill[red!40!white] (5.5,0.87) circle (0.075cm);
            \fill[red!40!white] (5,0) circle (0.075cm);
            \node at (5, -1) {$\polOrder = 3$};
    
        % triangle 3:
            \draw (8,0) -- (10,0);
            \draw (8,0) -- (9,1.73);
            \draw (9,1.73) -- (10,0);
            \fill[blue!40!white] (10,0) circle (0.075cm);
            \fill[blue!40!white] (9,1.73) circle (0.075cm);
            \fill[blue!40!white] (8,0) circle (0.075cm);
            \fill[green] (8.925,0.505) rectangle (9.075,0.655);
    
            \draw [arrows=->,line width=.4pt] (8.5,0) -- (8.5,-0.5);
            \draw [arrows=->,line width=.4pt] (9,0) -- (9,-0.5);
            \draw [arrows=->,line width=.4pt] (9.5,0) -- (9.5,-0.5);
    
            \fill[red!40!white]  (8.75,0) circle (0.075cm);
            \fill[red!40!white] (9.25,0) circle (0.075cm);
            \fill[red!40!white] (8.625,1.082531754730546) circle (0.075cm);
            \fill[red!40!white] (8.375,0.6495190528383272) circle (0.075cm);
            \fill[red!40!white] (9.375,1.0825317547305477) circle (0.075cm);
            \fill[red!40!white] (9.625,0.649519052838329) circle (0.075cm);
    
            \draw [arrows=->,line width=.4pt] (8.25,0.43301270189221786) -- (7.81585938579443,0.683663902369954);
            \draw [arrows=->,line width=.4pt] (8.5,0.8660254037844366) -- (8.065859385794429,1.1166766042621727);
            \draw [arrows=->,line width=.4pt] (8.75,1.2990381056766562) -- (8.315859385794429,1.5496893061543924);
    
            \draw [arrows=->,line width=.4pt] (9.75,0.43301270189221874) -- (10.184140614205571,0.683663902369954);
            \draw [arrows=->,line width=.4pt] (9.5,0.8660254037844393) -- (9.934140614205571,1.1166766042621745);
            \draw [arrows=->,line width=.4pt] (9.25,1.299038105676658) -- (9.684140614205571,1.5496893061543933);
            \node at (9, -1) {$\polOrder = 4$};
    
        % triangle 4:
            \draw (12,0) -- (14,0);
            \draw (12,0) -- (13,1.73);
            \draw (13,1.73) -- (14,0);
            \fill[blue!40!white] (14,0) circle (0.075cm);
            \fill[blue!40!white] (13,1.73) circle (0.075cm);
            \fill[blue!40!white] (12,0) circle (0.075cm);
    
            \fill[green] (12.925,0.605) rectangle (13.075,0.755);
            \fill[green] (12.825,0.405) rectangle (12.975,0.555);
            \fill[green] (13.025,0.405) rectangle (13.175,0.555);
    
            \fill[red!40!white] (13.497533850441435,0.8702969001189368) circle (0.075cm);
            \fill[red!40!white] (12.295067700882871,0.5281582349396556) circle (0.075cm);
            \fill[red!40!white] (12.495067700882872,0.874568396453431) circle (0.075cm);
            \fill[red!40!white] (12.695067700882873,1.2209785579672063) circle (0.075cm);
    
            \fill[red!40!white] (13.3,1.2124355652982128) circle (0.075cm);
            \fill[red!40!white] (13.5,0.8660254037844375) circle (0.075cm);
            \fill[red!40!white] (13.7,0.5196152422706621) circle (0.075cm);
    
            \fill[red!40!white] (12.6,0) circle (0.075cm);
            \fill[red!40!white] (13,0) circle (0.075cm);
            \fill[red!40!white] (13.4,0) circle (0.075cm);
    
            \draw [arrows=->,line width=.4pt] (13.2,0) -- (13.2,-0.5);
            \draw [arrows=->,line width=.4pt] (13.6,0) -- (13.6,-0.5);
            \draw [arrows=->,line width=.4pt] (12.4,0) -- (12.4,-0.5);
            \draw [arrows=->,line width=.4pt] (12.8,0) -- (12.8,-0.5);
    
            \draw [arrows=->,line width=.4pt] (13.2,1.3856406460551014) -- (13.579838221736962,1.5988388306005594);
            \draw [arrows=->,line width=.4pt] (13.4,1.0392304845413243) -- (13.778876472189724,1.2540944681667543);
            \draw [arrows=->,line width=.4pt] (13.6,0.6928203230275507) -- (13.982025138915313,0.9022306559081561);
            \draw [arrows=->,line width=.4pt] (13.8,0.34641016151377535) -- (14.186094576856972,0.5487720211211804);
    
            \draw [arrows=->,line width=.4pt] (12.795067700882871,1.394183638724095) -- (12.429550603832414,1.6054755208373084);
            \draw [arrows=->,line width=.4pt] (12.595067700882872,1.0477734772103178) -- (12.243583761909406,1.2833715021035295);
            \draw [arrows=->,line width=.4pt] (12.395067700882873,0.7013633156965442) -- (12.03859355275168,0.9283180447881755);
            \draw [arrows=->,line width=.4pt] (12.195067700882873,0.35495315418276885) -- (11.842317615947124,0.5883581499395047);
            \node at (13, -1) {$\polOrder = 5$};
        \end{tikzpicture}
        % }
        \caption{Degrees of freedom for polynomial orders $\polOrder=2,3,4,5$ on triangles for the local VEM space. Circles at vertices represent vertex dofs, arrows represent edge normal dofs, circles on edges represent edge value moments and interior squares represent inner dofs.}
        \label{figure: C^1 non-conforming dofs for l=2,3,4,5 on triangles}
    \end{figure}
}
{
    \begin{figure}[!ht]
        \centering
        \resizebox{\textwidth}{!}{%
        \begin{tikzpicture}
        %triangle 1:
            \draw (0,0) -- (2,0);
            \draw (0,0) -- (1,1.73);
            \draw (1,1.73) -- (2,0);
            \fill[blue!40!white] (0,0) circle (0.075cm);
            % \fill[blue!40!white] (0,0) rectangle (0.1,0.1);
            \fill[blue!40!white] (1,1.73) circle (0.075cm);
            \fill[blue!40!white] (2,0) circle (0.075cm);
            \draw[arrows=->,line width=.4pt] (1.5,0.87) -- (1.93,1.12);
            \draw[arrows=->,line width=.4pt] (1,0) -- (1,-0.5);
            \draw[arrows=->,line width=.4pt] (0.5,0.87) -- (0.067,1.12);
            \node at (1, -1) {$\polOrder = 2$};
        
        %triangle 2:
            \draw (4,0) -- (6,0);
            \draw (4,0) -- (5,1.73);
            \draw (5,1.73) -- (6,0);
            \fill[blue!40!white] (6,0) circle (0.075cm);
            \fill[blue!40!white] (5,1.73) circle (0.075cm);
            \fill[blue!40!white] (4,0) circle (0.075cm);
            \draw[arrows=->,line width=.4pt] (4.8,0) -- (4.8,-0.5);
            \draw[arrows=->,line width=.4pt] (5.2,0) -- (5.2,-0.5);
            \draw[arrows=->,line width=.4pt] (5.367,1.046) -- (5.81,1.302);
            \draw[arrows=->,line width=.4pt] (5.61,0.674) -- (6.06,0.93);
            \draw[arrows=->,line width=.4pt] (4.595,1.03) -- (4.152,1.28);
            \draw[arrows=->,line width=.4pt] (4.39,0.68) -- (3.941,0.94);
            \fill[red!40!white] (4.5,0.87) circle (0.075cm);
            \fill[red!40!white] (5.5,0.87) circle (0.075cm);
            \fill[red!40!white] (5,0) circle (0.075cm);
            \node at (5, -1) {$\polOrder = 3$};
    
        % triangle 3:
            \draw (8,0) -- (10,0);
            \draw (8,0) -- (9,1.73);
            \draw (9,1.73) -- (10,0);
            \fill[blue!40!white] (10,0) circle (0.075cm);
            \fill[blue!40!white] (9,1.73) circle (0.075cm);
            \fill[blue!40!white] (8,0) circle (0.075cm);
            \fill[green] (8.925,0.505) rectangle (9.075,0.655);
    
            \draw [arrows=->,line width=.4pt] (8.5,0) -- (8.5,-0.5);
            \draw [arrows=->,line width=.4pt] (9,0) -- (9,-0.5);
            \draw [arrows=->,line width=.4pt] (9.5,0) -- (9.5,-0.5);
    
            \fill[red!40!white]  (8.75,0) circle (0.075cm);
            \fill[red!40!white] (9.25,0) circle (0.075cm);
            \fill[red!40!white] (8.625,1.082531754730546) circle (0.075cm);
            \fill[red!40!white] (8.375,0.6495190528383272) circle (0.075cm);
            \fill[red!40!white] (9.375,1.0825317547305477) circle (0.075cm);
            \fill[red!40!white] (9.625,0.649519052838329) circle (0.075cm);
    
            \draw [arrows=->,line width=.4pt] (8.25,0.43301270189221786) -- (7.81585938579443,0.683663902369954);
            \draw [arrows=->,line width=.4pt] (8.5,0.8660254037844366) -- (8.065859385794429,1.1166766042621727);
            \draw [arrows=->,line width=.4pt] (8.75,1.2990381056766562) -- (8.315859385794429,1.5496893061543924);
    
            \draw [arrows=->,line width=.4pt] (9.75,0.43301270189221874) -- (10.184140614205571,0.683663902369954);
            \draw [arrows=->,line width=.4pt] (9.5,0.8660254037844393) -- (9.934140614205571,1.1166766042621745);
            \draw [arrows=->,line width=.4pt] (9.25,1.299038105676658) -- (9.684140614205571,1.5496893061543933);
            \node at (9, -1) {$\polOrder = 4$};
    
        % triangle 4:
            \draw (12,0) -- (14,0);
            \draw (12,0) -- (13,1.73);
            \draw (13,1.73) -- (14,0);
            \fill[blue!40!white] (14,0) circle (0.075cm);
            \fill[blue!40!white] (13,1.73) circle (0.075cm);
            \fill[blue!40!white] (12,0) circle (0.075cm);
    
            \fill[green] (12.925,0.605) rectangle (13.075,0.755);
            \fill[green] (12.825,0.405) rectangle (12.975,0.555);
            \fill[green] (13.025,0.405) rectangle (13.175,0.555);
    
            \fill[red!40!white] (13.497533850441435,0.8702969001189368) circle (0.075cm);
            \fill[red!40!white] (12.295067700882871,0.5281582349396556) circle (0.075cm);
            \fill[red!40!white] (12.495067700882872,0.874568396453431) circle (0.075cm);
            \fill[red!40!white] (12.695067700882873,1.2209785579672063) circle (0.075cm);
    
            \fill[red!40!white] (13.3,1.2124355652982128) circle (0.075cm);
            \fill[red!40!white] (13.5,0.8660254037844375) circle (0.075cm);
            \fill[red!40!white] (13.7,0.5196152422706621) circle (0.075cm);
    
            \fill[red!40!white] (12.6,0) circle (0.075cm);
            \fill[red!40!white] (13,0) circle (0.075cm);
            \fill[red!40!white] (13.4,0) circle (0.075cm);
    
            \draw [arrows=->,line width=.4pt] (13.2,0) -- (13.2,-0.5);
            \draw [arrows=->,line width=.4pt] (13.6,0) -- (13.6,-0.5);
            \draw [arrows=->,line width=.4pt] (12.4,0) -- (12.4,-0.5);
            \draw [arrows=->,line width=.4pt] (12.8,0) -- (12.8,-0.5);
    
            \draw [arrows=->,line width=.4pt] (13.2,1.3856406460551014) -- (13.579838221736962,1.5988388306005594);
            \draw [arrows=->,line width=.4pt] (13.4,1.0392304845413243) -- (13.778876472189724,1.2540944681667543);
            \draw [arrows=->,line width=.4pt] (13.6,0.6928203230275507) -- (13.982025138915313,0.9022306559081561);
            \draw [arrows=->,line width=.4pt] (13.8,0.34641016151377535) -- (14.186094576856972,0.5487720211211804);
    
            \draw [arrows=->,line width=.4pt] (12.795067700882871,1.394183638724095) -- (12.429550603832414,1.6054755208373084);
            \draw [arrows=->,line width=.4pt] (12.595067700882872,1.0477734772103178) -- (12.243583761909406,1.2833715021035295);
            \draw [arrows=->,line width=.4pt] (12.395067700882873,0.7013633156965442) -- (12.03859355275168,0.9283180447881755);
            \draw [arrows=->,line width=.4pt] (12.195067700882873,0.35495315418276885) -- (11.842317615947124,0.5883581499395047);
            \node at (13, -1) {$\polOrder = 5$};
        \end{tikzpicture}
        }
        \caption{Degrees of freedom for polynomial orders $\polOrder=2,3,4,5$ on triangles for the local VEM space. Circles at vertices represent vertex dofs, arrows represent edge normal dofs, circles on edges represent edge value moments and interior squares represent inner dofs.}
        \label{figure: C^1 non-conforming dofs for l=2,3,4,5 on triangles}
    \end{figure}
}

In order to define the local VEM space we now introduce the following projection operators: an element value projection $\valueProj : \localenlargedVemSpace \rightarrow \prob_{\polOrder}(K)$, an edge value projection $\edgeProj : \localenlargedVemSpace \rightarrow \prob_{\polOrder}(e)$, and an edge normal projection $\edgeNormalProj : \localenlargedVemSpace \rightarrow \prob_{\polOrder-1}(e)$.
These projections must be computable from the reduced set of dofs $\Lambda^{\element}$ and must also satisfy the following.

\begin{assumption}\label{ass: dof compatible projs}
    For $v_h \in \enlargedVemSpace$, assume the value projection $\valueProj$, edge projection $\edgeProj$, and edge normal projection $\edgeNormalProj$ are a linear combination of the dofs $\Lambda^{\element}(v_h)$ and satisfy the additional properties.  
    \begin{itemize}
        \item The value projection $\valueProj v_h \in \prob_{\polOrder}(\element)$ satisfies
        \begin{align*}
            \int_K \valueProj v_h p \, \dx = \int_K v_h p \, \dx \quad \forall p \in \prob_{\polOrder -4}(K)
        \end{align*}
        and $\valueProj q = q$ for all $q \in \prob_{\polOrder}(\element)$.
        \item For each edge, the edge projection $\edgeProj v_h \in \prob_{\polOrder}(e)$ satisfies $ \edgeProj v_h(e^{\pm}) = v_h(e^{\pm})$,
        \begin{align*}
           \int_e \edgeProj v_h p \, \ds = \int_e v_h p \, \ds \quad \forall p \in \prob_{\polOrder-3}(e), 
        %    \quad \int_e \edgeProj v_h p \, \ds = \int_e \valueProj v_h p \, \ds \quad \forall p \in \prob_{\polOrder-2}(e)\backslash \prob_{\polOrder-3}(e),
        \end{align*}
        and $\edgeProj q = q|_e$ for all $q \in \prob_{\polOrder}(\element)$.
        \item For each edge, the edge normal projection $\edgeNormalProj v_h \in \prob_{\polOrder-1}(e)$ satisfies
        \begin{align*}
            \int_e \edgeNormalProj v_h p \, \ds = \int_e \partial_n v_h p \, \ds \quad \forall p \in \prob_{\polOrder-2}(e),
        \end{align*}
        and $\edgeNormalProj q = \partial_n q|_e$ for all $q \in \prob_{\polOrder}(\element)$.
    \end{itemize}
\end{assumption}
Note that there are multiple ways of defining the value, edge, and edge normal projections so that they satisfy Assumption~\ref{ass: dof compatible projs}.
An example choice based on constrained least squares problems can be found in \cite{10.1093/imanum/drab003,dedner2022framework}, where we refer the reader for more details.
Finally, we are now able to define both the gradient $\gradProj$ and hessian $\hessProj$ projections.
\begin{definition}
    The \emph{gradient projection}, $\gradProj : \enlargedVemSpace^K \rightarrow [\prob_{\polOrder-1}(K)]^2$, is defined as
        \begin{align*}\label{eqn: grad proj}
            \int_K \gradProj v_h p \, \dx = - \int_K \valueProj v_h \nabla p \, dx + \sum_{e \subset \partial K } \int_e \edgeProj v_h p n \, \ds, 
            \quad \forall p \in [\prob_{\polOrder-1 }(K)]^2.
        \end{align*}
\end{definition}
        
\begin{definition}
    The \emph{hessian projection}, $\hessProj : \enlargedVemSpace^K \rightarrow [\prob_{\polOrder-2}(K)]^{2 \times 2}$, is defined as
    \begin{align*}
        \int_K \hessProj v_h p \, \dx = - \int_K \gradProj v_h \otimes \nabla p \, \dx + \sum_{e \subset \partial K } \int_{e} \big(  \edgeNormalProj v_h n \otimes n p + \partial_s ( \edgeProj v_h ) \tau \otimes n p \big) \, \ds,
    \end{align*}
    for all $ p \in [\prob_{\polOrder-2}(K)]^{2\times 2}$.
    Here $n,\tau$ denote the unit normal and tangent vectors of $e$, respectively.
\end{definition}

% By defining the projections in this way, it follows that they are indeed computable from the dofs \cite{10.1093/imanum/drab003}. 

\begin{definition}
    We use the notation $\Pi^h_0, \Pi^h_1$ and $\Pi^h_2$ to denote the global value, gradient, and hessian projections, respectively. 
    Therefore for $s=0,1,2$,
    \begin{align*}
        (\Pi^h_s v_h) |_{\element} := \Pi^K_s v_h \quad \forall v_h \in \vemSpace.
    \end{align*}      
\end{definition}

\begin{definition}[Local virtual space]
    The \emph{local virtual element space} $\localVEMSpace$ is defined as 
    \begin{equation}\label{eqn: local space}
        \begin{split}
            \localVEMSpace := \Big\{ v_h \in \enlargedVemSpace^K : \ (v_h - \valueProj v_h, p)_K &= 0 \quad \forall p \in \prob_{\polOrder}(K) \backslash \prob_{\polOrder -4}(K), 
            \\
            (v_h - \edgeProj v_h , p)_e &= 0 \quad \forall p \in \prob_{\polOrder-2}(e) \backslash \prob_{\polOrder-3}(e) \Big\}.
        \end{split}
    \end{equation}
\end{definition}
The set of local degrees of freedom $\Lambda^K$ is unisolvent over $\localVEMSpace$, a proof of which can be found in \cite{10.1093/imanum/drab003}. 
Also shown in \cite{10.1093/imanum/drab003} is the proof of the subsequent lemma, detailing that all the projections satisfy a crucial $L^2$ projection property.
This property follows as a consequence of the construction of the value, gradient, and hessian projections.  
\begin{lemma}\label{eqn: projections and l2 projection result}
    Assume that the value, edge, and edge normal projections satisfy Assumption~\ref{ass: dof compatible projs}.
    Then, the value, gradient, and hessian projections satisfy
    \begin{alignat}{3}
        \Pi_s^{K} v_h &= \cP_{\polOrder-s}^K(D^s v_h) &&\quad \forall v_h \in \localVEMSpace,&& \quad s=0,1,2.
    \intertext{It also holds that $\prob_{\polOrder}(K) \subset \localVEMSpace$ and therefore}
        \Pi_s^K p &= D^s p &&\quad \forall p \in \prob_{\polOrder}(K),&& \quad s=0,1,2.
    \end{alignat}
\end{lemma}

\subsection{Global spaces and the discrete forms}
The global VEM space can now be defined in the standard way as follows. 
\begin{definition}[Global virtual space]\label{defn: global VEM space}
    The \emph{global VEM space} is defined as
    \begin{align}\label{eqn: global VEM space}
        \vemSpace := \left\{ v_h \in \HTwoNCSpace : v_h|_K \in \localVEMSpace \quad \forall K \in \cT_h \right\}
    \end{align}
    where $\HTwoNCSpace$ is the nonconforming space given in Definition \ref{defn: H2 nonconforming space}. 
\end{definition}

We can define the corresponding global dofs in the usual way \cite{beirao_da_veiga_basic_2013}.
We set the local normal dofs which correspond to boundary edges to zero.
Note that the global degrees of freedom are unisolvent - this follows from the unisolvency of the local
degrees of freedom and the definition of the local spaces.

Now that the virtual spaces \corrections{and} projection operators \corrections{are in place,} we are able to define the discrete bilinear forms\corrections{.}

\begin{definition}[Discrete forms]\label{defn: discrete forms}
    For $v_h,w_h,z_h \in \vemSpace$ define the discrete forms as
    \begin{align*}
        \ahhessianLocal (v_h,w_h) &= \int_K \hessProj v_h : \hessProj w_h \, \dx + h_K^{-2} \stabilisation(v_h-\valueProj v_h,w_h-\valueProj w_h),
        \\
        \ahlowerLocal (v_h,w_h) &= \int_K \valueProj v_h \valueProj w_h \, \dx + h_K^2 \stabilisation(v_h-\valueProj v_h,w_h-\valueProj w_h),
        \\
        \rh^K ( z_h;  v_h, w_h) &= \int_K  \phi^{\prime} ( z_h ) \, \gradProj v_h \cdot \gradProj w_h \, \dx
        + \beta_K \stabilisation (v_h - \valueProj v_h,w_h-\valueProj w_h ),
    \end{align*}
    where $\stabilisation(\cdot,\cdot)$ is, for example, the standard ``dofi-dofi'' stabilization, see e.g. \cite{beirao_da_veiga_basic_2013}: $\stabilisation(v_h,w_h) := \sum_{\lambda_i \in \Lambda^{\element}} \lambda_i(v_h)\lambda_i(w_h)$ and $\beta_K$ is constant. 
    We note that an exploration of the stabilization and its role within the virtual element method can be found in \cite{mascotto2023role}.

    \jsccorrections{
    The global forms can be defined in the usual way, 
    \begin{align*}
        \ahhessian (v_h,w_h) &:= \sum_{K \in \mesh} \ahhessianLocal (v_h,w_h),
        \\
        \ahlower (v_h,w_h) &:= \sum_{K \in \mesh} \ahlowerLocal (v_h,w_h),
        \\
        \rh (z_h;v_h,w_h) &:= \sum_{K \in \mesh} \rhLocal (z_h;v_h,w_h),
    \end{align*}
    for $z_h,v_h,w_h \in \vemSpace$.}
\end{definition}

\subsection{The semidiscrete problem}\label{sec: semidiscrete scheme}
The semidiscrete problem is defined as follows: find $u_h(\cdot,t) \in \vemSpace$ such that 
\begin{equation}\label{eqn: semidiscrete scheme}
    \begin{split}
        &m_h(\partial_t u_h, v_h) + \eps^2 a_h ( u_h, v_h ) + \rh(\Pi^h_0 u_h;u_h,v_h) = 0 \quad \forall \, v_h \in \vemSpace, \text{ a.e. $t$ in } (0,T),
        \\
        &u_h(\cdot,0) = u_{h,0}(\cdot) \in \vemSpace,
    \end{split}
\end{equation}
where $u_{h,0}$ is some approximation of $u_0$ and the discrete forms are given in Definition \ref{defn: discrete forms}.

\begin{remark}
        Note that in order to define the semidiscrete scheme \jsccorrections{in \eqref{eqn: semidiscrete scheme}}, the value projection \jsccorrections{onto higher order polynomials} is used in the first argument of $r_h^{\element}(\cdot;\cdot,\cdot)$.
        However, we need the form as it is \jsccorrections{introduced in Definition~\ref{defn: discrete forms}} to be able to apply it to $u$ which is necessary for the analysis.
\end{remark}

\jsccorrections{Lastly, we detail some important properties of the discrete forms, necessary for the error analysis of problem \eqref{eqn: semidiscrete scheme}.}
Due to Lemma \ref{eqn: projections and l2 projection result} it is immediate that the discrete forms $\ahhessianLocal(\cdot,\cdot)$ and $\ahlowerLocal(\cdot,\cdot)$ possess the standard consistency property, \corrections{implying} that whenever one of the entries in the bilinear form is a polynomial of degree $\polOrder$, the form is exact.
\begin{lemma}[Polynomial consistency]\label{lemma: consistency}
    For any $w_h \in \vemSpace$, it holds that 
    \begin{align*}
        \ahhessianLocal ( p , w_h ) = a^K ( p , w_h), \quad 
        % \text{and}
        \quad
        \ahlowerLocal ( p , w_h ) = (p,w_h)_K
    \end{align*}
    for all $p \in \prob_{\polOrder}(K)$.
\end{lemma}

We also have the standard stability property for the\corrections{se} forms. The proof is standard and is based on a scaling argument: in the conforming case it can be found for example in \cite{beirao2017stability,brenner2018virtual} and the nonconforming case is shown in \cite{mascotto2018non} but for a different choice of stabilization $\stabilisation(\cdot,\cdot)$.

\begin{lemma}[Stability]\label{lemma: stability}
    There exist positive constants $\alpha_*, \alpha^*, \mu_*,$ and $\mu^*$ such that for all $v_h \in \localVEMSpace$
    \begin{align*}
        \alpha_* a^K (v_h,v_h) \leq \ahhessianLocal (v_h,v_h) &\leq \alpha^* a^K (v_h,v_h),
        \\
        \mu_* (v_h,v_h)_K \leq \ahlowerLocal (v_h,v_h) &\leq \mu^* (v_h,v_h)_K.
    \end{align*}
\end{lemma}

Finally, we note that there exists an interpolation operator $\globalInterpolation$, defined in the usual way, which satisfies interpolation estimates \cite{beirao_da_veiga_basic_2013,mora2015virtual} i.e. under Assumption~\ref{assumption: mesh regularity}, for $s=0,1,2$ and any $w \in H^m({\element})$ with $s\leq m\leq \polOrder+1$, the following estimate holds  
\begin{align}\label{eqn: approximation properties of interpolation operator}
    | w - \globalInterpolation w |_{s,K} \lesssim h^{m-s} |w |_{m,K}.
\end{align}
\section{Error analysis of the semidiscrete scheme}\label{sec: error anal}
In this section we detail the error analysis for the semidiscrete (continuous-in-time) scheme. 
\corrections{We focus on the spatial discretization,} proving $L^2$ convergence of the scheme \eqref{eqn: semidiscrete scheme} in Theorem \ref{thm: L2 convergence}. 
\corrections{S}tandard arguments can be employed in the fully discrete case (using the methods for example in \cite{thomee2007galerkin})\corrections{.}

Firstly, we make the following assumption on the discrete solution $u_h$ of \eqref{eqn: semidiscrete scheme}, \corrections{which is standard and well accepted in the analysis of this problem \cite{antonietti_$c^1$_2016}}. 
For more details on this assumption and for a full justification see \cite{elliott1989nonconforming}.
\begin{assumption}\label{assumption: boundedness of discrete solution}
    The solution $u_h$ to \eqref{eqn: semidiscrete scheme} satisfies for all $t \in (0,T]$ 
    \begin{align*}
        \| u_h(\cdot,t) \|_{1,\infty;h} \leq C_T
    \end{align*}
    for a constant $C_T$ independent of $h$, which depends on $T$, where 
    \begin{align*}
        \| v_h \|_{m,\infty;h} := \max_{\substack{1 \leq j \leq m \\ K \in \cT_h}} | v_h |_{j,\infty,K}. 
    \end{align*}
\end{assumption}

\corrections{Note that in the following we do not include the dependence of $u$ and $u_h$ on time $t$ and the bounds involving $u$ or $u_h$ hold for all $t \in (0,T]$.
The proof techniques in this chapter follow along the lines of \cite{antonietti_$c^1$_2016}.}

\subsection{The elliptic projection}
In this subsection we introduce the elliptic projection, which is fundamental for the proof of the main $L^2$ convergence theorem detailed in Theorem \ref{thm: L2 convergence}.
To this end, we define the elliptic projection $\ellipticProj v \in \vemSpace$ for $v \in H^{4}(\Omega)$ as the solution of 
\begin{align}\label{eqn: elliptic proj}
    \bh ( \ellipticProj v, \psi_h) = (\eps^2 \Delta^2 v - \nabla \cdot (\phi^{\prime} (u) \nabla v) + \alpha v, \psi_h ) \quad \forall \, \psi_h \in \vemSpace 
\end{align}
where the bilinear form $\bh (\cdot,\cdot)$ is defined as 
\begin{align}\label{eqn: bilinear form b_h}
    \bh (v_h , w_h) := \eps^2 a_h (v_h,w_h) + r_h(u;v_h,w_h) + \alpha(v_h,w_h),
\end{align}
for a positive $\alpha$, chosen so that the bilinear form $b_h$ is coercive.

The main results of this section are the approximation properties of $\ellipticProj$ detailed in Lemma~\ref{lemma: elliptic proj bounds}.
Before we state these properties, we require the following lemma, the proof of which can be found in \cite{10.1093/imanum/drab003}.
\begin{lemma}\label{lemma: nonconformity error}
    \corrections{For the solution} $u \in H^{\polOrder+1}(\Omega)$ \corrections{to \eqref{eqn: cts weak form} and $w \in \HTwoNCSpace$} the nonconformity error satisfies  
    \begin{align*}
        \left| \nonconformity (u,w) \right|
        =
        \left|
        \eps^2 \sum_{\element \in \mesh}
        \int_{\partial \element} \left( ( \Delta u - \partial_{ss} u) \partial_n w
        +
        \partial_{ns} u \partial_s w 
        - \partial_n (\Delta u) w
        \right)
        \, \ds 
        \right|
        \lesssim h^{\polOrder-1} |w|_{2,h}\corrections{,}
    \end{align*}
    \corrections{where the nonconformity error is defined as follows}
    \begin{align}\label{eqn: nonconformity defn}
        \corrections{\nonconformity (u,w) := \eps^2 \sum_{\element \in \mesh} a^{\element}(u,w) - (\Delta^2 u,w).}
    \end{align}
\end{lemma} 

\ifthenelse{\boolean{thesis}}
{
    \begin{align*}
        r_h ( u; v_h, v_h) 
        = 
        \sum_{K \in \mesh} \int_K  \phi^{\prime} (u) \, \gradProj v_h \cdot \gradProj v_h \, \dx  
        &\leq  
        \sum_{K \in \mesh}  \| \phi^{\prime}(u) \ltwo_{\polOrder-1} \nabla v_h \|_{0,K} \| \ltwo_{\polOrder-1} \nabla v_h \|_{0,K}
        \\
        &\leq 
        \sum_{K \in \mesh} 
        \| \phi^{\prime}(u) \|_{L^{\infty}} | v_h |_{1,K}^2
        \\
        &= 
        \| \phi^{\prime}(u) \|_{L^{\infty}} |v_h |_{1,h}^2
    \end{align*}
}
{}

\begin{lemma}\label{lemma: elliptic proj bounds}
    Let \corrections{$u \in H^{4}(\Omega) \cap H^{\polOrder+1}(\Omega)$} be the solution to \eqref{eqn: cts weak form} and let $\ellipticProj u$ be the elliptic projection defined in \eqref{eqn: elliptic proj}. Then, it holds that
    \begin{align}
        \|u -\ellipticProj u \|_{2,h} &\lesssim h^{\polOrder-1}, 
        \label{eqn: elliptic proj in 2norm}
        \\
        \| u - \ellipticProj u \|_{1,h} &\lesssim h^{\polOrder}, 
        \label{eqn: elliptic proj in 1norm}
        \\
        \| u_t - (\ellipticProj u)_t \|_{2,h} &\lesssim h^{\polOrder-1}, 
        \label{eqn: elliptic proj time in 2 norm} \\
        \| u_t - (\ellipticProj u)_t \|_{1,h} &\lesssim h^{\polOrder}.
        \label{eqn: elliptic proj time in 1 norm}
    \end{align}
\end{lemma}
Proof of the estimate \eqref{eqn: elliptic proj in 2norm} in Lemma \ref{lemma: elliptic proj bounds} is a direct consequence of the energy norm convergence proof covered in \cite[Theorem 5.7]{10.1093/imanum/drab003} for a general fourth-order problem with varying coefficients.

In order to prove \eqref{eqn: elliptic proj in 1norm} we first study the following problem: find $z \in V$ such that 
\begin{align}\label{eqn: dual problem}
    b(z,w) = (u-\ellipticProj u, w)_{1,h} + (u-\ellipticProj u, w)_{0,h} =: L_h(w)
    \quad \forall w \in V 
\end{align}
where the bilinear form $b(\cdot,\cdot)$ is defined as
\begin{align*}
    b(v,w) &:= \eps^2 a(v,w) + r(u;v,w) + \alpha (v,w)
    % \intertext{and the functional $l_h(\cdot)$ as}
    % l_h(w) &:= \sum_{K \in \mesh} \int_K \nabla (u-\ellipticProj u) \cdot \nabla w \, \dx + \int_K (u-\ellipticProj u) \, w \, \dx.
\end{align*}
for all $v,w \in V$. 
We assume the validity of the following regularity result which is shown in the case of a rectangular domain in \cite{elliott1989nonconforming}: there exists a solution $z \in H^{3}(\Omega)$ to \eqref{eqn: dual problem} such that
% \red{Make assumption as shown for rectangular domain in  or check \cite{gazzola2010polyharmonic}}. Then it follows from \cite{elliott1989nonconforming} in the case of a rectangular domain that 
\begin{align}\label{eqn: elliptic regularity}
    \| z \|_{3,\Omega} 
    % \leq \| l_h \|_{-1} \leq \| l_h \|_{-1,h} 
    \leq C_{\Omega} \| u - \ellipticProj u \|_{1,h},
\end{align}
where $C_{\Omega}$ depends only on $\Omega$.

\begin{lemma}\label{lemma: term 2}
    Let \corrections{$u \in H^{4}(\Omega) \cap H^{\polOrder+1}(\Omega)$} be the solution to \eqref{eqn: cts weak form} and let $\ellipticProj u$ be the elliptic projection defined in \eqref{eqn: elliptic proj}. 
    For the solution $z \in H^3(\Omega)$ of the dual problem \eqref{eqn: dual problem}, it holds that  
    \begin{align}\label{eqn: lemma term 2}
        \left| b(z,u-\ellipticProj u) \right|
        \lesssim 
        h^{\polOrder} \| u -\ellipticProj u \|_{1,h}.
    \end{align}
\end{lemma}
We give the proof of Lemma \ref{lemma: term 2} in Appendix \ref{sec: appendix}. 
\ifthenelse{\boolean{thesis}}{
Want to define the functional $l_h$ such that 
\begin{align*}
    \| l_h \|_{-1,h} &\leq \| u - \ellipticProj u \|_{1,h}
    \intertext{and}
    l_h (u - \ellipticProj u) &= \| u - \ellipticProj u \|_{1,h}^2
\end{align*}
clear that this is satisfied by choosing
\begin{align*}
    l_h (w_h) := \sum_{K \in \mesh} \int_K \nabla (u-\ellipticProj u) \cdot \nabla w_h + \int_K (u-\ellipticProj u) \, w_h 
\end{align*}
for any $w_h \in \vemSpace$.
}
{}
Also necessary for the proof of \eqref{eqn: elliptic proj time in 2 norm}-\eqref{eqn: elliptic proj time in 1 norm} in Lemma \ref{lemma: elliptic proj bounds} is the following lemma, the proof is given in Appendix \ref{sec: appendix}.
\begin{lemma}\label{lemma: extra terms}
    Let \corrections{$u \in H^{4}(\Omega) \cap H^{\polOrder+1}(\Omega)$} be the solution to \eqref{eqn: cts weak form} and let $\ellipticProj u$ be the elliptic projection defined in \eqref{eqn: elliptic proj}. For any $\eta_h \in \vemSpace$, it holds
    \begin{align*}
        | (\phi^{\prime \prime}(u) u_t \, \gradProj \ellipticProj u,\gradProj \eta_h)_{0,h} 
        - 
        (\phi^{\prime \prime}(u) u_t \, \nabla u,\nabla \eta_h)_{0,h} 
        |
        \lesssim
        h^{\polOrder} \| \eta_h \|_{2,h}.
    \end{align*}
\end{lemma}
We now give the proof of Lemma \ref{lemma: elliptic proj bounds}.
\begin{proof}[Proof of Lemma \ref{lemma: elliptic proj bounds}]   
Define $\rho := u-\ellipticProj u$. Then, it follows that 
\begin{align*}
    \|\rho \|_{1,h}^2 
    =
    \left[ L_h(\rho)
    - b(z,\rho) 
    \right]
    +  b(z,\rho) 
    = I + II.
\end{align*}
In view of Lemma \ref{lemma: term 2} we only need to estimate term $I$. 
To this end, we use the same technique considered in \cite{zhao_morley-type_2018} for proving $H^1$ estimates for the fourth-order plate bending problem and introduce the interpolation $v_h^*$ into the lowest order $H^1$-conforming VEM space presented in e.g. \cite{beirao_da_veiga_basic_2013}.
The following estimate therefore holds. For any $w \in H^2(K)$, there exists $w^*$ in the lowest order $H^1$-conforming VEM space such that    
\begin{align}\label{eqn: interpolation lowest order estimates}
    \| w - w^* \|_{0,\element} + h_K | w -w^* |_{1,K} \lesssim h_K^2 |w|_{2,K}.
\end{align}
Notice that it also follows from \cite[(5.6)]{elliott1989nonconforming} using the density of $V$ in $H^1$ that for any $v \in H^1(\Omega)$ 
\begin{align}\label{eqn: dual problem with H1 vem function}
    \eps^2 (-\nabla \Delta z,\nabla v) + r(u;z,v) + \alpha (z,v) = L_h(v).
\end{align}
Therefore, using \eqref{eqn: dual problem with H1 vem function} with $v=\rho^*$ we see that 
\begin{align*}
    I = \, &
    L_h(\rho) - b(z,\rho) + \eps^2 (-\nabla \Delta z,\nabla \rho^*) + r(u;z,\rho^*) + \alpha (z,\rho^*) - L_h(\rho^*)
    \\
    = \, &
    L_h(\rho-\rho^*) - \eps^2 (D^2 z ,D^2 \rho) + \eps^2 (-\nabla \Delta z,\nabla \rho^*) 
    + r(u;z,\rho^*-\rho) + \alpha(z,\rho^*-\rho).
\end{align*}
We also observe that the following holds using integration by parts
\begin{align*}
    \eps^2 (D^2 z,D^2 \rho) =
    - \eps^2 (\nabla \Delta z,\nabla \rho)
    +
    \eps^2 \sum_{K \in \mesh} \int_{\partial K} \big( (\Delta z - \partial_{ss} z ) \partial_n \rho  + \partial_{ns} z \partial_s \rho \big) \, \ds
\end{align*}
and so, after combining this with the expression for $I$, \corrections{we get}
\begin{align*}
    I 
    % = \, & L_h(\rho-\rho_h^*) + \eps^2 (-\nabla \Delta z,\nabla \rho_h^*) - \eps^2 (D^2 z ,D^2 \rho) + r(u;z,\rho_h^*-\rho) + \alpha(z,\rho_h^*-\rho) \\
    = & \, 
    L_h(\rho-\rho^*) + \eps^2 ( \nabla \Delta z , \nabla (\rho - \rho^*)) + r(u;z,\rho^*-\rho) + \alpha(z,\rho^*-\rho) 
    \\
    &+ 
    \eps^2 \sum_{K \in \mesh} \int_{\partial K} \big( (\Delta z - \partial_{ss} z ) \partial_n \rho  + \partial_{ns} z \partial_s \rho \big) \, \ds
    =: I_1 + I_2 + I_3 + I_4 + I_5.
\end{align*}
Using Cauchy-Schwarz and \eqref{eqn: approximation properties of interpolation operator} it holds that
\begin{align*}
    I_1 \leq | L_h(\rho-\rho^*) | &= \left| \sum_{K \in \mesh} \int_K \nabla \rho \cdot \nabla (\rho -\rho^*) + \rho (\rho - \rho^*) \, \dx 
     \right|
    % &\leq  
    % | \rho |_{1,h} | \rho - \rho_h^*|_{1,h} + \| \rho \|_{0,h} \| \rho - \rho_h^* \|_{0,h}
    \\
    &\leq \sum_{\element \in \mesh} \| \rho \|_{1,\element} \| \rho -\rho^* \|_{1,\element}
    \\
    &\lesssim \| u - \ellipticProj u \|_{1,h} h | u - \ellipticProj u |_{2,h}
    \lesssim h^{\polOrder} \| u - \ellipticProj u \|_{1,h}.
    % \\
    % &\leq \| u - \ellipticProj u \|_{1,h} \| u - \ellipticProj - (I_h u - \ellipticProj u)^* \|_{1,h}
    % \\
    % &\leq \| \rho \|_{1,h} ( \| u - \ellipticProj u - (I_h u - \ellipticProj u) \|_{1,h} + \| I_h u - \ellipticProj u - (I_h u - \ellipticProj u)^* \|_{1,h})
    % \\
    % &\leq \| \rho \|_{1,h} ( \| u - I_h u \|_{1,h} + \| \ellipticProj u -  \ellipticProj u^* \|_{1,h})
    % &\lesssim
    % h \|u-\ellipticProj u\|_{1,h} \|u-\ellipticProj u\|_{2,h}
    % \lesssim h^{\polOrder} \| u - \ellipticProj u\|_{1,h}.
\end{align*}
In the last step we have applied \eqref{eqn: elliptic proj in 2norm}. Similarly for the next three terms, we can apply the lowest order conforming interpolation estimates from \eqref{eqn: interpolation lowest order estimates}, the properties of the elliptic projection \eqref{eqn: elliptic proj in 2norm} and the regularity results \eqref{eqn: elliptic regularity}, \corrections{yielding}
\begin{align*}
    I_2 \leq \left| \eps^2 (\nabla \Delta z,\nabla (\rho - \rho^*)) \right| \lesssim 
    \eps^2 \| z\|_{3,\Omega} h |u-\ellipticProj u|_{2,h} 
    &\lesssim 
    h^{\polOrder} \| u-\ellipticProj u\|_{1,h},
    \\
    I_3 \leq | r (u;z,\rho^* - \rho) |
    \lesssim
    h \| z\|_{3,\Omega} |u-\ellipticProj u|_{2,h}
    &\lesssim
    h^{\polOrder} \| u - \ellipticProj u \|_{1,h},
    \\ 
    I_4 \leq | \alpha (z,\rho^* - \rho ) |
    \lesssim
    h \| z \|_{3,\Omega} |u-\ellipticProj u|_{2,h}
    &\lesssim
    h^{\polOrder} \| u - \ellipticProj u \|_{1,h}.
\end{align*}
Finally, we notice that $I_5$ can be estimated using the same method as in Lemma \ref{lemma: nonconformity error}, the proof of which can be found in \cite[Theorem 5.5]{10.1093/imanum/drab003}.
Therefore
\begin{align*}
    I \lesssim h^{\polOrder}\| u - \ellipticProj u \|_{1,h}
\end{align*}
and \eqref{eqn: elliptic proj in 1norm} holds.

It remains to show \eqref{eqn: elliptic proj time in 2 norm}-\eqref{eqn: elliptic proj time in 1 norm}. For this, we notice that for any $w_h \in \vemSpace$
\begin{align}\label{eqn: time split for bh}
    b_h ((\ellipticProj u)_t,w_h) = \ & b(u_t,w_h) 
    - \nonconformity(u_t,w_h)
    \notag
    \\
    &+
    (\phi^{\prime \prime} (u) u_t \nabla u, \nabla w_h ) - 
    (\phi^{\prime \prime} (u) u_t \gradProj \ellipticProj u, \gradProj w_h).
\end{align}
\jsccorrections{
    To see that \eqref{eqn: time split for bh} holds, we can use the definition of the form $b_h$ which is given in \eqref{eqn: bilinear form b_h}, to write
    \begin{align*}
        b_h ((\ellipticProj u)_t ,w_h)  
        =& \ 
        \eps^2 a_h ( (\ellipticProj u)_t,w_h) + r_h(u; (\ellipticProj u)_t,w_h) + \alpha ( (\ellipticProj u)_t,w_h) 
        \\
        &+ (\phi^{\prime \prime}(u) u_t, \gradProj \ellipticProj u, \gradProj w_h) - (\phi^{\prime \prime}(u) u_t, \gradProj \ellipticProj u, \gradProj w_h)
        \\
        =& \
        \frac{d}{dt} ( b_h( \ellipticProj u,w_h )) - (\phi^{\prime \prime}(u) u_t, \gradProj \ellipticProj u, \gradProj w_h),
    \end{align*}
    % \begin{align*}
    %     (b_h(\ellipticProj u,w_h))_t = (\gamma^2 a_h ( \ellipticProj u,w_h) + r_h(u; \ellipticProj u,w_h) + \alpha ( \ellipticProj u,w_h) )_t
    %     \\
    %     = \gamma^2 a_h ( (\ellipticProj u)_t,w_h) + r_h(u; (\ellipticProj u)_t,w_h) + \alpha ( (\ellipticProj u)_t,w_h) 
    %     \\
    %     &+ (\phi^{\prime \prime}(u) u_t, \gradProj \ellipticProj u, \gradProj w_h)
    % \end{align*}
    where we have used that 
    $$\frac{d}{dt} r_h(u;\ellipticProj u, w_h) = r_h(u;(\ellipticProj u)_t ,w_h) + (\phi^{\prime \prime}(u)u_t, \gradProj \ellipticProj u,\gradProj w_h).$$
    Furthermore, due to \eqref{eqn: elliptic proj}, and the nonlinearity in $r(u;u,w_h)$, it follows that
    \begin{align*}
        b_h ((\ellipticProj u)_t ,w_h)  
        =& \ 
        \frac{d}{dt} ((\eps^2 \Delta^2 u - \nabla \cdot( \phi^{\prime}(u) \nabla u) + \alpha u ,w_h)) 
        - (\phi^{\prime \prime}(u) u_t, \gradProj \ellipticProj u, \gradProj w_h)
        \\
        =& \
        \frac{d}{dt} (b(u,w_h) - \nonconformity(u,w_h))
        - (\phi^{\prime \prime}(u) u_t, \gradProj \ellipticProj u, \gradProj w_h)
        \\
        =& \ 
        b(u_t,w_h) + (\phi^{\prime \prime}(u) u_t \nabla u,\nabla w_h) - \nonconformity(u_t,w_h)
        \\
        &- (\phi^{\prime \prime}(u) u_t, \gradProj \ellipticProj u, \gradProj w_h),
    \end{align*}
    as required.
    % \begin{align*}
    %     \frac{d}{dt} (b_h(\ellipticProj u,w_h)) 
    %     &= 
    %     \frac{d}{dt} ((\eps^2 \Delta^2 u - \nabla \cdot( \phi^{\prime}(u) \nabla u) + \alpha u ,w_h))
    %     \\
    %     &= 
    %     \frac{d}{dt} (b(u,w_h) - \nonconformity(u,w_h))
    %     \\
    %     &= 
    %     b(u_t,w_h) + (\phi^{\prime \prime}(u) u_t \nabla u,\nabla w_h) - \nonconformity(u_t,w_h)
    % \end{align*}
}
\jsccorrections{Recall that we can choose the constant $\alpha$ in the definition of $b_h$ in \eqref{eqn: bilinear form b_h} so that $b_h$ is coercive. Therefore,} using the coercivity of the bilinear form $b_h$, alongside \eqref{eqn: time split for bh} and the definition of $\ellipticProj u$, we have the following 
\begin{align*}
    \| \ellipticProj (u_t) - (\ellipticProj u)_t \|_{2,h}^2 \lesssim \ &
    b_h(\ellipticProj (u_t) - (\ellipticProj u)_t,\ellipticProj (u_t) - (\ellipticProj u)_t)
    \\
    = \ &
    b_h ( \ellipticProj (u_t) , \ellipticProj (u_t) - (\ellipticProj u)_t)
    -
    b_h ( (\ellipticProj u)_t, \ellipticProj (u_t) - (\ellipticProj u)_t,)
    \\
    % = \ &
    % (\eps^2 \Delta^2 u_t - \nabla \cdot (\phi^{\prime} (u) \nabla u_t) + \alpha u_t, \ellipticProj (u_t) - (\ellipticProj u)_t )
    % \\
    % &-
    % b(u_t,\ellipticProj (u_t) - (\ellipticProj u)_t) 
    % +
    % \nonconformity(u_t,\ellipticProj (u_t) - (\ellipticProj u)_t)
    % \\
    % &-
    % (\phi^{\prime \prime} (u) u_t \nabla u, \nabla \ellipticProj (u_t) - (\ellipticProj u)_t ) 
    % +
    % (\phi^{\prime \prime} (u) u_t \gradProj \ellipticProj u, \gradProj \ellipticProj (u_t) - (\ellipticProj u)_t)
    % \\
    = \ &
    (\phi^{\prime \prime} (u) u_t \gradProj \ellipticProj u, \gradProj \big( \ellipticProj (u_t) - (\ellipticProj u)_t \big) )
    \\
    &- 
    (\phi^{\prime \prime} (u) u_t \nabla u, \nabla \big( \ellipticProj (u_t) - (\ellipticProj u)_t \big) )
    \\
    \lesssim \ & 
    h^{\polOrder} \| \ellipticProj (u_t) - (\ellipticProj u)_t  \|_{2,h},
\end{align*}
where we have applied Lemma \ref{lemma: extra terms} with $\eta_h = \ellipticProj (u_t) - (\ellipticProj u)_t$ in the last step. 
It therefore follows that 
\begin{align*}
    \| u_t - (\ellipticProj u)_t \|_{2,h} \leq 
    \| u_t - \ellipticProj (u_t) \|_{2,h} + \| \ellipticProj (u_t) - (\ellipticProj u)_t \|_{2,h} 
    \lesssim
    h^{\polOrder-1}
\end{align*}
\jsccorrections{where we have bound the first term using \eqref{eqn: elliptic proj in 2norm}.}

In order to prove \eqref{eqn: elliptic proj time in 1 norm}, we proceed in the exact same way as the proof of \eqref{eqn: elliptic proj in 1norm}. We consider again a dual problem: find $\tilde z \in V$ such that 
\begin{align*}
    b(\tilde z,w) =
    ( u_t - (\ellipticProj u)_t , w )_{1,h} + ( u_t - (\ellipticProj u)_t ,w)_{0,h}
    % =: \tilde L_h(w)
    \quad \forall \, w \in V.
\end{align*}
\jsccorrections{We can then follow in the exact steps of the proof of Lemma~\ref{lemma: term 2} and introduce the interpolation of $\tilde z$, $\globalInterpolation \tilde z$, in order to bound $b(\tilde z, u_t - (\ellipticProj u)_t)$. This, together with a regularity result for $\tilde z$ gives us the following estimate
%  that
% We can then use Lemma \ref{lemma: extra terms} with $\eta_h = \globalInterpolation \tilde z$ as well as a regularity result for $\tilde z$ to show that 
\begin{align*}
    | b( \tilde z, u_t - (\ellipticProj u)_t) | \lesssim h^{\polOrder} \| u_t - (\ellipticProj u)_t \|_{1,h}
\end{align*}}
and the result now follows as before, \jsccorrections{following the steps of the proof of \eqref{eqn: elliptic proj in 1norm}}. 
\end{proof}

\ifthenelse{\boolean{thesis}}{
\begin{align*}
    r(u;u_t,\psi_h) = \frac{d}{dt} r(u;u,\psi_h) - (\phi^{\prime \prime} (u) u_t \nabla u, \nabla \psi_h )
    \\
    r_h (u;u_t,\psi_h) = \frac{d}{dt} r_h(u;u,\psi_h) - (\phi^{\prime \prime} (u) u_t \gradProj u, \gradProj \psi_h)
\end{align*}
\begin{align*}
    \frac{d}{dt} (b_h(\ellipticProj u,\psi_h)) = b_h ((\ellipticProj u)_t,\psi_h) + (\phi^{\prime \prime} (u) u_t \gradProj \ellipticProj u, \gradProj \psi_h) 
    \\
    \frac{d}{dt} ( b(u,\psi_h)) = b(u_t,\psi_h) +  (\phi^{\prime \prime} (u) u_t \nabla u, \nabla \psi_h )
\end{align*}
\begin{align*}
    b_h ((\ellipticProj u)_t,\psi_h) + (\phi^{\prime \prime} (u) u_t \gradProj \ellipticProj u, \gradProj \psi_h)  = \frac{d}{dt} (b_h(\ellipticProj u,\psi_h))
    \\
    =
    \frac{d}{dt} ( (\eps^2 \Delta^2 u  - \nabla \cdot(\phi^{\prime}(u) \nabla u) + \alpha u,\psi_h ) )
    \\
    =
    \frac{d}{dt} ( b(u,\psi_h) - \nonconformity(u,\psi_h))
    \\
    =
    b(u_t,\psi_h) + (\phi^{\prime \prime} (u) u_t \nabla u, \nabla \psi_h ) - \nonconformity(u_t,\psi_h)
\end{align*}
\begin{align*}
    b_h ((\ellipticProj u)_t,\psi_h) - b(u_t,\psi_h) 
    =
    (\phi^{\prime \prime} (u) u_t \nabla u, \nabla \psi_h ) - 
    (\phi^{\prime \prime} (u) u_t \gradProj \ellipticProj u, \gradProj \psi_h) 
    -
    \nonconformity(u_t,\psi_h)
\end{align*}
\red{Re-define the right hand side of $b$ with $u_t - (\ellipticProj u)_t$?, and apply with $\psi_h=z_h$}
\red{Go through for remaining term}

Triangle inequality
\begin{align*}
    \| u_t - (\ellipticProj u)_t \|_{2,h} \leq \| u_t - \ellipticProj (u_t) \|_{2,h} + \| \ellipticProj (u_t) - (\ellipticProj u)_t \|_{2,h}
\end{align*}

Note that 
\begin{align*}
    b_h ((\ellipticProj u)_t,\psi_h) 
    = \ &
    b(u_t,\psi_h) 
    -
    \nonconformity(u_t,\psi_h)
    \\
    &+
    (\phi^{\prime \prime} (u) u_t \nabla u, \nabla \psi_h ) - 
    (\phi^{\prime \prime} (u) u_t \gradProj \ellipticProj u, \gradProj \psi_h) 
\end{align*}

Using coercivity of the bilinear form $b_h$
\begin{align*}
    \| \ellipticProj (u_t) - (\ellipticProj u)_t \|_{2,h}^2 \lesssim \ &
    b_h(\ellipticProj (u_t) - (\ellipticProj u)_t,\ellipticProj (u_t) - (\ellipticProj u)_t)
    \\
    = \ &
    b_h ( \ellipticProj (u_t) , \ellipticProj (u_t) - (\ellipticProj u)_t)
    -
    b_h ( (\ellipticProj u)_t, \ellipticProj (u_t) - (\ellipticProj u)_t,)
    \\
    = \ &
    (\eps^2 \Delta^2 u_t - \nabla \cdot (\phi^{\prime} (u) \nabla u_t) + \alpha u_t, \ellipticProj (u_t) - (\ellipticProj u)_t )
    \\
    &-
    b(u_t,\ellipticProj (u_t) - (\ellipticProj u)_t) 
    +
    \nonconformity(u_t,\ellipticProj (u_t) - (\ellipticProj u)_t)
    \\
    &-
    (\phi^{\prime \prime} (u) u_t \nabla u, \nabla \ellipticProj (u_t) - (\ellipticProj u)_t ) 
    +
    (\phi^{\prime \prime} (u) u_t \gradProj \ellipticProj u, \gradProj \ellipticProj (u_t) - (\ellipticProj u)_t)
    \\
    = \ &
    (\phi^{\prime \prime} (u) u_t \gradProj \ellipticProj u, \gradProj \ellipticProj (u_t) - (\ellipticProj u)_t)
    - 
    (\phi^{\prime \prime} (u) u_t \nabla u, \nabla \ellipticProj (u_t) - (\ellipticProj u)_t ) 
\end{align*}

Therefore
\begin{align*}
    \| \ellipticProj u_t - (\ellipticProj u)_t \|_{2,h} \lesssim h^{\polOrder-1}
\end{align*}

Similarly, for \eqref{eqn: elliptic proj time in 1 norm}, we can consider the dual problem again \eqref{eqn: dual problem} with right hand side now given by 
\begin{align*}
    \tilde L_h (w) := ( u_t - (\ellipticProj u)_t , w )_{1,h} + ( u_t - (\ellipticProj u)_t ,w)_{0,h}
\end{align*}
and noting that 
}
{}

\ifthenelse{\boolean{thesis}}{
    \begin{lemma}[Boundedness of projection]
        It holds that 
        \begin{align*}
            \| \ellipticProj u \|_{1,\infty;h} + \| (\ellipticProj u)_t \|_{1,\infty;h} \leq C
        \end{align*}
    \end{lemma}

    \begin{proof}
        \red{Go over (3.2c) in Charlies paper again}. Inverse inequality yields
        \begin{align*}
            \| \globalInterpolation u - \ellipticProj u \|_{1,\infty;h} 
            &\leq C h^{-1} \| \globalInterpolation u - \ellipticProj u \|_{1,2;h } 
            \\
            &\leq Ch^{-1} (\| \globalInterpolation u - u \|_{1,2;h }  + \| u - \ellipticProj u \|_{1,2;h }  ) 
            \\
            &\leq Ch^{-1} ( \| \globalInterpolation u - u \|_{1,2;h } + h^{\polOrder}  )
            \\
            &\leq Ch^{-1} (h^{\polOrder} |u|_{\polOrder+1} + h^{\polOrder}  )
            \\
            &\leq Ch^{\polOrder-1}
        \end{align*}
        Therefore 
        \begin{align*}
            \| \ellipticProj u \|_{1,\infty;h} 
            &\leq 
            \| \globalInterpolation u - \ellipticProj u \|_{1,\infty;h} + \| \globalInterpolation u \|_{1,\infty;h}
            \\
            &\leq C h^{-1} (h^{\polOrder} + \|\globalInterpolation u \|_{1,2;h} )
        \end{align*}
    \end{proof}
}
{}

\begin{remark}
    Notice that it follows from \eqref{eqn: approximation properties of interpolation operator} and \eqref{eqn: elliptic proj in 1norm} as well as stability properties of the interpolation operator that the elliptic projection is bounded \cite{elliott1989nonconforming}. In particular,
    \begin{align}\label{eqn: bdd elliptic projection}
        \| \ellipticProj u \|_{1,\infty;h} \leq C.
    \end{align}
    We use this property in the proof of some technical lemmas described in the next subsection.
\end{remark}

\subsection{Technical results}
This intermediate subsection is dedicated to two additional preliminary results that are required before we can prove the $L^2$ error estimate presented in Theorem \ref{thm: L2 convergence}. 
The first is an estimate for the semilinear term $r_h$. 
We use the following standard error decomposition arising in the study of time-dependent problems\corrections{:}
\begin{align}\label{eqn: error decomposition}
    u - u_h = (u - \ellipticProj u ) + (\ellipticProj u- u_h) =: \rho + \theta,
\end{align}
for \corrections{$u$} the solution to \eqref{eqn: cts weak form} and \corrections{$\ellipticProj u$} the projection defined in \eqref{eqn: elliptic proj}. 

\begin{lemma}\label{lemma: nonlinear bounds}
    \corrections{If $u \in H^{4}(\Omega) \cap H^{\polOrder+1}(\Omega)$} it follows that 
    \ifthenelse{\boolean{thesis}}{
    \begin{align}
        \begin{split}
            \big| \rh ( \valueProj u_h;u_h,\theta) - &\rh ( u;\ellipticProj u, \theta) \big|
            \\
            &= \sum_{K \in \cT_h} \int_K \phi^{\prime} (\valueProj u_h ) \, \gradProj u_h \cdot \gradProj \theta \, \dx 
            - 
            \sum_{K \in \cT_h} \int_K \phi^{\prime} (u) \, \gradProj (\ellipticProj u) \cdot \gradProj \theta \,
            \dx
            \\
            &\leq 
            C
            \big(
                \| \theta \|_{0,h} + \| \rho \|_{0,\Omega} + h^{\polOrder} 
                \big)
                | \theta |_{1,h}.
        \end{split}
    \end{align}
    }
    {
        \begin{align}
            \big| \rh ( \Pi^h_0 u_h;u_h,\theta) - &\rh ( u;\ellipticProj u, \theta) \big|
            \lesssim 
            \left( |\theta|_{1,h} + \| \theta \|_{0,h} + \| \rho \|_{0,h} + h^{\polOrder} \right)
                | \theta |_{1,h}.
        \end{align} 
    }
\end{lemma}

\begin{proof}
    Using the definition of $\rh$ we have that 
    \begin{align*}
        \big| \rh (\Pi^h_0 u_h;u_h,\theta) &- \rh ( u;\ellipticProj u, \theta) \big|
        % &= 
        % \sum_{K \in \cT_h} \int_K \Big( 
        %     \phi^{\prime} (\valueProj u_h) \gradProj u_h \cdot \gradProj \theta - \phi^{\prime} (u) \gradProj ( \ellipticProj u) \cdot \gradProj \theta 
        % \Big) 
        % \dx 
        \\
        &=
        \left|
        \sum_{K \in \cT_h} \int_K \big( 
            \phi^{\prime} (\valueProj u_h) \gradProj u_h - \phi^{\prime} (u) \gradProj ( \ellipticProj u) 
        \big) 
        \cdot \gradProj \theta \, \dx
        \right|
        \\
        &\leq 
        \sum_{K \in \mesh}
        \left\| 
            \phi^{\prime} (\valueProj u_h) \gradProj u_h - \phi^{\prime} (u) \gradProj ( \ellipticProj u)   
        \right\|_{0,K} 
        | \theta |_{1,K}
    \end{align*}
    where we have used Lemma \ref{eqn: projections and l2 projection result} since $\theta \in \vemSpace$. Therefore, using the triangle inequality, we see that 
    \begin{align*}
        \left\| 
            \phi^{\prime} (\valueProj u_h) \gradProj u_h - \phi^{\prime} (u) \gradProj ( \ellipticProj u) 
        \right\|_{0,K} 
        \leq& \
        \left\| 
            \phi^{\prime} (\valueProj u_h) \left( \gradProj u_h - \gradProj (\ellipticProj u) \right) 
        \right\|_{0,K} 
        \\
        &+ 
        \| \left( \phi^{\prime}(\valueProj u_h) - \phi^{\prime} (u) \right) \gradProj (\ellipticProj u ) \|_{0,K}
        \\
        &=: J_1+J_2.
    \end{align*}

    To estimate the first term $J_1$ it follows that 
    \begin{align*}
        J_1
        =
        \| 
            \phi^{\prime} (\valueProj u_h) \big(   \gradProj u_h - \gradProj (\ellipticProj u) \big) 
        \|_{0,K}  
        = 
        \| 
            \phi^{\prime} (\valueProj u_h) \, \gradProj \theta 
        \|_{0,\Omega}  
        \lesssim
        |\theta |_{1,K} 
    \end{align*}
    where we have used Assumption \ref{assumption: boundedness of discrete solution}, alongside $L^{\infty}$ stability properties of the $L^2$ projection (for more details see e.g. the theory in \cite{crouzeix1987stability}).

    To estimate the second term $J_2$, we have that
    \begin{align*}
        J_2 = \left\| \left( \phi^{\prime}(\valueProj u_h) - \phi^{\prime} (u) \right) \gradProj (\ellipticProj u ) \right\|_{0,K} 
        &\lesssim
        \| \ellipticProj u \|_{1,\infty;h} 
        \| 
            \valueProj u_h - u
        \|_{0,K} 
    \end{align*}
    where we have used the bounded property of $\ellipticProj u$ in \eqref{eqn: bdd elliptic projection} \jsccorrections{as well as Assumption~\ref{assumption: boundedness of discrete solution} and the definition of $\phi^{\prime}$}. 

    Notice that using the triangle inequality, the definitions of $\rho, \theta$, and properties of the $L^2$ projection, we have
    \begin{align*}
        \| \valueProj u_h - u \|_{0,K} 
        &\leq 
        \| \valueProj u_h - \valueProj \ellipticProj u \|_{0,K}
        + 
        \| \ltwo_{\polOrder} \ellipticProj u - \ltwo_{\polOrder} u \|_{0,K}
        + 
        \| \ltwo_{\polOrder} u - u \|_{0,K}
        \\
        &=
        \| \valueProj \theta  \|_{0,K} 
        +
        \| \ltwo_{\polOrder}  \rho \|_{0,K} 
        +
        \| (I - \ltwo_{\polOrder}) u \|_{0,K}
        \\
        &\lesssim
        \| \theta \|_{0,K} + \| \rho \|_{0,K} + h^{\polOrder}.
    \end{align*}
    % Therefore 
    % \begin{align*}
    %     II 
    %     \lesssim
    %     \| \valueProj u_h - u \|_{0,K} | \ellipticProj u |_{1,\infty;h}
    %     \lesssim
    %     \big( 
    %         \| \theta \|_{0,K} + \| \rho \|_{0,K} + h^{\polOrder} 
    %     \big).
    % \end{align*}

    Hence, by combining the estimates for $J_1$ and $J_2$: 
    \begin{align*}
        \left| \rh( \Pi^h_0 u_h;u_h,\theta) - \rh(u,\ellipticProj u,\theta) \right|
        \lesssim
        \left( 
            | \theta |_{1,h} 
            + 
            \| \theta \|_{0,h} + \| \rho \|_{0,h} + h^{\polOrder} 
        \right)
        |\theta |_{1,h},
    \end{align*}
    as required. 
\end{proof}

\corrections{W}e require one additional lemma. 
% In order to present the convergence result we require one additional lemma. 
The proof is given in Appendix \ref{sec: appendix}.
\begin{lemma}\label{lemma: boundary term bound}
    For any $w_h, z_h \in \vemSpace$, it holds that   
    \begin{align}\label{eqn: bdry lemma}
        \left| \sum_{K \in \cT_h} \int_{\partial K} (\partial_n z_h) w_h \, \ds \right| 
        \lesssim 
        h \left( |w_h|_{1,h} |z_h|_{2,h} + |w_h|_{2,h} |z_h|_{2,h} \right) .
    \end{align} 
\end{lemma}
\subsection{Error estimate for the semidiscrete scheme}
In this subsection we prove the error estimate for the scheme detailed in \eqref{eqn: semidiscrete scheme}. 
\begin{theorem}\label{thm: L2 convergence}
    Assume that \corrections{$u \in H^4(\Omega) \cap H^{\polOrder+1}(\Omega)$} is the solution to the continuous problem \eqref{eqn: cts weak form} and $u_h$ is the solution to \eqref{eqn: semidiscrete scheme}. Then, for all $t \in [0,T]$,
    \begin{align}\label{eqn: L2 convergence}
        \| u - u_h \|_{0,h} \lesssim h^{\polOrder}.
    \end{align}
\end{theorem}

\begin{proof}
    Recall the error decomposition detailed in \eqref{eqn: error decomposition}, and notice that due to Lemma \ref{lemma: elliptic proj bounds}, we only need to estimate $\theta$.
    Following the ideas in \cite{antonietti_$c^1$_2016}, we use the definition of $\theta$ and the semidiscrete scheme \eqref{eqn: semidiscrete scheme} to show that 
    \begin{align}\label{eqn: expansion of discrete form with theta}
        \ahlower ( \theta_t, \chi_h) &+ \eps^2 \ahhessian (\theta,\chi_h ) 
        % = \ &
        % \ahlower ( (\ellipticProj u - u_h)_t , \chi_h ) + \ahhessian (\ellipticProj u - u_h,\chi_h) 
        \notag
        \\
        = \ &
        \ahlower ( (\ellipticProj u)_t , \chi_h) + \eps^2 \ahhessian ( \ellipticProj u ,\chi_h) 
        \nonumber
        - 
        \left( 
            \ahlower (  (u_h)_t , \chi_h) + \eps^2 \ahhessian (  u_h ,\chi_h) 
        \right)
        \notag
        \\ 
        = \ &
        \ahlower ( (\ellipticProj u)_t , \chi_h) + \eps^2 \ahhessian ( \ellipticProj u ,\chi_h) 
        + \rh(\Pi^h_0 u_h;u_h,\chi_h).
    \end{align}
    Using the definition of the elliptic projection in \eqref{eqn: elliptic proj} and \eqref{eqn: bilinear form b_h}, it follows that 
    \begin{align*}
        \eps^2 \ahhessian ( \ellipticProj u, \chi_h) 
        &=
        b_h ( \ellipticProj u ,\chi_h) - \rh ( u; \ellipticProj u, \chi_h) - \alpha (\ellipticProj u, \chi_h) 
        \\
        &= 
        (\eps^2 \Delta^2 u -\nabla \cdot ( \phi^{\prime} (u) \nabla u) + \alpha u , \chi_h ) - \rh ( u; \ellipticProj u, \chi_h) - \alpha (\ellipticProj u, \chi_h). 
    \end{align*}
    Therefore, substituting this into \eqref{eqn: expansion of discrete form with theta} we see that 
    \begin{align}\label{eqn: expanded scheme}
        \ahlower ( \theta_t, \chi_h) + \eps^2 \ahhessian (\theta,\chi_h ) 
        % & \
        % \ahlower ( (\ellipticProj u)_t , \chi_h) 
        % +
        % (\eps^2 \Delta^2 u -\nabla \cdot ( \phi^{\prime} (u) \nabla u) + \alpha u , \chi_h ) 
        % \\
        % &- \alpha (\ellipticProj u, \chi_h) 
        % + \rh(\valueProj u_h;u_h,\chi_h) - \rh ( u; \ellipticProj u, \chi_h) 
        \notag
        =& \
        \ahlower ( (\ellipticProj u)_t , \chi_h)  
        +
        (\eps^2 \Delta^2 u -\nabla \cdot ( \phi^{\prime} (u) \nabla u), \chi_h )  
        \notag
        \\
        &+ \alpha ( u - \ellipticProj u, \chi_h) 
        +
        \rh(\Pi^h_0 u_h;u_h,\chi_h) - \rh ( u; \ellipticProj u, \chi_h) 
        \notag
        \\
        =& \
        \ahlower ( (\ellipticProj u)_t , \chi_h) - (u_t,\chi_h) + \alpha ( \rho ,\chi_h) 
        \notag
        \\
        &+ \rh(\Pi^h_0 u_h;u_h,\chi_h) - \rh ( u; \ellipticProj u, \chi_h). 
    \end{align}

    Following the same method used in \cite{zhao2019nonconforming}, combined with polynomial consistency from Lemma \ref{lemma: consistency}, it follows that 
    \begin{align*}
        | \ahlower ( (\ellipticProj u)_t , \chi_h) - (u_t,\chi_h) |
        = \ & 
        \left|
        \sum_{K \in \cT_h} \ahlowerLocal ( (\ellipticProj u)_t, \chi_h) - (u_t,\chi_h)_K 
        \right| 
        % \\
        % = \ &
        % \sum_{K \in \cT_h} \ahlowerLocal ( (\ellipticProj u)_t, \chi_h) - (u_t ,\chi_h)_K 
        % - \ahlowerLocal( \cP_{\polOrder}^K u_t, \chi_h) + ( \cP_{\polOrder}^K u_t, \chi_h)_K
        \\
        = \ &
        \left|
        \sum_{K \in \cT_h} \ahlowerLocal ( (\ellipticProj u)_t - \cP_{\polOrder}^K u_t, \chi_h)
        -
        (u_t - \cP_{\polOrder}^K u_t , \chi_h)_K
        \right|
        \\ 
        \lesssim \ & 
        \sum_{K \in \cT_h} \| (\ellipticProj u)_t - \cP_{\polOrder}^K u_t \|_{0,K} \| \chi_h \|_{0,K} 
        \\
        &+ \| u_t - \cP_{\polOrder}^K u_t \|_{0,K} \| \chi_h \|_{0,K}.
    \end{align*}
    Where we have used stability of the bilinear form (Lemma \ref{lemma: stability}) in the last step. We now use Lemma \ref{lemma: elliptic proj bounds} and properties of the $L^2$ projection detailed in Theorem \ref{thm: interpolation estimates}, to show that 
    \begin{align*}
        \left| \ahlower ( (\ellipticProj u)_t , \chi_h) - (u_t,\chi_h) \right|
        \lesssim
        h^{\polOrder} \| \chi_h \|_{0,h}.
    \end{align*}

    Now, we take $\chi_h = \theta$ in \eqref{eqn: expanded scheme} and see that   
    \begin{align*}
        \ahlower ( \theta_t, \theta ) + \eps^2 \ahhessian (\theta,\theta )  
        =& \
        \ahlower ((\ellipticProj u)_t, \theta) - (u_t,\theta) + \alpha (\rho, \theta) 
        \\
        &+
        \rh(\Pi^h_0 u_h;u_h,\theta) - \rh ( u; \ellipticProj u, \theta).
    \end{align*} 
    Using stability properties of the discrete forms alongside Lemmas \ref{lemma: elliptic proj bounds} and \ref{lemma: nonlinear bounds}, it holds that
    \begin{align*}
        \half \frac{d}{dt} \| \theta \|_{0,h}^2 + \eps^2 | \theta |_{2,h}^2 
        &\lesssim
            ( h^{\polOrder} + \alpha \| \rho \|_{0,h} ) \| \theta \|_{0,h} 
            + \left| \rh(\Pi^h_0 u_h;u_h,\theta) - \rh ( u; \ellipticProj u, \theta) \right|
        \\
        &\lesssim
            h^{\polOrder} \| \theta \|_{0,h} 
            + ( \| \theta \|_{0,h} + h^{\polOrder} 
            )
            | \theta |_{1,h}.
    \end{align*}
    After an application of Young's inequality, we get 
    \begin{align*}
        \half \frac{d}{dt} \| \theta \|_{0,h}^2 + \eps^2 | \theta |_{2,h}^2 
        &\lesssim
            h^{2\polOrder} + \| \theta \|_{0,h}^2 + | \theta |_{1,h}^2.    
    \end{align*}
    In order to conclude the proof, we observe that for any $v_h \in \vemSpace$ and $0< \gamma \leq \half$, we can use an application of integration by parts, Cauchy-Schwarz, and Lemma \ref{lemma: boundary term bound}, to show the following holds 
    \jsccorrections{\begin{align*}
        | \theta |_{1,h}^2 = \sum_{\element \in \mesh} | \theta |_{1,K}^2 &= \sum_{\element \in \mesh} \left( - \int_K (\Delta \theta) \theta \, \dx + \int_{\partial K} \theta \partial_n \theta \, \ds \right)
        \\
        &\lesssim ( \gamma + h) | \theta |_{2,h}^2 + C_{\gamma} \| \theta \|_{0,h}^2 + h | \theta |_{1,h} | \theta |_{2,h}
        \\
        &\lesssim
        ( \gamma + 3h/2 ) | \theta |_{2,h}^2 + C_{\gamma} \| \theta \|_{0,h}^2 + h/2 |\theta |_{1,h}^{2}.
    \end{align*}
    }\jsccorrections{Where we have applied Young's inequality in the last step. Therefore, provided that $h$ is sufficiently small,} it holds that $|\theta|^2_{1,h} \lesssim (\gamma +3h/2) |\theta|_{2,h}^2 + C_{\gamma} \| \theta \|^2_{0,h}$.
    Hence assuming further that $\gamma$ is also sufficiently small \jsccorrections{i.e. $(\gamma + 3h/2) \leq \half \eps^2$}, it holds that
    \begin{align*}
        \frac{d}{dt} \| \theta \|_{0,h}^2 +  \eps^2 | \theta |_{2,h}^2 
        &\lesssim
            h^{2\polOrder} + \| \theta \|_{0,h}^2.
    \end{align*}
    Therefore equation \eqref{eqn: L2 convergence} follows from an application of Gronwall's lemma. 
\end{proof}

\jsccorrections{\begin{remark}\label{rmk: better CH estimates}
    We note that using the techniques to prove Theorem~\ref{thm: L2 convergence} leads to error bounds which depend on $T$ and $1/\eps^2$ exponentially. 
    It is worth mentioning that there are alternative techniques for error estimates of this form where convergence results lead to polynomially dependence only \cite{li2019error,wu2020analysis}. 
    See also the related work on Allen-Cahn \cite{chrysafinos2020posteriori}.
\end{remark}}
\section{Numerical Results}\label{sec: numerical testing}
In this section we present a fully discrete scheme and investigate its behaviour numerically.
The fully discrete scheme couples the VEM spatial discretization with a Runge-Kutta (RK) scheme, which combines a convex splitting (CS) method with a multi-stage additive RK method. 
\corrections{Since we present a higher order spatial discretization we couple the VEM discretization with a higher order time stepping method.}
We use a nonlinear convex splitting of the energy $E(u)$, defined in \eqref{eqn: energy functional}, see e.g. \cite{elliott1993global,eyre1998unconditionally1,lee2020stability,shin2017unconditionally}, where we split the energy into a contractive and expansive part as follows
\begin{align*}
    E(u) = E_c(u) - E_e(u) = \int_{\Omega} \left( \frac{u^4}{4} + \frac{1}{4} + \frac{\eps^2}{2} |\nabla u |^2 \right) \, \dx - \int_{\Omega} \frac{u^2}{2} \, \dx,
\end{align*}
where $E_c(u)$ is treated implicitly \corrections{while} $E_e(u)$ is treated explicitly. 
It is straightforward to show that $E_c(u)$ and $E_e(u)$ are both convex.
\ifthenelse{\boolean{thesis}}{
    We denote two RK methods using standard Butcher notation
    \begin{align}\label{eqn: Butcher table}
        \begin{array}{c|c}
            c & A
            \\
            \hline
            & b^T      
        \end{array}
        =
        \begin{array}{c|ccc}
            c_1 & a_{11} & \cdots & a_{1s} 
            \\   
            \vdots & \vdots & \ddots & \vdots
            \\
            c_s & a_{s1} & \cdots & a_{ss}
            \\
            \hline
            & b_{1} & \cdots & b_{s}   
        \end{array}
        ,
        \quad 
        \begin{array}{c|c}
            \hat c & \hat A
            \\
            \hline
            & \hat b^T        
        \end{array}
        =
        \begin{array}{c|ccc}
            \hat c_1 & \hat a_{11} & \cdots & \hat a_{1s} 
            \\   
            \vdots & \vdots & \ddots & \vdots
            \\
            \hat c_s & \hat a_{s1} & \cdots & \hat a_{ss}
            \\
            \hline
            & \hat b_{1} & \cdots & \hat b_{s}   
        \end{array}
    \end{align}
    for $A, \hat A \in \R^{s \times s}$, $b,\hat b \in \R^s$ and the coefficients $c, \hat c \in \R^s$ are given by $c=A1, \hat c = \hat A1$. 
}{}

\begin{definition}[Fully discrete scheme]
    We split $[0,T]$ into uniform subintervals, \corrections{each of length $\tau=T/\hat N$} and \corrections{discretize} the solution at time $t_n$, \corrections{${n=0,\dots,\hat N}$,} \corrections{as} $u^{n}_{h \tau} = u_h(\cdot,t_n) \in \vemSpace$. We first split the semilinear form $r_h$ according to the convex splitting above,
    recalling that ${\phi^{\prime} (w) = 3w^2 - 1}$, this is given by
    \begin{align*}
        r_{h,c}  ( u_{h \tau}^n ; v_h, w_h ) &-  r_{h,e} ( u_{h \tau}^n ; v_h, w_h )
        \\
        &
        =
        \sum_{K \in \mesh} \int_K 3 ( u_{h \tau}^n)^2 \gradProj v_h \cdot \gradProj w_h \, \dx
        -
        \sum_{K \in \mesh} \int_K \gradProj v_h \cdot \gradProj w_h \, \dx
    \end{align*}
    for $u_{h \tau}^n, v_h, w_h \in \vemSpace$. 
    Our fully discrete scheme is therefore \corrections{as follows: given $u^0_{h \tau} = u_{h,0} \in \vemSpace$, for $n=0,\dots,\hat N-1,$ find $u_{h \tau}^n \in \vemSpace$ such that}   
    \begin{align*}
        m_h ( u^{n+1}_{h \tau} - u^n_{h \tau} , v_h )
        + \tau \sum_{i=1}^s &\bigg(
        b_i \left( \eps^2 a_h ( U^i , v_h) + r_{h,c} ( \Pi^h_0 U^i ; U^i, v_h) \right) 
        \\
        &- 
        \hat b_i r_{h,e} ( \Pi^h_0 U^i ; U^i, v_h) \bigg) = 0
    \end{align*}
    % for all $v_h \in \vemSpace$,
    where the $i$-th stage is defined by
    \begin{align*}
        m_h ( U^i - u^n_{h\tau},v_h ) + \tau \sum_{j=1}^s &\bigg(
        a_{ij} \left( \eps^2 a_h ( U^j , v_h) + r_{h,c} ( \Pi^h_0 U^j ; U^j, v_h) \right)
        \\ 
        &- \hat a_{ij} r_{h,e} ( \Pi^h_0 U^j ; U^j, v_h) \bigg) = 0  
    \end{align*}
    for all $v_h \in \vemSpace$, for $i=1,\dots,s$.
    The coefficients are defined using standard Butcher notation for RK methods and satisfy $A, \hat A \in \R^{s \times s}$, $b,\hat b \in \R^s$. 
    % and the $c, \hat c \in \R^s$ are given by $c=A1, \hat c = \hat A1$. 
\end{definition}

The first scheme we use is the simple first order CS ``Forward-Backward Euler" method (CSRK-1) \cite{shin2017unconditionally}, with the standard Butcher notation
\begin{align}\label{eqn: IMEX1 Butcher tables}
    \begin{array}{c|c}
        c & A 
        \\
        \hline
        & b^T        
    \end{array}
    =
    \begin{array}{c|cc}
        0 & 0 & 0 
        \\
        1 & 0 & 1 
        \\   
        \hline
         & 0 & 1   
    \end{array}
    ,
    \quad 
    \begin{array}{c|c}
        \hat c & \hat A 
        \\
        \hline
        & \hat b^T        
    \end{array}
    =
    \begin{array}{c|cc}
        0 & 0 & 0 
        \\
        1 & 1 & 0 
        \\   
        \hline
         & 1 & 0   
    \end{array}
    .
\end{align}
The other method we consider is the second order CS Runge-Kutta method (CSRK-2) presented in \cite{lee2020stability}, rewritten in Butcher notation as
\begin{align}\label{eqn: IMEX2 Butcher tables}
    \begin{array}{c|c}
        c & A 
        \\
        \hline
        & b^T        
    \end{array}
    =
    \begin{array}{c|cccc}
        0 & 0 & 0 & 0 & 0\\
        1 & 0 & 1 & 0 & 0\\
        \frac{3}{2} & 0 & \half & 1 & 0\\
        1 & 0 & 1 & -1 & 1 \\
        \hline
        & 0
        & 1 
        & 
        -1
        &
        1
    \end{array}
    ,
    \qquad
    \begin{array}{c|c}
        \hat c & \hat A 
        \\
        \hline
        & \hat b^T        
    \end{array}
    =
    \begin{array}{c|cccc}
        0 & 0 & 0 & 0 & 0\\
        1 & 1 & 0 & 0 & 0\\
        \frac{3}{2} & \half & 1 & 0 & 0\\
        1 & 1 & -1 & 1 & 0\\
        \hline
        & 1 
        & 
        -1
        &
        1
        &
        0
    \end{array}
    .
\end{align}
The schemes defined in both \eqref{eqn: IMEX1 Butcher tables} and \eqref{eqn: IMEX2 Butcher tables} are shown to be energy stable for the semidiscrete (continuous-in-space) schemes \cite{lee2020stability,shin2017unconditionally}. 
We \emph{numerically} investigate the energy decay property of our proposed fully discrete scheme in Test 3 in Section~\ref{sec: evo cross}.

\corrections{We point out that the lowest order ($\polOrder=2$) nonconforming VEM on triangles is identical to the Morley nonconforming finite element and so, with the exception of Test 1 in Section~\ref{subsec: test 1}, we only investigate the behaviour of the higher order $\polOrder=4$ VEM since the other case has been studied extensively in the literature.}

The code we use to carry out the simulations is based on the Distributed and Unified Numerics Environment (\textsc{Dune}) software framework \cite{dunegridpaperII} and has been implemented within the \textsc{Dune-Fem} module \cite{dedner2010generic}. 
\textsc{Dune} is open source software implemented in \texttt{C++}, but since a user has access to a Python frontend \cite{dedner_dune_2018}, they can easily perform numerical experiments by describing mathematical models using the domain specific form language UFL \cite{alnaes_unified_2012}. 
Tutorials including some VEM examples can be found in \cite{dedner2020python} \corrections{while further implementation details can be found in \cite{dedner2022framework}}.

\ifthenelse{\boolean{thesis}}{
% subsection with VEM convergence test, don't include here but submit in thesis chapter. 
\subsection{Test 1: VEM convergence}
For this first experiment we investigate convergence to an exact solution to verify convergence of our VEM discretization applied to the Cahn-Hilliard equation. We set the forcing $f$ so that the exact solution is given by 
\begin{align*}
    u(x,y,t) = \sin(2\pi t)\cos(2\pi x)\cos(2\pi y). 
\end{align*}
In this test we use a small time step, namely $\tau=10^{-5}$ and fix $\eps = 1/10$. We run the convergence test on both a structured simplex grid and a Voronoi grid for discretizing the unit square. The $L^2, H^1,$ and $H^2$ errors are computed at every time step and we present the maximum errors in each case.

We investigate the convergence for polynomial orders $\polOrder=2$ and for the first order Forward-Backward Euler (CSRK-1) \eqref{eqn: IMEX1 Butcher tables} time stepping method. 

The results detailed in Tables \ref{tab: ch_convergencespatial_FE-BE_criss_p2} and \ref{tab: ch_convergencespatial_FE-BE_voronoi_p2} show the errors and experimental orders of convergence (eocs) from Test 1, and are consistent with the result from Theorem \ref{thm: L2 convergence}.

}
{}

% \red{Make a remark that we will focus on CSRK-2 and $l=4$ except for Test 1 and why}
% -----------------------------------------------------------------------------------
% -- Convergence test
% -----------------------------------------------------------------------------------

\subsection{Test 1: \corrections{Convergence to an exact solution}}\label{subsec: test 1}
The first numerical experiment we consider is a non-physical test to recover the order of convergence presented in Theorem~\ref{thm: L2 convergence} 
\corrections{as we investigate convergence} to a \corrections{known} exact solution.
In this test we reduce $\tau$ alongside $h$, starting with $\tau = 10^{-2}$. We fix $\eps = 1/10$ for these experiments and run them on both the structured simplex ``criss'' grid consisting of half square triangles and a sequence of Voronoi grids discretizing the unit square $\Omega = (0,1)^2$. 
The Voronoi grids are randomly seeded and smoothed using Lloyd's algorithm.
The $L^2, H^1,$ and $H^2$ errors are computed at every time step and we present the maximum errors in each case.
We set the forcing $f$ so that the exact solution is given by 
\begin{align*}
    u(x,y,t) = \sin(2\pi t)\cos(2\pi x)\cos(2\pi y). 
\end{align*}
This is a modification of the test considered in e.g. \cite{antonietti_$c^1$_2016,chave2016hybrid}
where we have changed the exact solution to be nonlinear in $t$.   

We investigate the convergence for polynomial orders $\polOrder=2,4$ for the CSRK-1 \eqref{eqn: IMEX1 Butcher tables} and CSRK-2 \eqref{eqn: IMEX2 Butcher tables} time stepping methods. 
\corrections{For the sake of brevity we don't show results for $\polOrder=3$ but note that they are in line with expectations.}   
The results on the simplex and Voronoi polygonal grids are shown in Tables~\ref{tab: ch_convergencetemporal_FE-BE_criss_p2}-\ref{tab: ch_convergencetemporal_SecondOrder1_criss_p4} and Tables~\ref{tab: ch_convergencetemporal_FE-BE_voronoi_p2}-\ref{tab: ch_convergencetemporal_SecondOrder1_voronoi_p4}, respectively.

The results from this convergence test are in line with expectations. Since the CSRK-1 method is first order accurate in time and the CSRK-2 method is second order accurate, when combining one of these methods with the order $\polOrder$ VEM method, we should expect to see $L^2$ convergence rates of order $\min\{ \polOrder,m \}$ where $m \in \{1,2 \}$ is the order from the CSRK-$m$ method. 
To see the optimal $L^2$ convergence $O(h^{\polOrder})$ (Theorem \ref{thm: L2 convergence}), we look at the $\polOrder=2$ VEM method coupled with the CSRK-2 time stepping. 
These results are shown in Tables~\ref{tab: ch_convergencetemporal_SecondOrder1_criss_p2} and \ref{tab: ch_convergencetemporal_SecondOrder1_voronoi_p2}.

Note that the convergence rates in the $H^1$ and $H^2$ norms are also what we would expect to see if we extended Theorem \ref{thm: L2 convergence} to include convergence results in both $H^1$ and $H^2$ norms according to classic FE theory. 
\corrections{For example, the rate for the $H^2$ error with $\polOrder=2$ is equal to $1$ while the higher order is clearly visible in the $\polOrder=4$, CSRK-2 results (Tables~\ref{tab: ch_convergencetemporal_SecondOrder1_criss_p4} and \ref{tab: ch_convergencetemporal_SecondOrder1_voronoi_p4}).}

% \begin{table}\label{tab: test 1 simplex grid result}
    % \subcaption{Test 1: $L^2, H^1,$ and $H^2$ errors and convergence rates for the convergence test on a structured simplex grid.}
    \begin{table}[!h]
   \caption{Test 1: $L^2, H^1,$ and $H^2$ errors and convergence rates for the convergence test with CSRK-1 time stepping and polynomial order $\polOrder=2$ on a structured simplex grid.}
   \label{tab: ch_convergencetemporal_FE-BE_criss_p2}
   {\small
   \begin{tabular*}{\textwidth}{@{\extracolsep{\fill}}rrc|cc|cc|cc}
      \toprule
      size &  dofs &    $h$ & $L^2$-error & $L^2$-eoc & $H^1$-error & $H^1$-eoc & $H^2$-error & $H^2$-eoc \\
      \midrule
      50 &   121 & 0.2828 &  1.9412e-01 &       --- &  1.8409e+00 &       --- &  1.8815e+01 &       --- \\
      200 &   441 & 0.1414 &  9.1284e-02 &      1.09 &  8.5348e-01 &      1.11 &  8.9738e+00 &      1.07 \\
      800 &  1681 & 0.0707 &  4.5175e-02 &      1.01 &  4.1419e-01 &      1.04 &  4.3675e+00 &      1.04 \\
      3200 &  6561 & 0.0354 &  2.2522e-02 &       1.0 &  2.0391e-01 &      1.02 &  2.1492e+00 &      1.02 \\
      \bottomrule
   \end{tabular*}}
\end{table}
    \begin{table}[!h]
   \caption{Test 1: $L^2, H^1,$ and $H^2$ errors and convergence rates for the convergence test with CSRK-2 time stepping and polynomial order $\polOrder=2$ on a structured simplex grid.}
   \label{tab: ch_convergencetemporal_SecondOrder1_criss_p2}
   {\small
   \begin{tabular*}{\textwidth}{@{\extracolsep{\fill}}rrc|cc|cc|cc}
      \toprule
      size &  dofs &    $h$ & $L^2$-error & $L^2$-eoc & $H^1$-error & $H^1$-eoc & $H^2$-error & $H^2$-eoc \\
      \midrule
      50 &   121 & 0.2828 &  1.1203e-01 &       --- &  1.1001e+00 &       --- &  1.6705e+01 &       --- \\
      200 &   441 & 0.1414 &  2.9845e-02 &      1.91 &  3.0147e-01 &      1.87 &  8.0622e+00 &      1.05 \\
      800 &  1681 & 0.0707 &  7.3555e-03 &      2.02 &  7.6541e-02 &      1.98 &  3.9588e+00 &      1.03 \\
      3200 &  6561 & 0.0354 &  1.7710e-03 &      2.05 &  1.8949e-02 &      2.01 &  1.9692e+00 &      1.01 \\
      \bottomrule
   \end{tabular*}}
\end{table}
    \begin{table}[!h]
   \caption{Test 1: $L^2, H^1,$ and $H^2$ errors and convergence rates for the convergence test with CSRK-2 time stepping and polynomial order $\polOrder=4$ on a structured simplex grid.}
   \label{tab: ch_convergencetemporal_SecondOrder1_criss_p4}
   {\small
   \begin{tabular*}{\textwidth}{@{\extracolsep{\fill}}rrc|cc|cc|cc}      
      \toprule
      size &  dofs &    $h$ & $L^2$-error & $L^2$-eoc & $H^1$-error & $H^1$-eoc & $H^2$-error & $H^2$-eoc \\
      \midrule
      50 &   511 & 0.2828 &  2.2046e-02 &       --- &  1.9803e-01 &       --- &  1.8769e+00 &       --- \\
      200 &  1921 & 0.1414 &  5.5596e-03 &      1.99 &  5.0125e-02 &      1.98 &  4.8603e-01 &      1.95 \\
      800 &  7441 & 0.0707 &  1.0552e-03 &       2.4 &  9.7040e-03 &      2.37 &  1.0401e-01 &      2.22 \\
      3200 & 29281 & 0.0354 &  1.6197e-04 &       2.7 &  1.6259e-03 &      2.58 &  2.6636e-02 &      1.97 \\
      \bottomrule
   \end{tabular*}}
\end{table}
% \end{table}         

% \begin{table}\label{tab: test 1 voronoi grid result}
    % \caption{Test 1: $L^2, H^1,$ and $H^2$ errors and convergence rates for the convergence test on a Voronoi polygonal grid.}
    \begin{table}
   \caption{Test 1: $L^2, H^1,$ and $H^2$ errors and convergence rates for the convergence test with CSRK-1 time stepping and polynomial order $\polOrder=2$ on a Voronoi polygonal grid.}
   \label{tab: ch_convergencetemporal_FE-BE_voronoi_p2}
   {\small
   \begin{tabular*}{\textwidth}{@{\extracolsep{\fill}}rrc|cc|cc|cc}
      \toprule
      size &  dofs &    $h$ & $L^2$-error & $L^2$-eoc & $H^1$-error & $H^1$-eoc & $H^2$-error & $H^2$-eoc \\
      \midrule
         25 &   128 & 0.3288 &  2.5850e-01 &       --- &  2.3431e+00 &       --- &  2.2020e+01 &       --- \\
      100 &   503 & 0.1535 &  1.1302e-01 &      1.09 &  1.0281e+00 &      1.08 &  1.0548e+01 &      0.97 \\
      400 &  2003 & 0.0751 &  5.1155e-02 &      1.11 &  4.5992e-01 &      1.13 &  4.9724e+00 &      1.05 \\
      1600 &  8003 & 0.0402 &  2.3960e-02 &      1.21 &  2.1488e-01 &      1.22 &  2.4843e+00 &      1.11 \\
      \bottomrule
   \end{tabular*}}
\end{table}
    \begin{table}
   \caption{Test 1: $L^2, H^1,$ and $H^2$ errors and convergence rates for the convergence test with CSRK-2 time stepping and polynomial order $\polOrder=2$ on a Voronoi polygonal grid.}
   \label{tab: ch_convergencetemporal_SecondOrder1_voronoi_p2}
   {\small
   \begin{tabular*}{\textwidth}{@{\extracolsep{\fill}}rrc|cc|cc|cc}
      \toprule
      size &  dofs &    $h$ & $L^2$-error & $L^2$-eoc & $H^1$-error & $H^1$-eoc & $H^2$-error & $H^2$-eoc \\
      \midrule
         25 &   128 & 0.3288 &  1.9833e-01 &       --- &  1.7472e+00 &       --- &  2.0820e+01 &       --- \\
      100 &   503 & 0.1535 &  4.6127e-02 &      1.92 &  4.5037e-01 &      1.78 &  9.9688e+00 &      0.97 \\
      400 &  2003 & 0.0751 &  1.0867e-02 &      2.02 &  1.0755e-01 &       2.0 &  4.9206e+00 &      0.99 \\
      1600 &  8003 & 0.0402 &  2.5869e-03 &      2.29 &  2.5990e-02 &      2.27 &  2.4679e+00 &       1.1 \\
      \bottomrule
   \end{tabular*}}
\end{table}
    \begin{table}
   \caption{Test 1: $L^2, H^1,$ and $H^2$ errors and convergence rates for the convergence test with CSRK-2 time stepping and polynomial order $\polOrder=4$ on a Voronoi polygonal grid.}
   \label{tab: ch_convergencetemporal_SecondOrder1_voronoi_p4}
   {\small
      \begin{tabular*}{\textwidth}{@{\extracolsep{\fill}}rrc|cc|cc|cc}
         \toprule
         size &  dofs &    $h$ & $L^2$-error & $L^2$-eoc & $H^1$-error & $H^1$-eoc & $H^2$-error & $H^2$-eoc \\
      \midrule
         25 &   457 & 0.3288 &  2.2198e-02 &       --- &  1.9974e-01 &       --- &  1.9335e+00 &       --- \\
         100 &  1807 & 0.1535 &  5.5609e-03 &      1.82 &  5.0147e-02 &      1.81 &  4.8954e-01 &       1.8 \\
         400 &  7207 & 0.0751 &  1.0552e-03 &      2.33 &  9.7041e-03 &       2.3 &  1.0424e-01 &      2.16 \\
         1600 & 28807 & 0.0402 &  1.6145e-04 &       3.0 &  1.6223e-03 &      2.86 &  2.6661e-02 &      2.18 \\
         \bottomrule
      \end{tabular*}}
\end{table}
% \end{table}

% -----------------------------------------------------------------------------------
% -- Two bubbles
% -----------------------------------------------------------------------------------
\begin{figure}
    \begin{center}        
    \subfloat[$t=0$]{
        \includegraphics[width=0.275\textwidth]{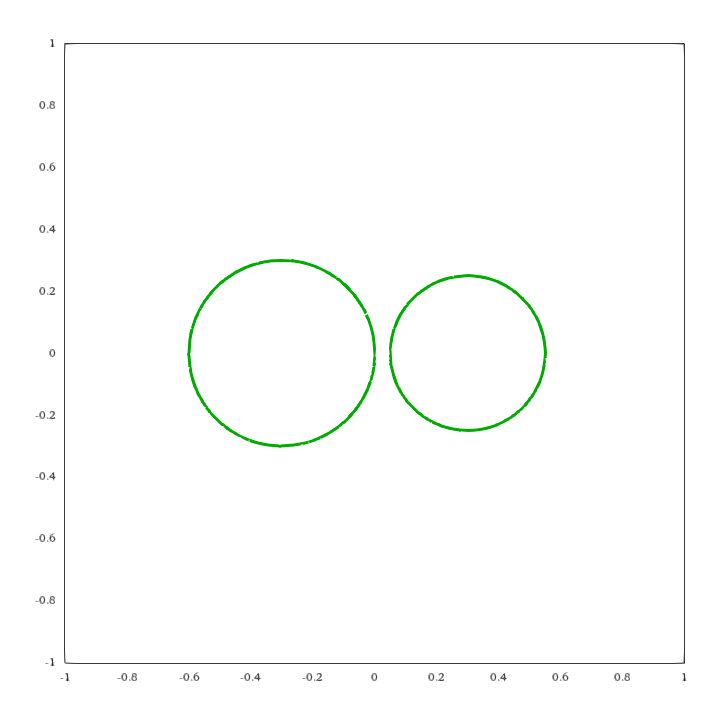}
    }
    \subfloat[$t=0.004$]{
        \includegraphics[width=0.275\textwidth]{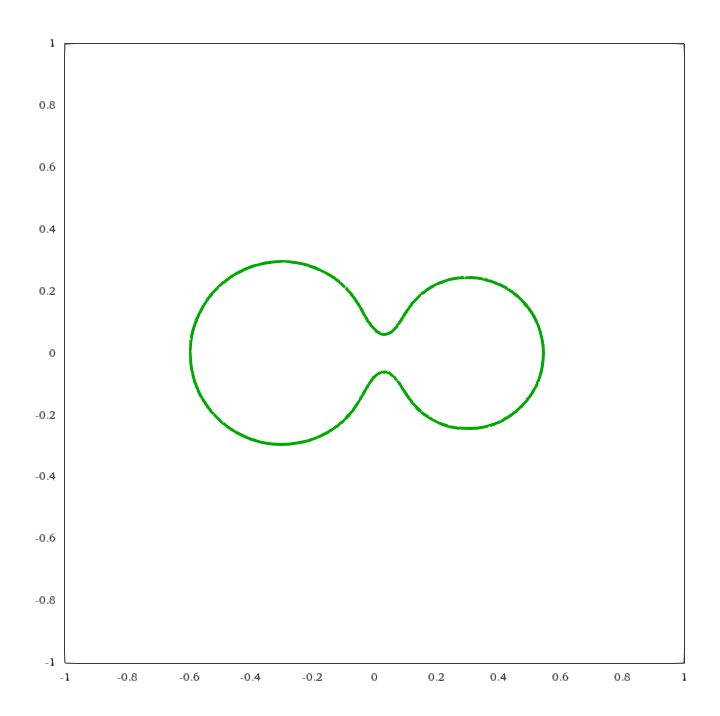}
    }
    \subfloat[$t=0.016$]{
        \includegraphics[width=0.275\textwidth]{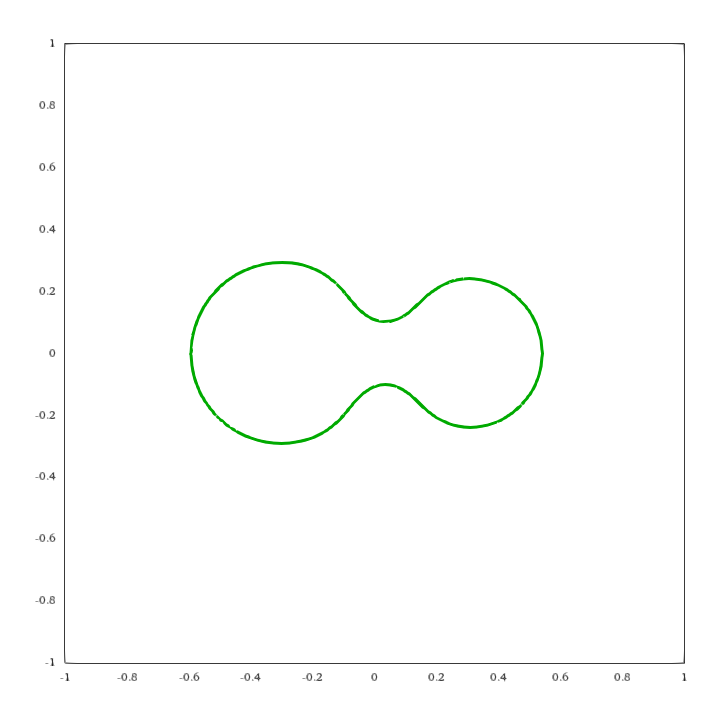}
    }
    \\
    \subfloat[$t=0.048$]{
        \includegraphics[width=0.275\textwidth]{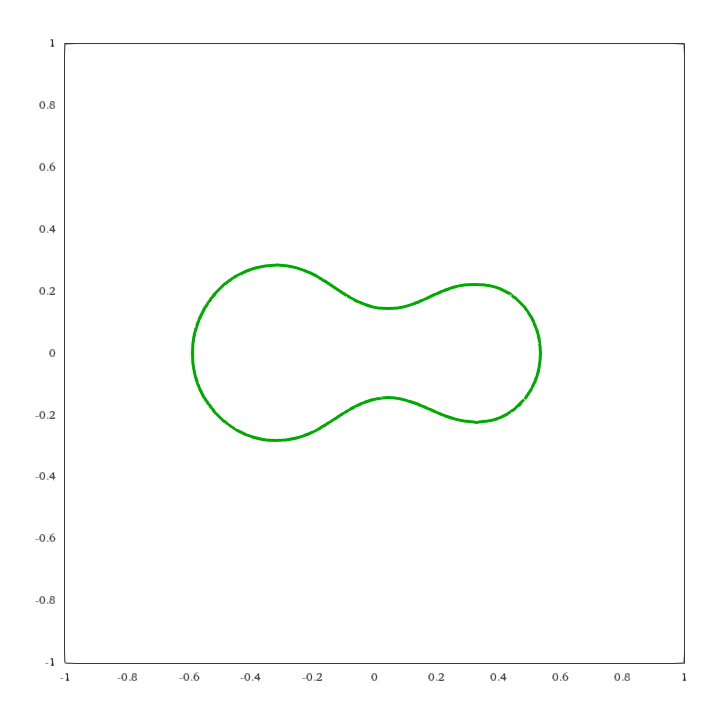}
    }
    \subfloat[$t=0.144$]{
        \includegraphics[width=0.275\textwidth]{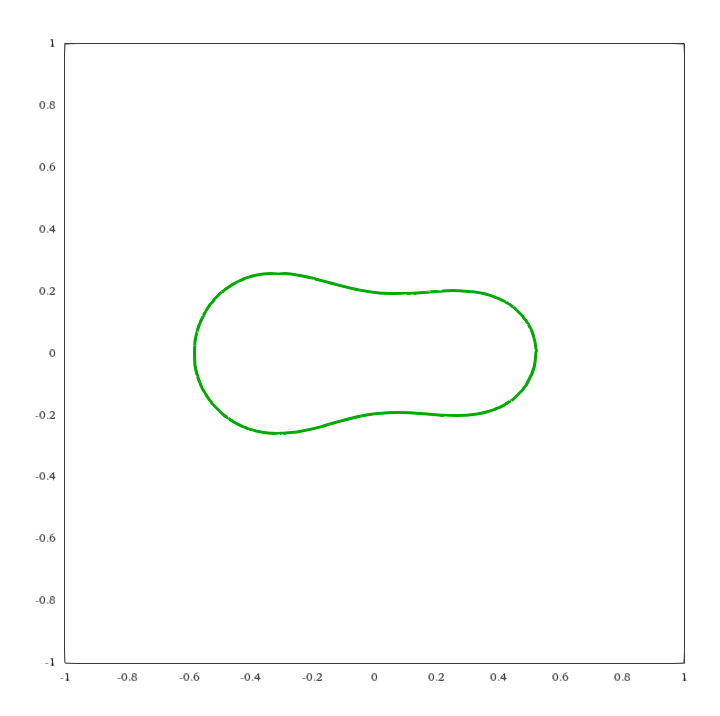}
    }
    \subfloat[$t=0.3$]{
        \includegraphics[width=0.275\textwidth]{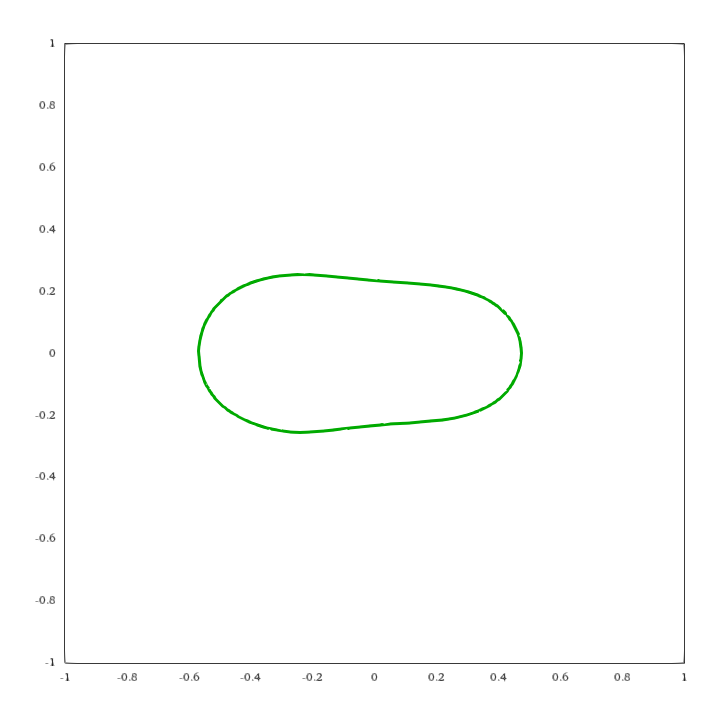}
    }
\end{center}
    \caption{
        Test 2: Two interacting bubbles. Screenshots of the zero-level sets at times $t=0,0.004,0.016,0.048,0.144,0.3$ with $\eps=3/100$ are displayed on the $25 \times 25$ Voronoi polygonal grid.
        }
\label{fig: zero level sets}
\end{figure}

\subsection{Test 2: Two interacting bubbles}\label{subsec: test 2}
\corrections{This test is taken from} \cite{feng2008posteriori,wu2020analysis} \corrections{and is dedicated to the evolution of two ellipses.}
A similar test is investigated in \cite{liu2020fully}.
\corrections{In particular, we take initial data as follows}
\begin{align*}
    u_0(x,y) = {\rm tanh} \left( ((x-0.3)^2 + y^2 - 0.25^2 ) / \eps \right){\rm tanh} \left( ((x+0.3)^2+y^2-0.3^2 )/ \eps\right) 
\end{align*}
where ${\rm tanh}(t) = (e^t - e^{-t})/(e^t+e^{-t})$. 
We monitor the evolution of two ellipses on the domain $\Omega=(-1,1)^2$ and fix $\polOrder=4$, $\eps=3/100$ \corrections{for this experiment}.
\corrections{We use the CSRK-2 \eqref{eqn: IMEX2 Butcher tables} time stepping method with $\tau=10^{-3}$. 
We show results on the $25 \times 25$ Voronoi mesh.
The screenshots of the numerical interface at 6 different fixed time frames is shown in Figure~\ref{fig: zero level sets}.
Here we see that the initial ellipses evolve and
coalesce over time to one ellipse. Note that the observed behaviour and evolution of
the interface is in line with the results displayed in \cite{wu2020analysis} for the Morley element.}

\FloatBarrier

% -----------------------------------------------------------------------------------
% -- Cross evolution
% -----------------------------------------------------------------------------------
\subsection{Test 3: Evolution of a cross}\label{sec: evo cross}
As considered in e.g. \cite{antonietti_$c^1$_2016,chave2016hybrid,liu2020fully}, for this experiment we monitor the evolution of initial data relating to a cross-shaped interface between phases. The initial data is described as follows.
\begin{align*}
    u_0(x,y) = \begin{cases}
        0.95 \quad &\text{ if }\quad  |(y-\half) - \frac{2}{5}(x-\half)| + | \frac{2}{5}(x-\half) + (y-\half) | < \frac{1}{5},
        \\
        0.95 \quad &\text{ if }\quad  |(x-\half) - \frac{2}{5}(y-\half)| + | \frac{2}{5}(y-\half) + (x-\half) | < \frac{1}{5},
        \\
        -0.95 \quad &\text{ otherwise.}
    \end{cases}
\end{align*}

We carry out this experiment on the unit square using \corrections{the same two mesh types as in Test 1 (as described in Section~\ref{subsec: test 1}) for three different grid sizes.
The grid data for these choices is detailed in Table~\ref{tab: grid size data} alongside the total number of dofs for each of these grid sizes.}
We also take multiple values for the interface parameter: ${\eps = 1/100, 1/50,}$ and $1/25$ and investigate the behaviour of the method for each \emph{fixed} value of epsilon. 
We also fix $\tau = 10^{-3}$, use the CSRK-2 time stepping method \eqref{eqn: IMEX2 Butcher tables} for these simulations and run the test to time $T=0.8$. 

\begin{table}[h]
        \caption{Combinations of $N \times N$ grids showing the size (total number of polygons), grid size $h$, and number of dofs for each of the grid choices.}
        \label{tab: grid size data}
        {\small
        \begin{tabular*}{\textwidth}{@{\extracolsep{\fill}}c|rcr|rcr}
        \toprule
        & \multicolumn{3}{c|}{structured simplex mesh} & \multicolumn{3}{c}{Voronoi mesh}
        \\
        $N \times N$ &  size & $h$ &  dofs &  size & $h$ &  dofs \\
        \midrule
        % $\polOrder=2$ &    50 &            5000 &        0.0283 &           10201 &            2500 &        0.0326 &           12503 \\
        % $\polOrder=2$ &   100 &           20000 &        0.0141 &           40401 &           10000 &        0.0169 &           50003 \\
        15 $\times$ 15 &             450 &        0.0943 &            4231 &             225 &        0.1008 &            4057 \\
        25 $\times$ 25 &            1250 &        0.0566 &           11551 &             625 &        0.0676 &           11257 \\
        45 $\times$ 45 &            4050 &        0.0314 &           36991 &            2025 &        0.0373 &           36457 \\
        \bottomrule
        \end{tabular*}}
\end{table}

\corrections{We only show the evolution of the method for this test on the $25 \times 25$ grids for two of the values of interface parameter $\eps=1/100$ (Figure~\ref{fig: evolutionScreenshots criss}) and $\eps=1/25$ (Figure~\ref{fig: evolutionScreenshots voronoi}) at three different fixed time frames.
We overlay the grids in the first images in Figures~\ref{fig: evolutionScreenshots criss} and \ref{fig: evolutionScreenshots voronoi}.}

\corrections{Figures~\ref{fig: end times for Test 2 screenshots simplex} and \ref{fig: end times for Test 2 screenshots voronoi} show the evolution at the end time frame ($T=0.8$) for $\eps=1/100$ and $\eps=1/25$, respectively for both the simplex criss mesh and Voronoi mesh.
We show the grids overlaid in the first figure, Figure~\ref{fig: end times for Test 2 screenshots simplex}.
Each figure contains the end time frame for all grid sizes from left to right $15 \times 15, 25 \times 25$, and $45 \times 45$.
We can see from Figures~\ref{fig: end times for Test 2 screenshots simplex} and \ref{fig: end times for Test 2 screenshots voronoi} that in all cases the initial data evolves to a circular interface even for the coarse $15 \times 15$ grid. 
Again, for the sake of brevity, we do not show the end evolution for $\eps=1/50$.}

\corrections{The energy decay for this problem is shown in Figure \ref{fig: energy plots}.
At each time step we compute the energy $E(u_h)$ \eqref{eqn: energy functional} of the discrete solution $u_h$.
As expected, the energy decreases in nearly all cases.
There is a slight increase in the left figure of Figure~\ref{fig: energy plots eps=1/100} which corresponds to the interface parameter $\eps=1/100$ on the coarse $15 \times 15$ criss grid, indicating that the interface is unresolved even with the higher order scheme on this grid.
We also see in Figures~\ref{fig: end times for Test 2 screenshots simplex} and \ref{fig: end times for Test 2 screenshots voronoi} a slightly larger circular interface on the coarsest grids.} 

% \FloatBarrier
% -----------------------------------------------------------------------------------
% -- Spinodal decomposition
% -----------------------------------------------------------------------------------

\subsection{Test 4: Spinodal decomposition}
For this experiment we turn our attention to the spinodal decomposition of a binary mixture. 
\corrections{As} in \cite{antonietti_$c^1$_2016,chave2016hybrid}, to model this phenomenon we choose the initial data $u_0$ to be a random perturbation between $-1$ and $1$ located in a circle of diameter $0.3$ in the centre of the domain and $0$ elsewhere. 
We take the interface parameter to be $\eps=1/100$, with time step $\tau = 10^{-2}$, and we use the CSRK-2 time stepping method \eqref{eqn: IMEX2 Butcher tables}. 
Snapshots of the results on both the structured simplex
grid and the Voronoi polygonal mesh are shown in Figure~\ref{fig: Test 3 results for p=4 on both grids}.
% We only present the results from this test for $\polOrder=4$ but note that the results for lower orders ($\polOrder=2,3$) \corrections{are} analogous.
Note that the random initial conditions used for the two grids are different and result in the difference in the end configurations seen in Figure \ref{fig: Test 3 results for p=4 on both grids}.

\begin{figure}[!ht]
    \begin{center}
        \subfloat[Results on the criss grid with $\eps=1/100$.]{
            \label{fig: evolution simplex gamma=1/100}
            \includegraphics[width=0.155\textwidth]{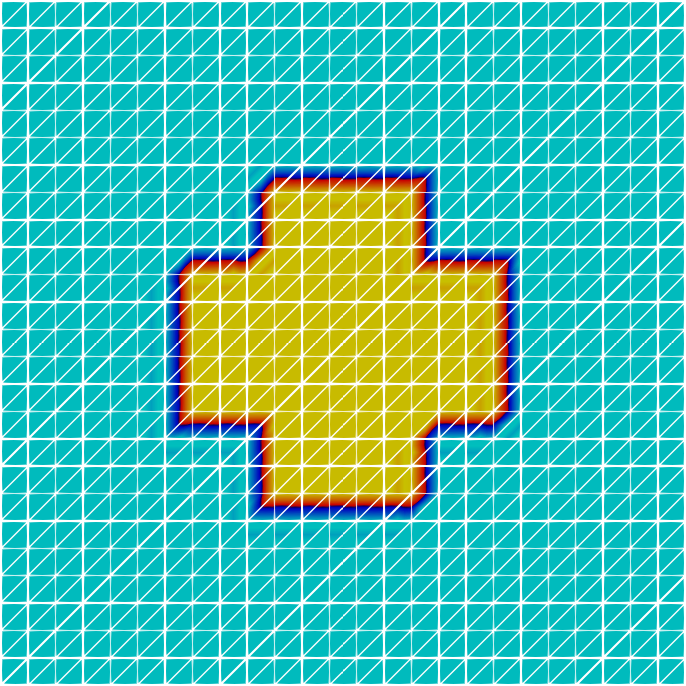}
            % \,
            \includegraphics[width=0.155\textwidth]{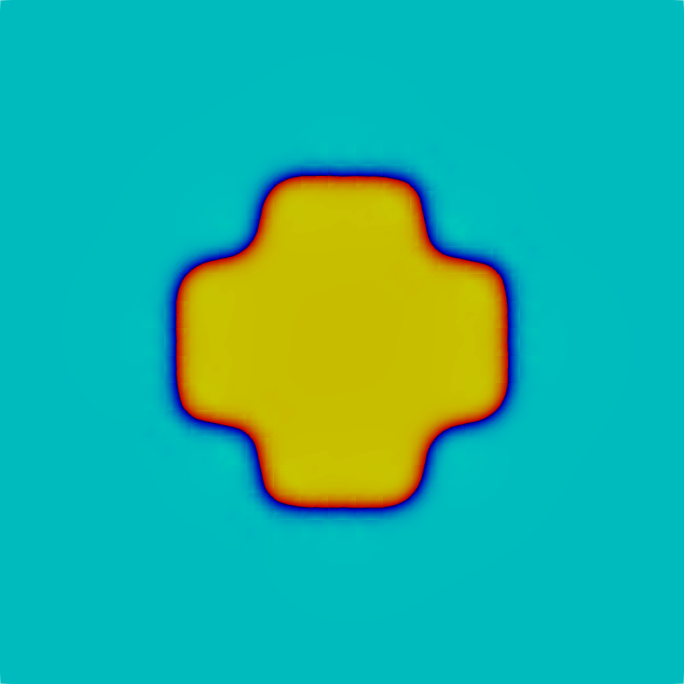}
            % \,
            \includegraphics[width=0.155\textwidth]{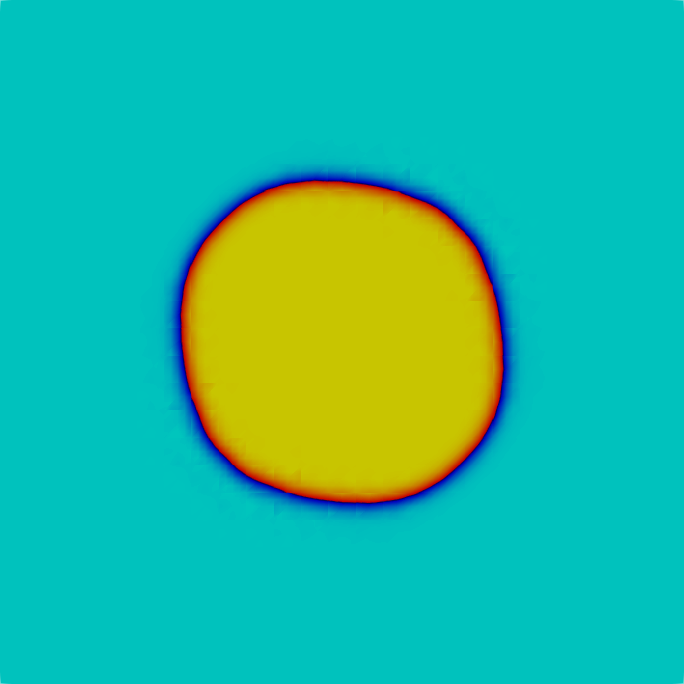}
        }
        \subfloat[Results on the Voronoi grid with $\eps=1/100$.]{
            \label{fig: evolution voronoi gamma=1/100}
            \includegraphics[width=0.155\textwidth]{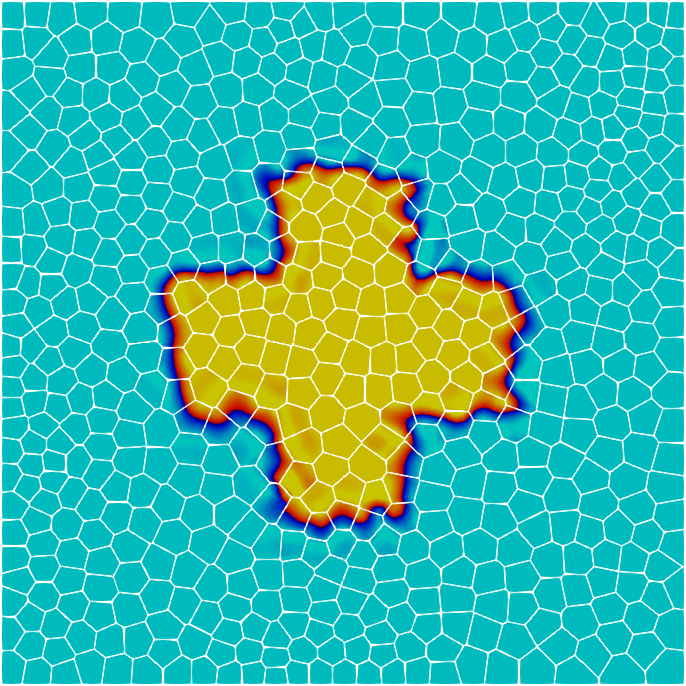}
            % \,
            \includegraphics[width=0.155\textwidth]{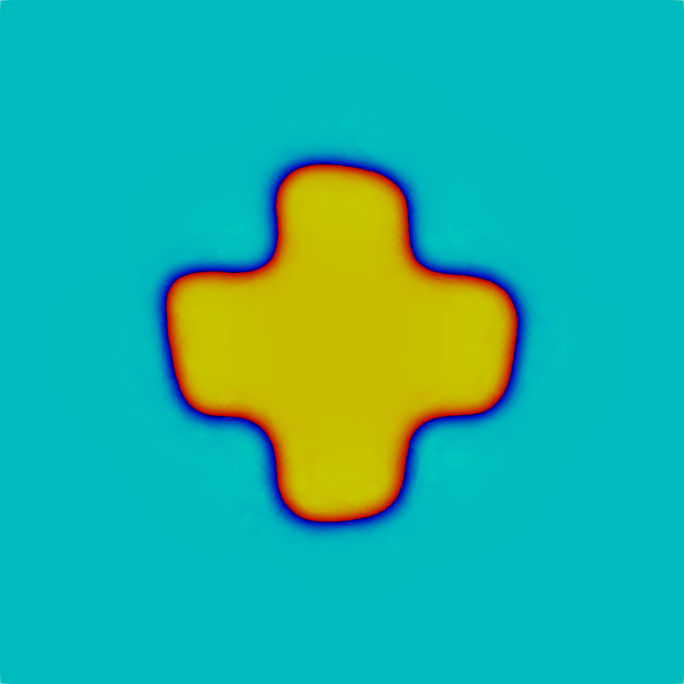}
            % \,
            \includegraphics[width=0.155\textwidth]{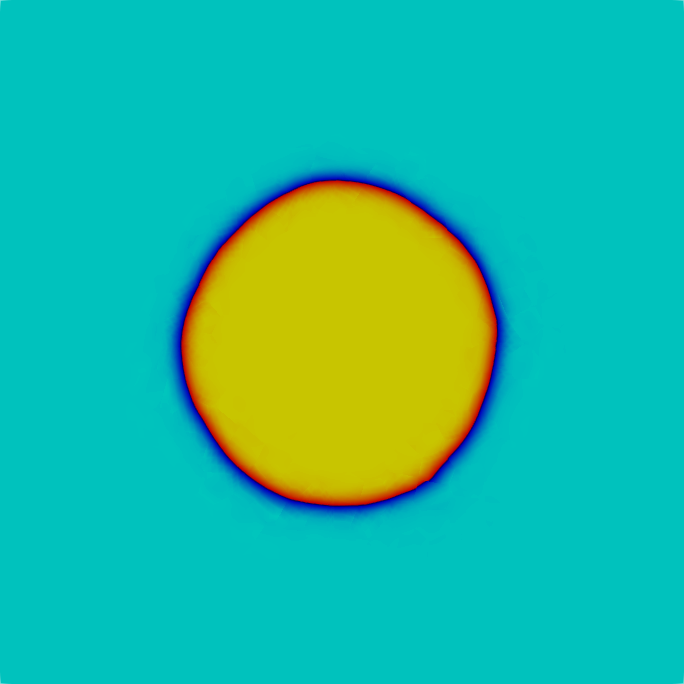}
        }
    \end{center}
    \caption{Test 3: evolution of a cross on the $25 \times 25$ grids displayed at three different time frames from left to right $(t=0,0.004,0.8)$ with $\eps=1/100$. 
    % The first two rows correspond to $\eps=1/100$ on the structured simplex mesh and the Voronoi polygonal mesh, respectively. The last two rows correspond to $\eps=1/25$ on the structured simplex mesh and the Voronoi polygonal mesh, respectively. 
    }
    \label{fig: evolutionScreenshots criss}
\end{figure}        

\begin{figure}[!ht]
    \begin{center}
        \subfloat[Results on the criss grid with $\eps=1/25$.]{
            \label{fig: evolution simplex gamma=1/25}
            \includegraphics[width=0.155\textwidth]{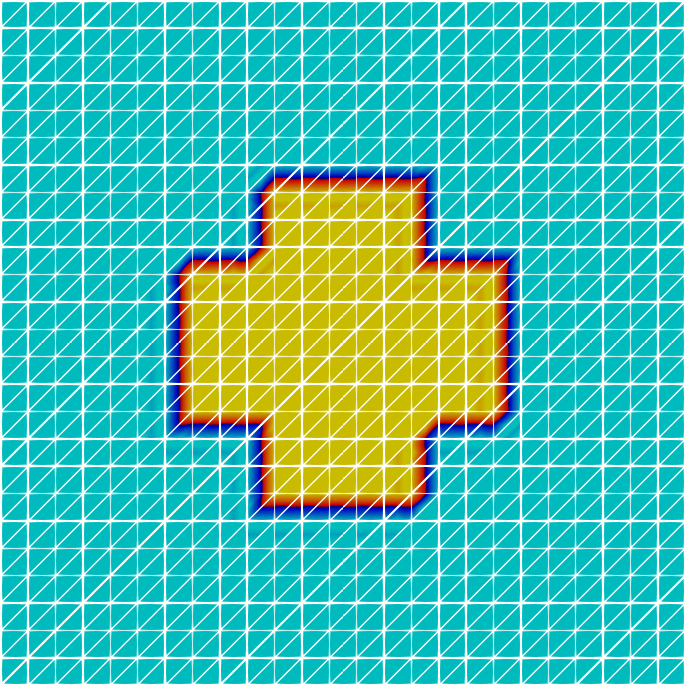}
            \includegraphics[width=0.155\textwidth]{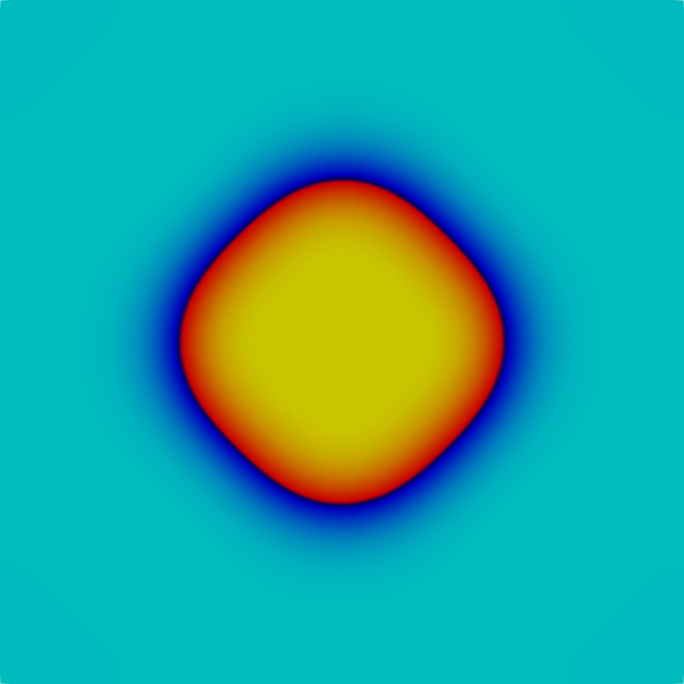}
            \includegraphics[width=0.155\textwidth]{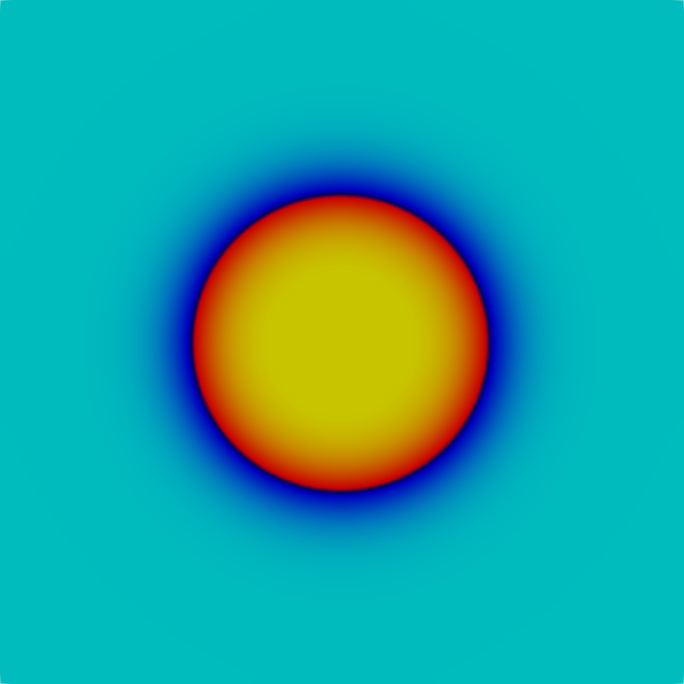}
        }
        \subfloat[Results on the Voronoi grid with $\eps=1/25$.]{
            \label{fig: evolution voronoi gamma=1/25}
            \includegraphics[width=0.155\textwidth]{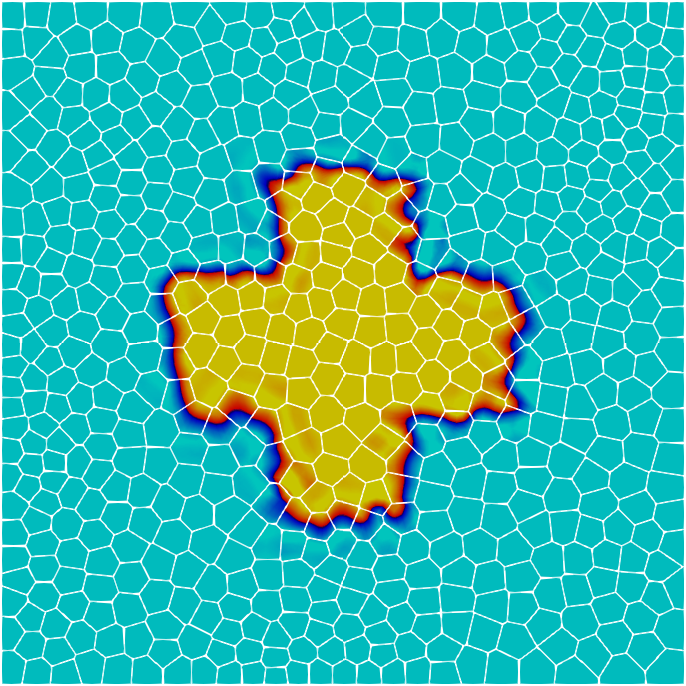}
            \includegraphics[width=0.155\textwidth]{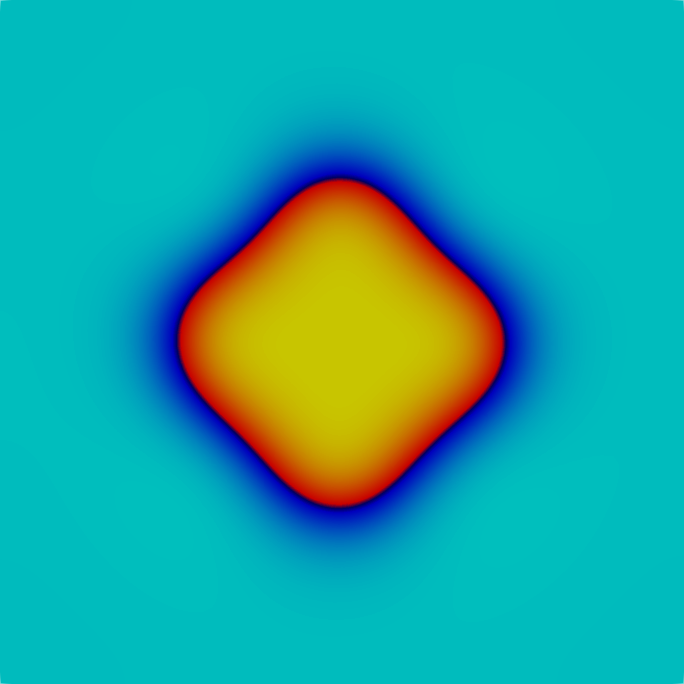}
            \includegraphics[width=0.155\textwidth]{Results/FiguresRevised/chTest_crossevo_SecondOrder1_voronoi_epsFac25_p4_N25_dt10_00199.png}
        }
    \end{center}    
    \caption{
        Test 3: evolution of a cross on the $25 \times 25$ grids displayed at three different time frames from left to right $(t=0,0.004,0.8)$ with $\eps=1/25$. 
        % The first two rows correspond to $\eps=1/100$ on the structured simplex mesh and the Voronoi polygonal mesh, respectively. The last two rows correspond to $\eps=1/25$ on the structured simplex mesh and the Voronoi polygonal mesh, respectively. 
        }
    \label{fig: evolutionScreenshots voronoi}
\end{figure}
\begin{figure}[!h]
    \begin{center}  
    \subfloat[Results on the criss grids with $\eps=1/100$. ]{
        \includegraphics[width=0.155\textwidth]{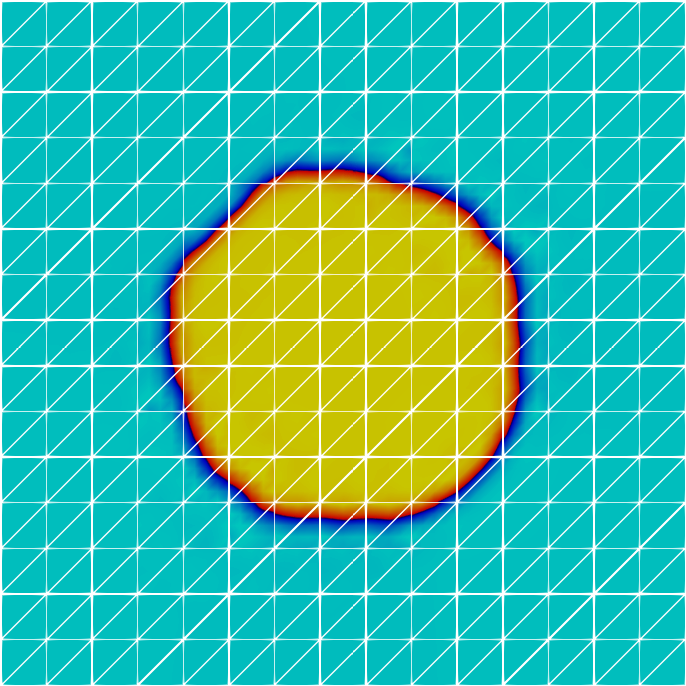}
        \includegraphics[width=0.155\textwidth]{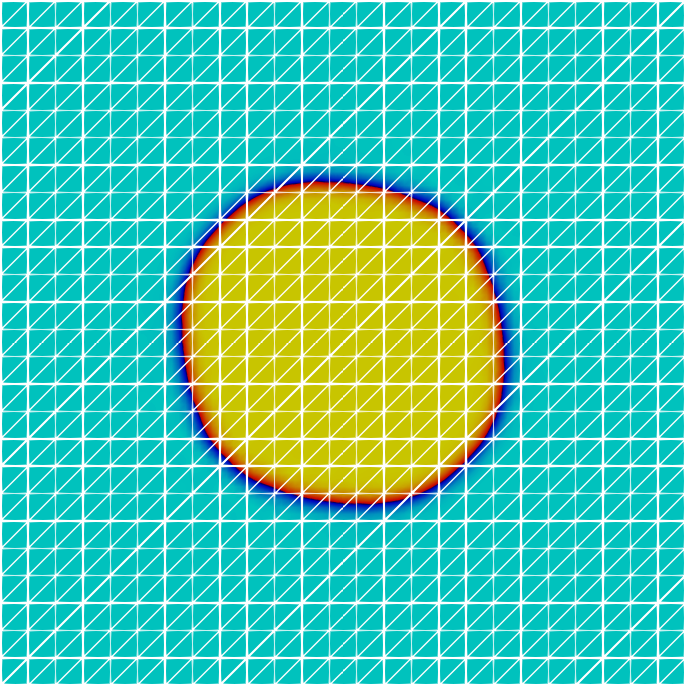}
        \includegraphics[width=0.155\textwidth]{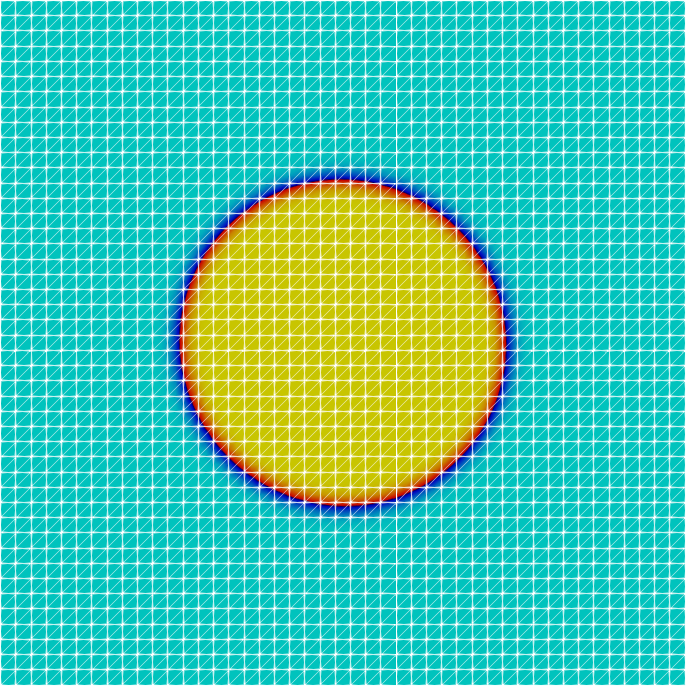}
    }
    \subfloat[Results on the Voronoi grids with $\eps=1/100$.]{
        \includegraphics[width=0.155\textwidth]{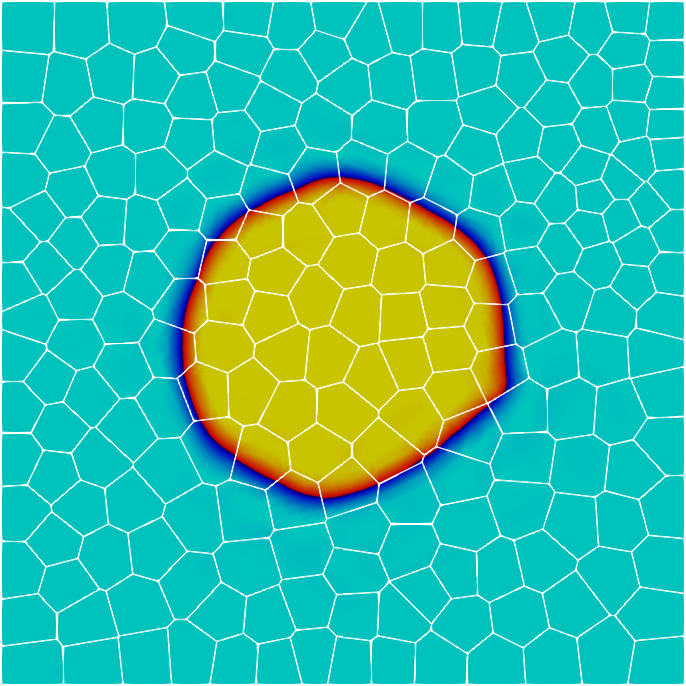}
        \includegraphics[width=0.155\textwidth]{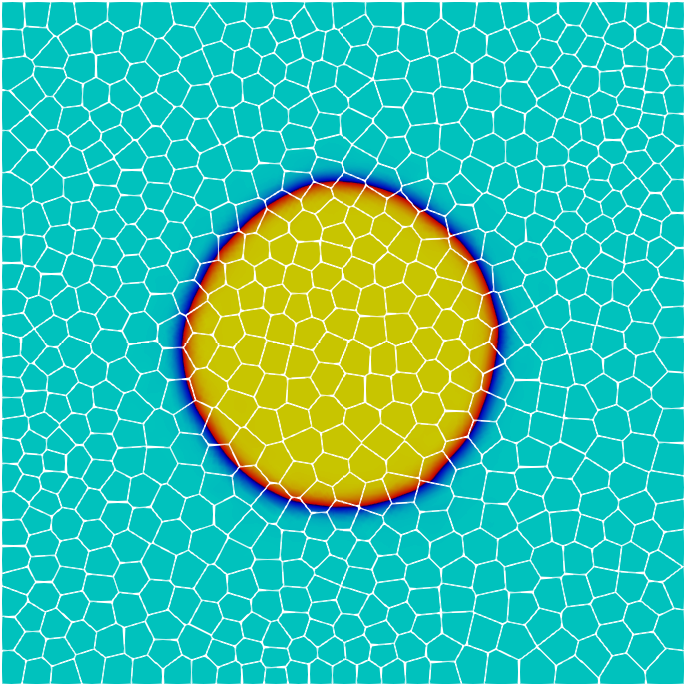}
        \includegraphics[width=0.155\textwidth]{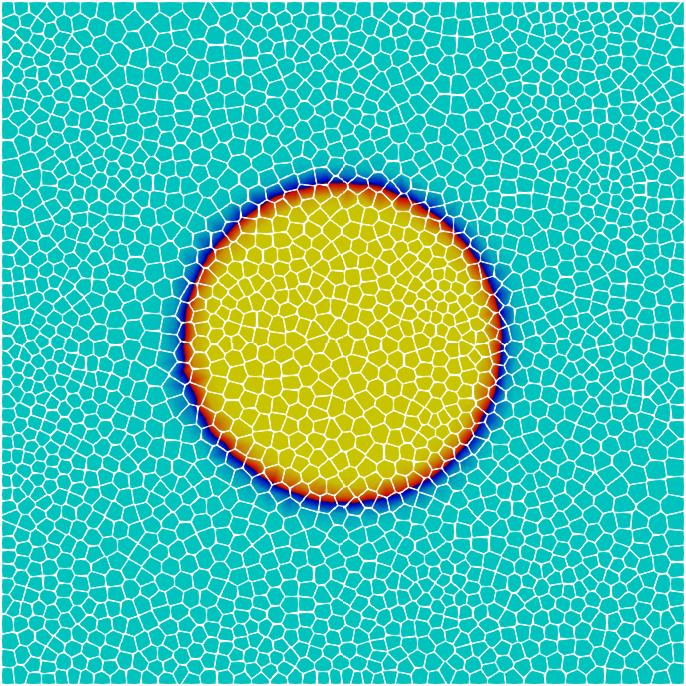}
    }
\end{center}
    \caption{
        Test 3: evolution of a cross displayed at the end time frame $t=0.8$ on the grid sizes from left to right $15 \times 15$, $25 \times 25$, and $45 \times 45$ with $\eps=1/100$.
        % Details of the number of elements and dofs in the meshes can be found in Table \ref{tab: grid size data}.
    }
    \label{fig: end times for Test 2 screenshots simplex}
\end{figure}
\begin{figure}[!h]
    \begin{center}
    \subfloat[Results on the criss grids with $\eps=1/25$.]{
        \includegraphics[width=0.155\textwidth]{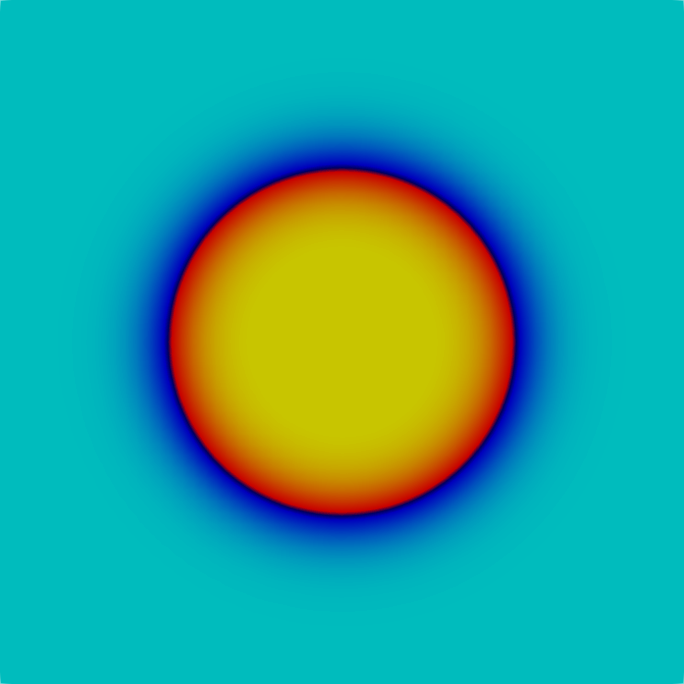}
        \includegraphics[width=0.155\textwidth]{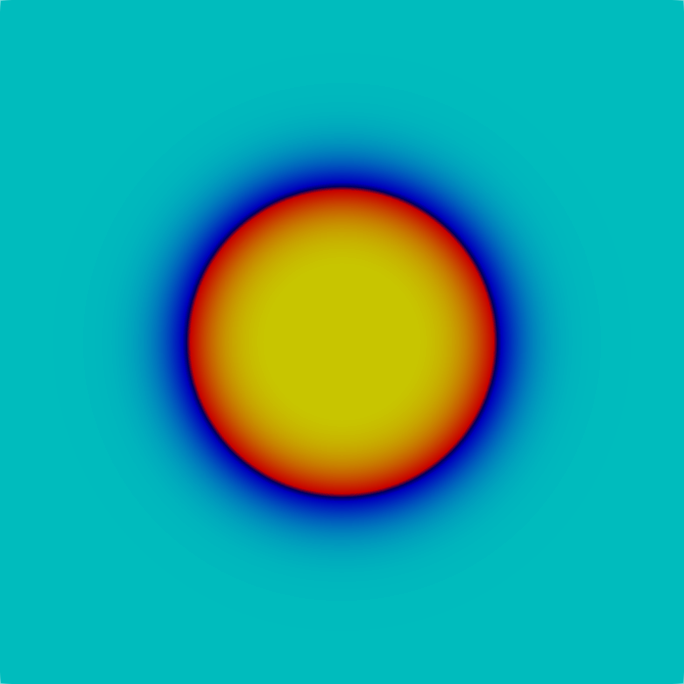}
        \includegraphics[width=0.155\textwidth]{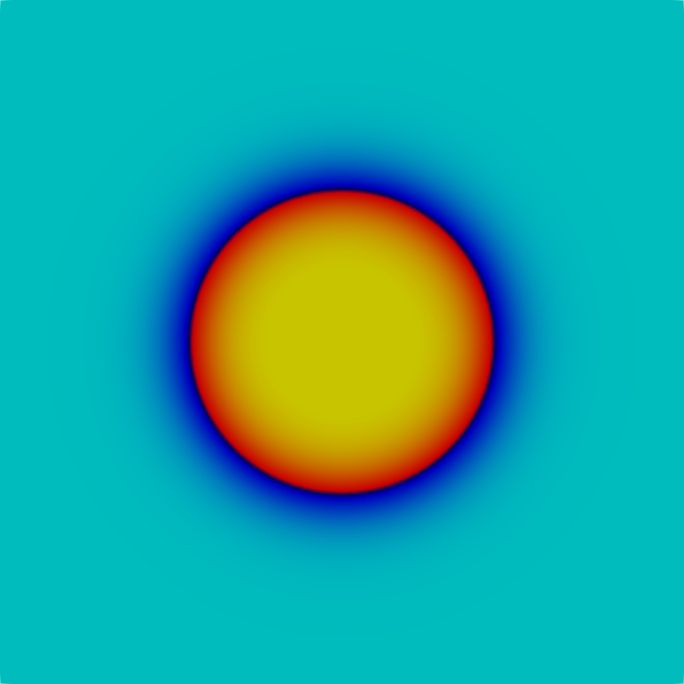}
    }
    \subfloat[Results on the Voronoi grids with $\eps=1/25$.]{
        \includegraphics[width=0.155\textwidth]{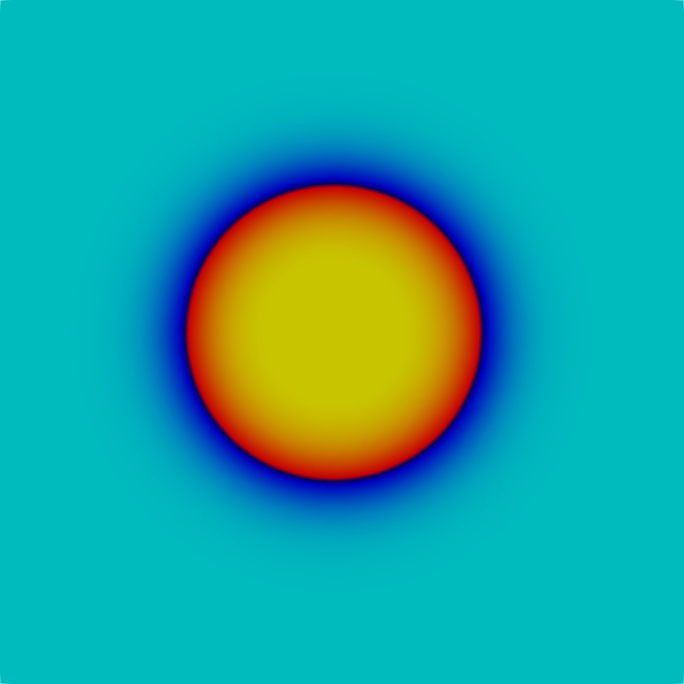}
        \includegraphics[width=0.155\textwidth]{Results/FiguresRevised/chTest_crossevo_SecondOrder1_voronoi_epsFac25_p4_N25_dt10_00199.png}
        \includegraphics[width=0.155\textwidth]{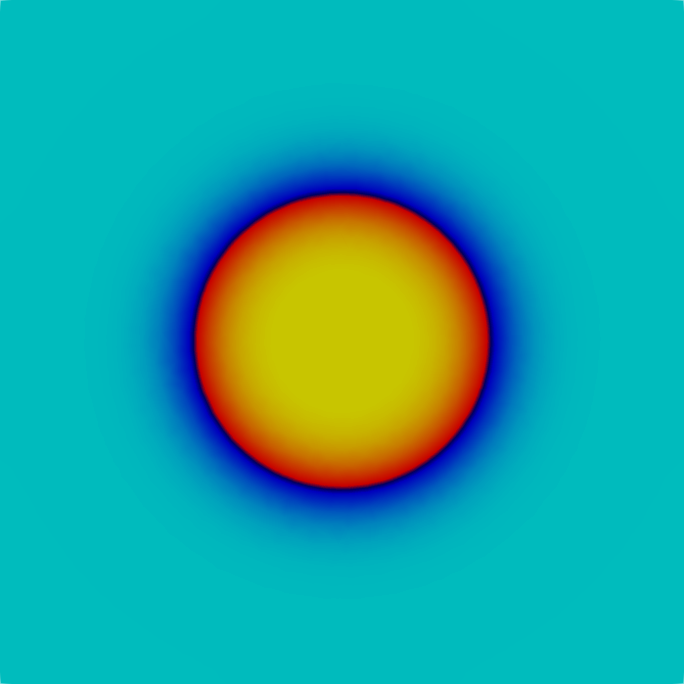}
    }
    \end{center}
    \caption{
        Test 3: evolution of a cross displayed at the end time frame $t=0.8$ on the grid sizes from left to right $15 \times 15$, $25 \times 25$, and $45 \times 45$ with $\eps=1/25$.
        \corrections{The grids used can be seen in Figure~\ref{fig: end times for Test 2 screenshots simplex}.}
        % Test 3: evolution of a cross displayed at the time frame $t=0.8$ on the three grid sizes from left to right $(15 \times 15)$, $(25 \times 25)$, and $(45 \times 45$). 
        % The first two rows correspond to $\eps=1/100$ on the simplex mesh and the Voronoi polygonal mesh, respectively. The last two rows correspond to $\eps=1/25$ on the simplex mesh and the Voronoi polygonal mesh, respectively.
        % Details of the number of elements and dofs in the meshes can be found in Table \ref{tab: grid size data}.
    }
    \label{fig: end times for Test 2 screenshots voronoi}
\end{figure}
\pgfplotscreateplotcyclelist{test}{
% teal,every mark/.append style={fill=teal!80!black},mark=otimes*\\
% orange,every mark/.append style={fill=orange!80!black},mark=diamond*\\
% red!70!white,densely dashed,mark=star\\
cyan!60!black,densely dashed,every mark/.append style={fill=cyan!80!black},mark=otimes*\\
lime!80!black,densely dashed,every mark/.append style={fill=lime},mark=diamond*\\
red,densely dashed,every mark/.append style={solid,fill=red!80!black},mark=star\\
yellow!60!black,densely dashed,
every mark/.append style={solid,fill=yellow!80!black},mark=square*\\
black,every mark/.append style={solid,fill=gray},mark=otimes*\\
blue,densely dashed,mark=star,every mark/.append style=solid\\
red,densely dashed,every mark/.append style={solid,fill=red!80!black},mark=diamond*\\
}
\begin{figure}[p]
  \begin{center}    
  \subfloat[Energy plots for $\eps=1/25$ on the structured simplex criss mesh (left) and Voronoi polygonal mesh (right).]{
    \begin{tikzpicture}
        \begin{axis}[
            width=0.45\textwidth, % Scale the plot to \linewidth
            grid=major, % Display a grid
            grid style={dashed,gray!20}, % Set the style
            xlabel=time $t$, % Set the labels
            ylabel=energy $E(u_h)$,
            legend style={nodes={scale=0.5, transform shape}},
            mark repeat=100,
            mark phase=0,
            mark size=1.5,
            cycle list name=test,
            label style={font=\footnotesize},
          ]
        %   \addplot+[]
        %   table[x=column 1,y=column 2,col sep=comma] {Results/CahnHilliard_results/Data/chTest_crossevo_SecondOrder1_criss_epsFac25_p2_N50_dt10_.csv}; 
        %   \addplot+[]
        %   table[x=column 1,y=column 2,col sep=comma] {Results/CahnHilliard_results/Data/chTest_crossevo_SecondOrder1_criss_epsFac25_p2_N100_dt10_.csv}; 
          \addplot+[]
          table[x=column 1,y=column 2,col sep=comma] {Results/Data/chTest_crossevo_SecondOrder1_criss_epsFac25_p4_N15_dt10_.csv}; 
          \addplot+[]
          table[x=column 1,y=column 2,col sep=comma] {Results/Data/chTest_crossevo_SecondOrder1_criss_epsFac25_p4_N25_dt10_.csv}; 
          \addplot+[]
          table[x=column 1,y=column 2,col sep=comma] {Results/Data/chTest_crossevo_SecondOrder1_criss_epsFac25_p4_N45_dt10_.csv}; 
          \legend{($\polOrder=4$, $N=15$),($\polOrder=4$, $N=25$),($\polOrder=4$, $N=45$)}
        \end{axis}
      \end{tikzpicture}
      \quad
      \begin{tikzpicture}
        \begin{axis}[
            width=0.45\textwidth, 
            grid=major, % Display a grid
            grid style={dashed,gray!20}, % Set the style
            xlabel=time $t$, % Set the labels
            ylabel=energy $E(u_h)$,
            legend style={nodes={scale=0.5, transform shape}},
            mark repeat=100,
            mark phase=0,
            mark size=1.5,
            cycle list name=test,
            label style={font=\footnotesize},
          ]
        %   \addplot+[]
        %   table[x=column 1,y=column 2,col sep=comma] {Results/CahnHilliard_results/Data/chTest_crossevo_SecondOrder1_voronoi_epsFac25_p2_N50_dt10_.csv}; 
        %   \addplot+[]
        %   table[x=column 1,y=column 2,col sep=comma] {Results/CahnHilliard_results/Data/chTest_crossevo_SecondOrder1_voronoi_epsFac25_p2_N100_dt10_.csv}; 
          \addplot+[]
          table[x=column 1,y=column 2,col sep=comma] {Results/Data/chTest_crossevo_SecondOrder1_voronoi_epsFac25_p4_N15_dt10_.csv}; 
          \addplot+[]
          table[x=column 1,y=column 2,col sep=comma] {Results/Data/chTest_crossevo_SecondOrder1_voronoi_epsFac25_p4_N25_dt10_.csv}; 
          \addplot+[]
          table[x=column 1,y=column 2,col sep=comma] {Results/Data/chTest_crossevo_SecondOrder1_voronoi_epsFac25_p4_N45_dt10_.csv}; 
          \legend{($\polOrder=4$, $N=15$),($\polOrder=4$, $N=25$),($\polOrder=4$, $N=45$)}
        \end{axis}
      \end{tikzpicture}}
    \begin{center}
          \subfloat[Energy plots for $\eps=1/50$ on the structured simplex criss mesh (left) and Voronoi polygonal mesh (right).]{
    \begin{tikzpicture}
      \begin{axis}[
          width=0.45\textwidth, % Scale the plot to \linewidth
          grid=major, % Display a grid
          grid style={dashed,gray!20}, % Set the style
          xlabel=time $t$, % Set the labels
          ylabel=energy $E(u_h)$,
          legend style={nodes={scale=0.5, transform shape}},
          mark repeat=100,
          mark phase=0,
          mark size=1.5,
          cycle list name=test,
          label style={font=\footnotesize},
        ]
        % \addplot+[]
        % table[x=column 1,y=column 2,col sep=comma] {Results/CahnHilliard_results/Data/chTest_crossevo_SecondOrder1_criss_epsFac50_p2_N50_dt10_.csv}; 
        % \addplot+[]
        % table[x=column 1,y=column 2,col sep=comma] {Results/CahnHilliard_results/Data/chTest_crossevo_SecondOrder1_criss_epsFac50_p2_N100_dt10_.csv}; 
        \addplot+[]
        table[x=column 1,y=column 2,col sep=comma] {Results/Data/chTest_crossevo_SecondOrder1_criss_epsFac50_p4_N15_dt10_.csv}; 
        \addplot+[]
        table[x=column 1,y=column 2,col sep=comma] {Results/Data/chTest_crossevo_SecondOrder1_criss_epsFac50_p4_N25_dt10_.csv}; 
        \addplot+[]
        table[x=column 1,y=column 2,col sep=comma] {Results/Data/chTest_crossevo_SecondOrder1_criss_epsFac50_p4_N45_dt10_.csv}; 
        \legend{($\polOrder=4$, $N=15$),($\polOrder=4$, $N=25$),($\polOrder=4$, $N=45$)}
      \end{axis}
    \end{tikzpicture}
    \quad
    \begin{tikzpicture}
      \begin{axis}[
          width=0.45\textwidth, 
          grid=major, % Display a grid
          grid style={dashed,gray!20}, % Set the style
          xlabel=time $t$, % Set the labels
          ylabel=energy $E(u_h)$,
          legend style={nodes={scale=0.5, transform shape}},
          mark repeat=100,
          mark phase=0,
          mark size=1.5,
          cycle list name=test,
          label style={font=\footnotesize},
        ]
        % \addplot+[]
        % table[x=column 1,y=column 2,col sep=comma] {Results/CahnHilliard_results/Data/chTest_crossevo_SecondOrder1_voronoi_epsFac50_p2_N50_dt10_.csv}; 
        % \addplot+[]
        % table[x=column 1,y=column 2,col sep=comma] {Results/CahnHilliard_results/Data/chTest_crossevo_SecondOrder1_voronoi_epsFac50_p2_N100_dt10_.csv}; 
        \addplot+[]
        table[x=column 1,y=column 2,col sep=comma] {Results/Data/chTest_crossevo_SecondOrder1_voronoi_epsFac50_p4_N15_dt10_.csv}; 
        \addplot+[]
        table[x=column 1,y=column 2,col sep=comma] {Results/Data/chTest_crossevo_SecondOrder1_voronoi_epsFac50_p4_N25_dt10_.csv}; 
        \addplot+[]
        table[x=column 1,y=column 2,col sep=comma] {Results/Data/chTest_crossevo_SecondOrder1_voronoi_epsFac50_p4_N45_dt10_.csv}; 
        \legend{($\polOrder=4$, $N=15$),($\polOrder=4$, $N=25$),($\polOrder=4$, $N=45$)}
      \end{axis}
    \end{tikzpicture}}
  \end{center}
    \subfloat[Energy plots for $\eps=1/100$ on the structured simplex criss mesh (left) and Voronoi polygonal mesh (right).]{\label{fig: energy plots eps=1/100}
    \begin{tikzpicture}
      \begin{axis}[
          width=0.45\textwidth, % Scale the plot to \linewidth
          grid=major, % Display a grid
          grid style={dashed,gray!20}, % Set the style
          xlabel=time $t$, % Set the labels
          ylabel=energy $E(u_h)$,
          legend style={nodes={scale=0.5, transform shape}},
          mark repeat=100,
          mark phase=0,
          mark size=1.5,
          cycle list name=test,
          label style={font=\footnotesize},
        ]
        % \addplot+[]
        % table[x=column 1,y=column 2,col sep=comma] {Results/CahnHilliard_results/Data/chTest_crossevo_SecondOrder1_criss_epsFac100_p2_N50_dt10_.csv}; 
        % \addplot+[]
        % table[x=column 1,y=column 2,col sep=comma] {Results/CahnHilliard_results/Data/chTest_crossevo_SecondOrder1_criss_epsFac100_p2_N100_dt10_.csv}; 
        \addplot+[]
        table[x=column 1,y=column 2,col sep=comma] {Results/Data/chTest_crossevo_SecondOrder1_criss_epsFac100_p4_N15_dt10_.csv}; 
        \addplot+[]
        table[x=column 1,y=column 2,col sep=comma] {Results/Data/chTest_crossevo_SecondOrder1_criss_epsFac100_p4_N25_dt10_.csv}; 
        \addplot+[]
        table[x=column 1,y=column 2,col sep=comma] {Results/Data/chTest_crossevo_SecondOrder1_criss_epsFac100_p4_N45_dt10_.csv}; 
        \legend{($\polOrder=4$, $N=15$),($\polOrder=4$, $N=25$),($\polOrder=4$, $N=45$)}
      \end{axis}
    \end{tikzpicture}
    \quad
    \begin{tikzpicture}
      \begin{axis}[
          width=0.45\textwidth, 
          grid=major, % Display a grid
          grid style={dashed,gray!20}, % Set the style
          xlabel=time $t$, % Set the labels
          ylabel=energy $E(u_h)$,
          legend style={nodes={scale=0.5, transform shape}},
          mark repeat=100,
          mark phase=0,
          mark size=1.5,
          cycle list name=test,
          label style={font=\footnotesize},
        ]
        % \addplot+[]
        % table[x=column 1,y=column 2,col sep=comma] {Results/CahnHilliard_results/Data/chTest_crossevo_SecondOrder1_voronoi_epsFac100_p2_N50_dt10_.csv}; 
        % \addplot+[]
        % table[x=column 1,y=column 2,col sep=comma] {Results/CahnHilliard_results/Data/chTest_crossevo_SecondOrder1_voronoi_epsFac100_p2_N100_dt10_.csv}; 
        \addplot+[]
        table[x=column 1,y=column 2,col sep=comma] {Results/Data/chTest_crossevo_SecondOrder1_voronoi_epsFac100_p4_N15_dt10_.csv}; 
        \addplot+[]
        table[x=column 1,y=column 2,col sep=comma] {Results/Data/chTest_crossevo_SecondOrder1_voronoi_epsFac100_p4_N25_dt10_.csv}; 
        \addplot+[]
        table[x=column 1,y=column 2,col sep=comma] {Results/Data/chTest_crossevo_SecondOrder1_voronoi_epsFac100_p4_N45_dt10_.csv}; 
        \legend{($\polOrder=4$, $N=15$),($\polOrder=4$, $N=25$),($\polOrder=4$, $N=45$)}
      \end{axis}
    \end{tikzpicture}}
  \end{center}
\caption{Test 3: energy decay plots (energy $E(u_h)$ vs time $t$) for the cross evolution problem.}
\label{fig: energy plots}
\end{figure} 
\FloatBarrier
\begin{figure}[!ht]
    \begin{center}
        \subfloat[Structured simplex mesh consisting of $2048$ elements ($18817$ dof).]{\label{fig: test 3 on criss grid}
        \includegraphics[width=0.2\textwidth]{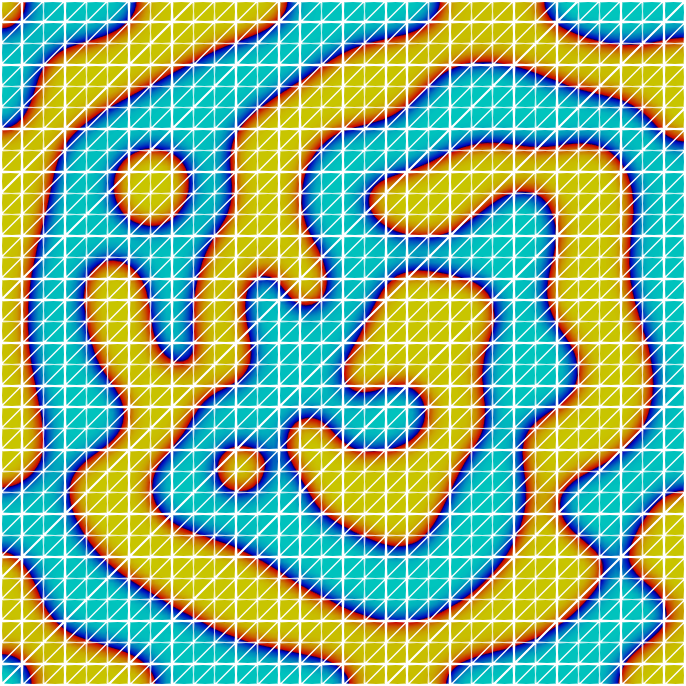}
        \,
        \includegraphics[width=0.2\textwidth]{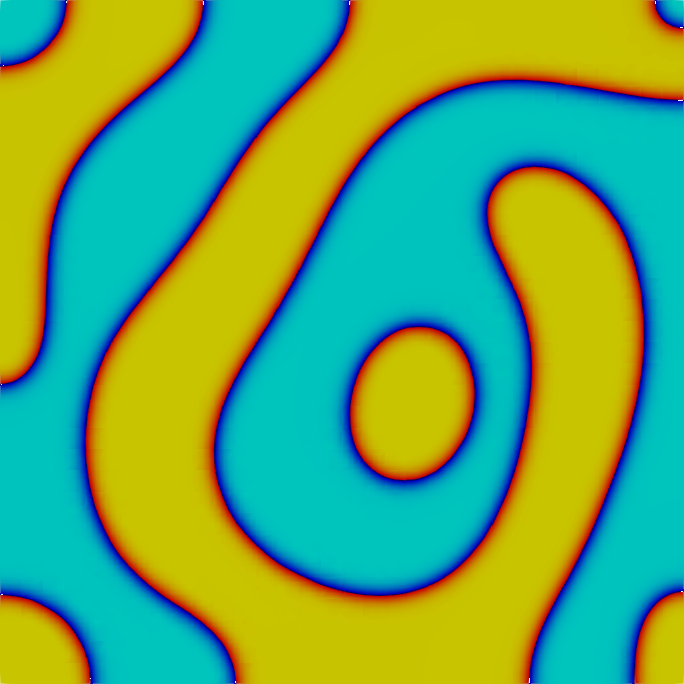}
        \,
        \includegraphics[width=0.2\textwidth]{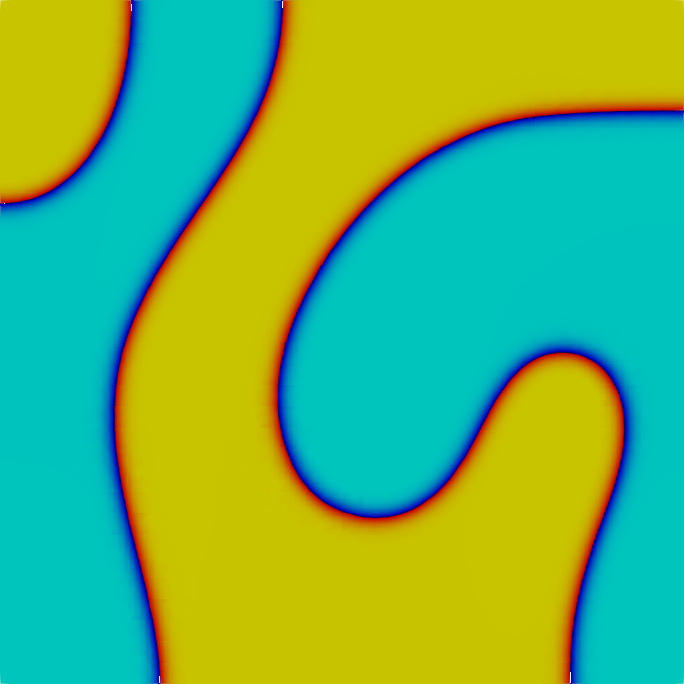}
        \,
        \includegraphics[width=0.2\textwidth]{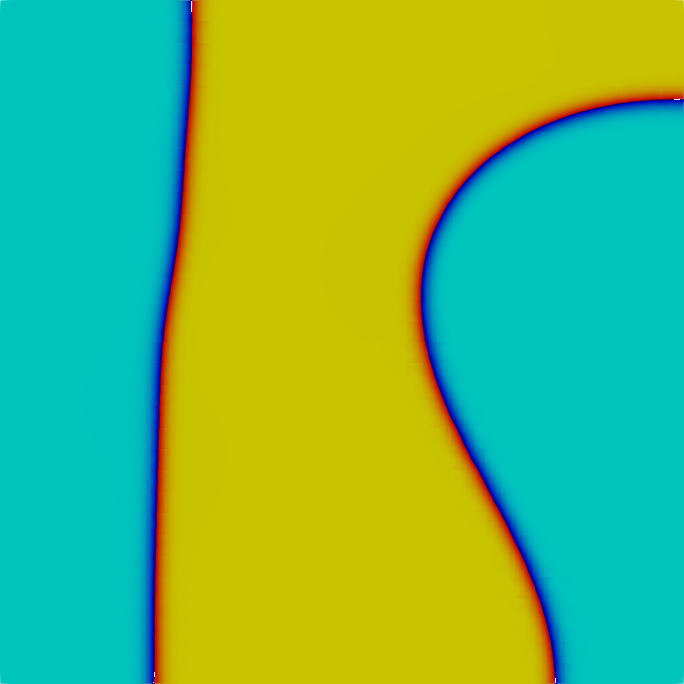}}
    \end{center}
    \begin{center}
        \subfloat[Voronoi polygonal mesh consisting of $1024$ elements ($18439$ dof).]{\label{fig: test 3 on voronoi}
        \includegraphics[width=0.2\textwidth]{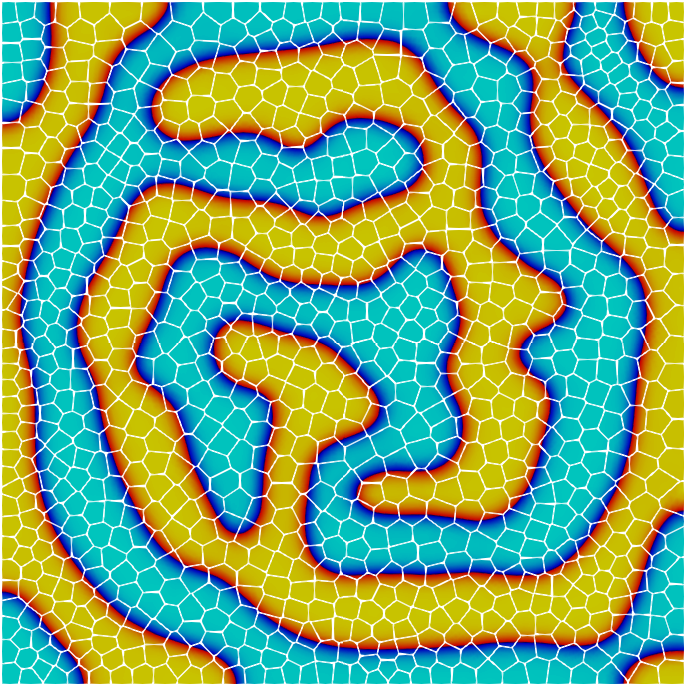}
        \,
        \includegraphics[width=0.2\textwidth]{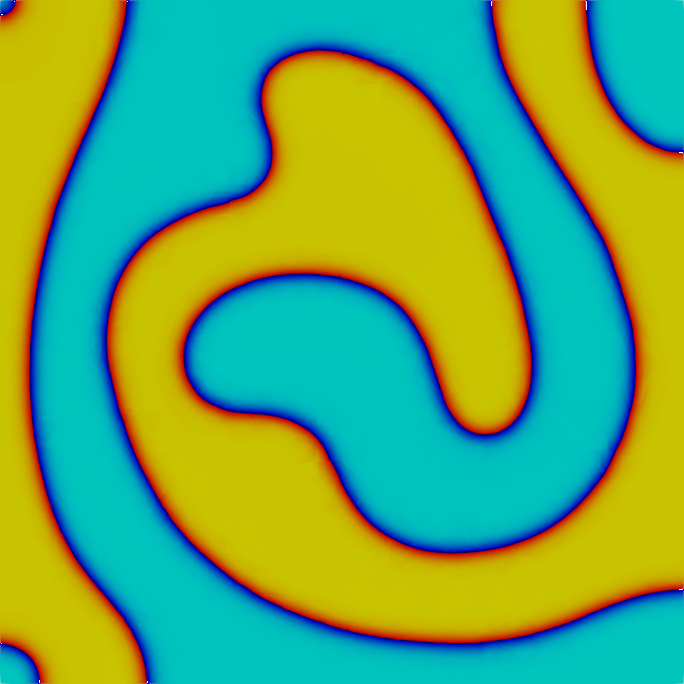}
        \,
        \includegraphics[width=0.2\textwidth]{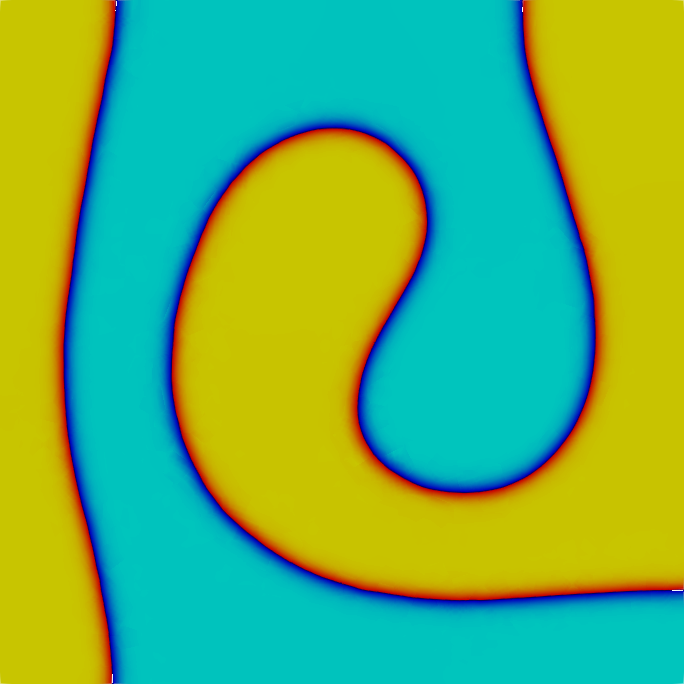}
        \,
        \includegraphics[width=0.2\textwidth]{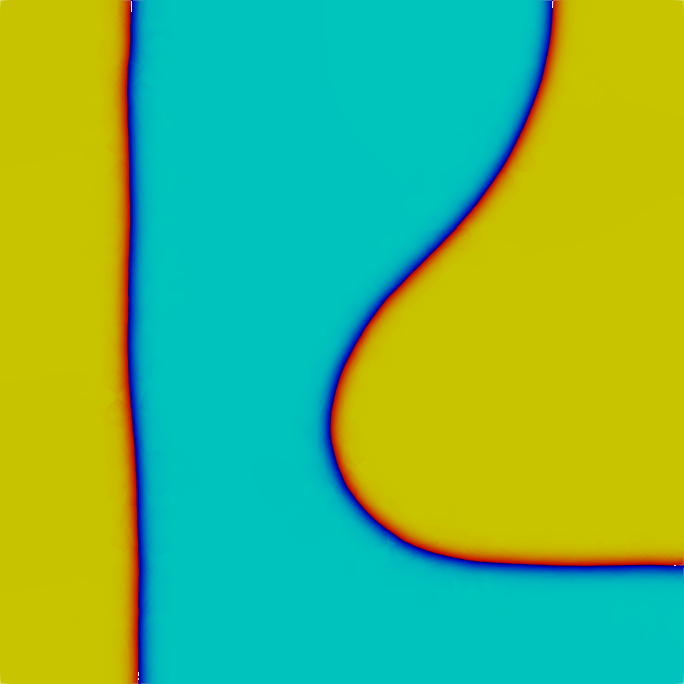}}
    \end{center}
\caption{Test 4: spinodal decomposition with order $\polOrder=4$ displayed at four time frames from left to right $(t=0.04,0.4,1.6,5).$ 
% Notice that the visible differences in the images is due to the difference in the starting configurations from different initial random conditions.
}
\label{fig: Test 3 results for p=4 on both grids}
\end{figure}

\FloatBarrier

\subsection{Comparison with conforming VEM}
\jsccorrections{In this section we briefly investigate numerically the behaviour of the $C^1$-conforming virtual element method, see e.g. \cite{antonietti_conforming_2018,brezzi_virtual_2013}.
A full set of numerical experiments for the lowest order $C^1$ method can be found in \cite{antonietti_$c^1$_2016}.
However, note that the treatment of the nonlinearity in \cite{antonietti_$c^1$_2016} differs from the approach taken in this paper.
Whilst the $C^1$ VEM has the advantage of being conforming in $H^2$, the implementation of the nonconforming VEM presented in this paper is simpler.
For example, the treatment of boundary conditions is more complex for the conforming space due to the tangential derivative degrees of freedom at the vertices.
The lowest order $C^1$ VEM ($\polOrder=2$) has only 3 degrees of freedom per vertex so on structured triangular grids (such as the ``criss'' grid considered in this section) it will generally have fewer dofs than the nonconforming space; however, this is not the case on other polygonal grids such as either structured quadrilateral grids or Voronoi grids.
Additionally, due to only having 9 dofs on triangles, extra care is required when implementing the value projection for the $C^1$-conforming VEM.
%  as this space does not fit straightforwardly into the framework introduced in \cite{dedner2022framework}. 
Further discussion on this as well as implementation details for both spaces can be found in \cite{dedner2022framework}. 
}

\jsccorrections{
    In this work, we have focused on developing higher order extensions of Morley which is often the only space available in finite element packages for solving fourth-order problems.
    Consequently, the tools and infrastructure necessary to implement the higher order nonconforming VEMs in this paper is likely to be more readily available.
    Importantly, for all grid resolutions considered in the tests so far, the advantage of using the higher order method is apparent.
    On top of this, we see in the subsequent comparison that both the higher order conforming and nonconforming VEM produce identical results.}

\jsccorrections{
We now turn our attention to repeating some of the tests already presented but with the $C^1$-conforming VEM.
Firstly, we repeat the non-physical test, Test 1 from section~\ref{subsec: test 1}.
We select the exact same parameters as in section~\ref{subsec: test 1} and the results on the structured simplex ``criss'' grid with the CSRK-2 \eqref{eqn: IMEX2 Butcher tables} time stepping method are shown in Figure~\ref{fig: test1 plots} for polynomial orders $\polOrder=2,3,4$.
We plot the $H^2$-error in the left figure for both the $C^1$-conforming and the nonconforming spaces and the order of convergence is shown in the right figure.
Since we use a second order time stepping method, the eocs for $\polOrder=3,4$ are the same.
For the lowest order spaces ($\polOrder=2$), we see that the $C^1$-conforming VEM outperforms the nonconforming space however, the results for $\polOrder=3,4$ are indistinguishable.
We point out that the same is true when looking at grid size $h$ vs error as well as dofs vs error.}

\jsccorrections{
Secondly, we investigate Test 2 with the $C^1$-conforming space. 
We use the exact same parameters as detailed in section~\ref{subsec: test 2} and the results are shown in Figure~\ref{fig: zero level sets C1conf version}.
As for the nonconforming space, the initial ellipses evolve and coalesce over time to one ellipse.
Despite the lowest order $C^1$-conforming VEM producing better errors for the non-physical test of Test 1, we point out that the results from this test are essentially identical to those shown in Figure~\ref{fig: zero level sets}.}

\pgfplotscreateplotcyclelist{test1}{
% teal,every mark/.append style={fill=teal!80!black},mark=otimes*\\
% orange,every mark/.append style={fill=orange!80!black},mark=diamond*\\
% red!70!white,densely dashed,mark=star\\
orange,densely dashed,every mark/.append style={solid,fill=orange},mark=*\\
blue,densely dashed,every mark/.append style={solid,fill=blue},mark=*\\
red,densely dashed,every mark/.append style={solid,fill=red},mark=*\\
orange,solid,mark=star,every mark/.append style=solid\\
blue,solid,mark=star,every mark/.append style=solid\\
red,solid,mark=star,every mark/.append style=solid\\
}
\pgfplotscreateplotcyclelist{test2}{
% teal,every mark/.append style={fill=teal!80!black},mark=otimes*\\
% orange,every mark/.append style={fill=orange!80!black},mark=diamond*\\
% red!70!white,densely dashed,mark=star\\
orange,densely dashed,every mark/.append style={solid,fill=orange},mark=*\\
orange,solid,mark=star,every mark/.append style=solid\\
blue,densely dashed,every mark/.append style={solid,fill=blue},mark=*\\
blue,solid,mark=star,every mark/.append style=solid\\
red,densely dashed,every mark/.append style={solid,fill=red},mark=*\\
red,solid,mark=star,every mark/.append style=solid\\
}
\begin{figure}[h!]
  \begin{center}    
%   \subfloat[]
%   {
    \begin{tikzpicture}
        \begin{loglogaxis}[
            width=0.45\linewidth,
            % Scale the plot to \linewidth
            grid=major, % Display a grid
            grid style={dashed,gray!20}, % Set the style
            xlabel=dofs, % Set the labels
            ylabel=$H^2$-error,
            legend columns=2,
            legend style={/tikz/column 2/.style={
                column sep=5pt,
            },at={(0.5,1.05)},anchor=south,nodes={scale=0.75}},
            % legend pos=south west,
            % legend style={nodes={scale=0.5, transform shape}},
            % mark repeat=100,
            % mark phase=0,
            mark size=1.5,
            cycle list name=test2,
            label style={font=\footnotesize},
          ]
          \addplot+[]
          table[x=dofs,y=h2err_max,col sep=comma] {Results/C1conf/results1/ch_convergencetemporal1_SecondOrder1_criss_conf_p2_16.csv};
          \addplot+[]
          table[x=dofs,y=h2err_max,col sep=comma] {Results/C1conf/results1/ch_convergencetemporal1_SecondOrder1_criss_nconf_p2_16.csv};
          \addplot+[]
          table[x=dofs,y=h2err_max,col sep=comma] {Results/C1conf/results1/ch_convergencetemporal1_SecondOrder1_criss_conf_p3_16.csv}; 
          \addplot+[]
          table[x=dofs,y=h2err_max,col sep=comma] {Results/C1conf/results1/ch_convergencetemporal1_SecondOrder1_criss_nconf_p3_16.csv}; 
          \addplot+[]
          table[x=dofs,y=h2err_max,col sep=comma] {Results/C1conf/results1/ch_convergencetemporal1_SecondOrder1_criss_conf_p4_16.csv};  
          \addplot+[]
          table[x=dofs,y=h2err_max,col sep=comma] {Results/C1conf/results1/ch_convergencetemporal1_SecondOrder1_criss_nconf_p4_16.csv};  
          \legend{C1-c: $\polOrder=2$,C1-nc: $\polOrder=2$,C1-c: $\polOrder=3$,C1-nc: $\polOrder=3$,C1-c: $\polOrder=4$,C1-nc: $\polOrder=4$}
        \end{loglogaxis}
      \end{tikzpicture}
      \quad
      \begin{tikzpicture}
        \begin{semilogxaxis}[
            width=0.45\linewidth, 
            grid=major, % Display a grid
            grid style={dashed,gray!20}, % Set the style
            ylabel=$H^2$-eoc, % Set the labels
            xmin=200,
            ymin=0,
            ymax=3,
            xlabel=dofs,
            % legend style={at={(0.5,-0.1)},anchor=north}
            % legend pos=south west,
            legend columns=2,
            legend style={/tikz/column 2/.style={
                column sep=5pt,
            },at={(0.5,1.05)},anchor=south,nodes={scale=0.75}},
            % mark repeat=100,
            % mark phase=0,
            mark size=1.5,
            cycle list name=test2,
            label style={font=\footnotesize},
          ]
          \addplot+[]
          table[x=dofs,y=h2eoc,col sep=comma] {Results/C1conf/results1/ch_convergencetemporal1_SecondOrder1_criss_conf_p2_16.csv};
          \addplot+[]
          table[x=dofs,y=h2eoc,col sep=comma] {Results/C1conf/results1/ch_convergencetemporal1_SecondOrder1_criss_nconf_p2_16.csv};
          \addplot+[]
          table[x=dofs,y=h2eoc,col sep=comma] {Results/C1conf/results1/ch_convergencetemporal1_SecondOrder1_criss_conf_p3_16.csv}; 
          \addplot+[]
          table[x=dofs,y=h2eoc,col sep=comma] {Results/C1conf/results1/ch_convergencetemporal1_SecondOrder1_criss_nconf_p3_16.csv}; 
          \addplot+[]
          table[x=dofs,y=h2eoc,col sep=comma] {Results/C1conf/results1/ch_convergencetemporal1_SecondOrder1_criss_conf_p4_16.csv};  
          \addplot+[]
          table[x=dofs,y=h2eoc,col sep=comma] {Results/C1conf/results1/ch_convergencetemporal1_SecondOrder1_criss_nconf_p4_16.csv};  
          \legend{C1-c: $\polOrder=2$,C1-nc: $\polOrder=2$,C1-c: $\polOrder=3$,C1-nc: $\polOrder=3$,C1-c: $\polOrder=4$,C1-nc: $\polOrder=4$}
        \end{semilogxaxis}
      \end{tikzpicture}
    \end{center}
%   }
\caption{\jsccorrections{Test 1 (section~\ref{subsec: test 1}) results for both the $C^1$-conforming (C1-c) and nonconforming (C1-nc) VEM spaces. We show the $H^2$-error (left) for polynomial orders $\polOrder=2,3,4$ and the convergence rates (right) for these errors on a structured simplex grid with the CSRK-2 time stepping method.
Note that the $\polOrder=4$ results start on a coarser grid level.}}
\label{fig: test1 plots}
\end{figure} 
\ifthenelse{\boolean{arxiv}}
{
    \begin{figure}
        \begin{center}        
        \subfloat[$t=0$]{
            \includegraphics[trim={2cm 2cm 2cm 2cm},clip,width=0.275\textwidth]{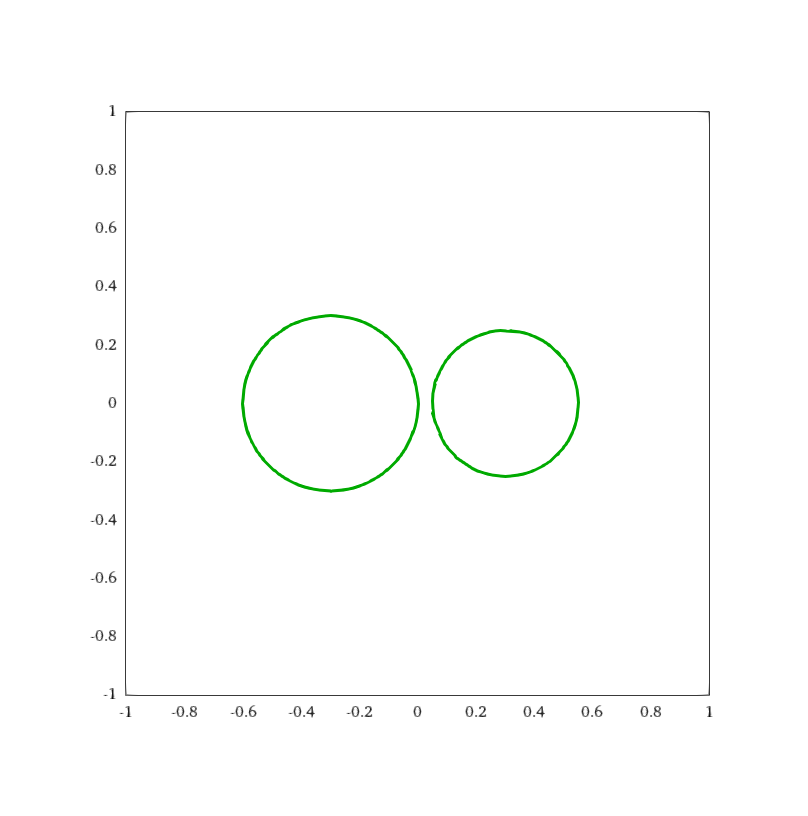}
        }
        \subfloat[$t=0.004$]{
            \includegraphics[trim={2cm 2cm 2cm 2cm},clip,width=0.275\textwidth]{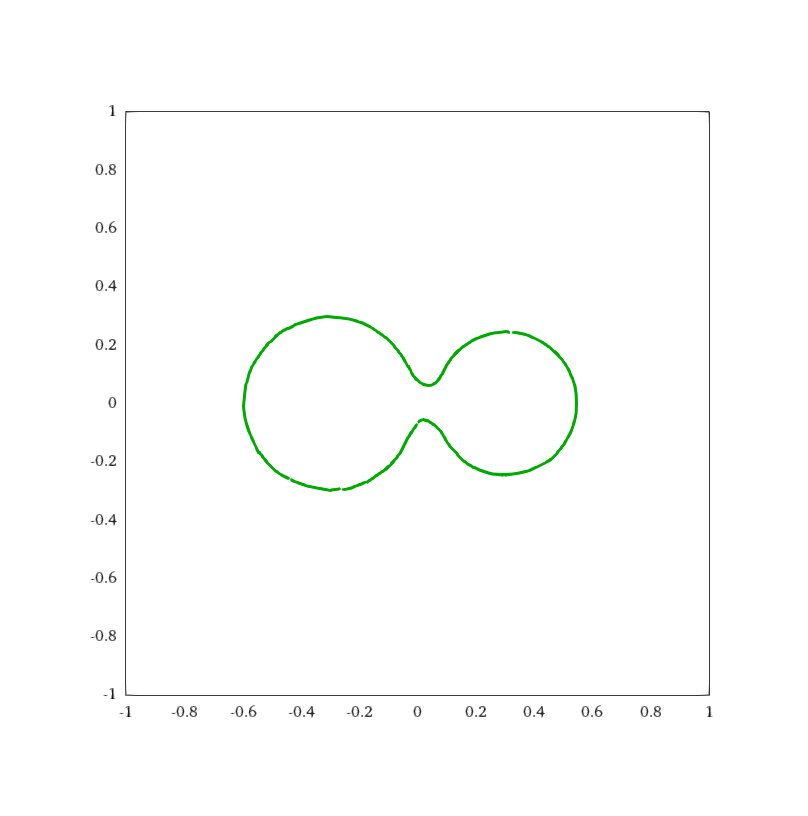}
        }
        \subfloat[$t=0.016$]{
            \includegraphics[trim={2cm 2cm 2cm 2cm},clip,width=0.275\textwidth]{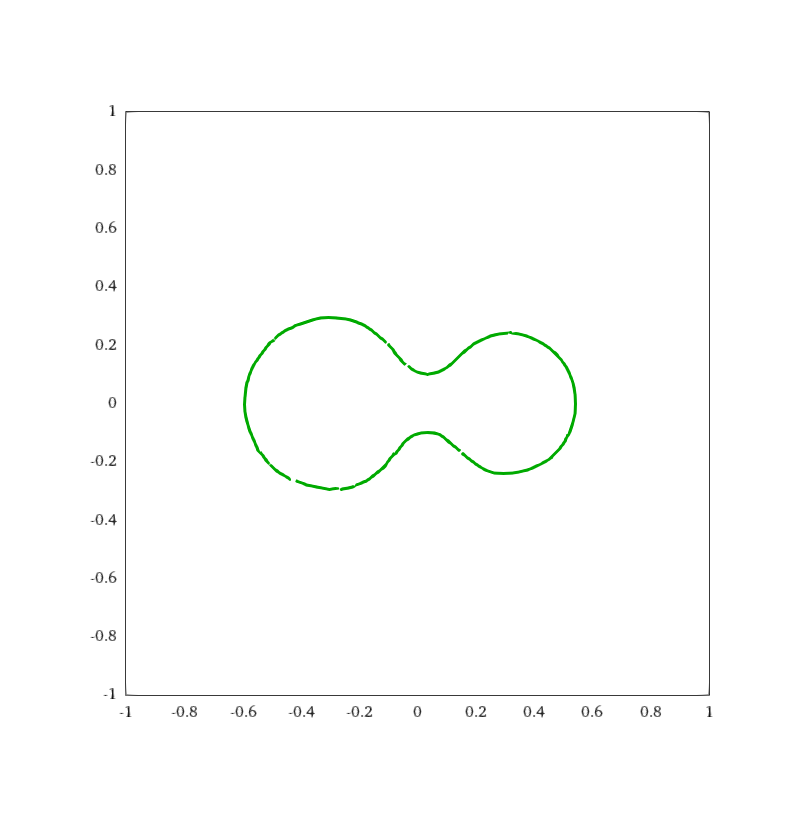}
        }
        \\
        \subfloat[$t=0.048$]{
            \includegraphics[trim={2cm 2cm 2cm 2cm},clip,width=0.275\textwidth]{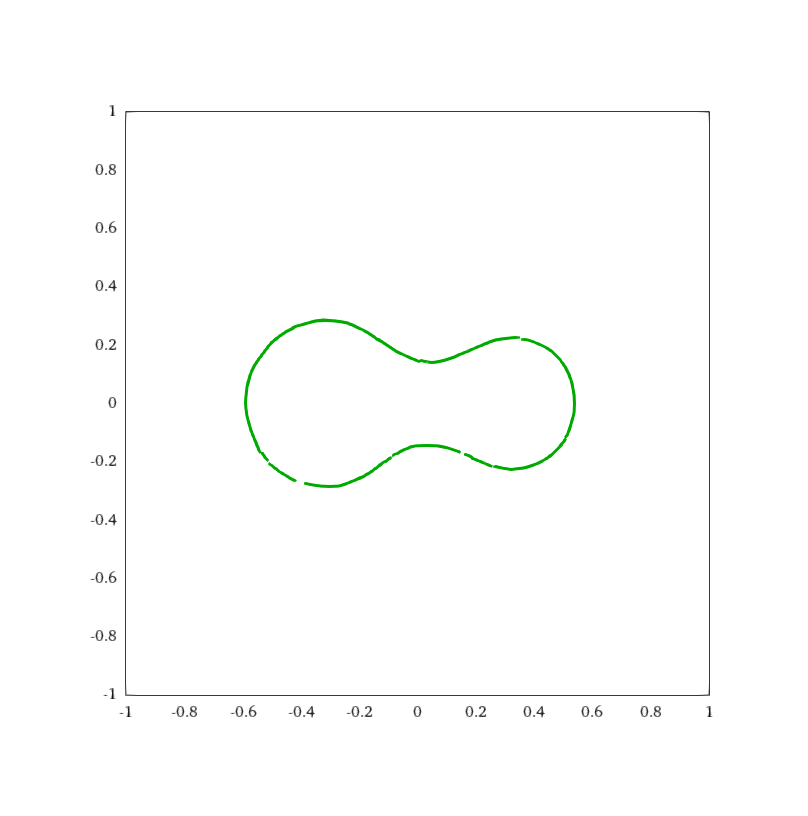}
        }
        \subfloat[$t=0.144$]{
            \includegraphics[trim={2cm 2cm 2cm 2cm},clip,width=0.275\textwidth]{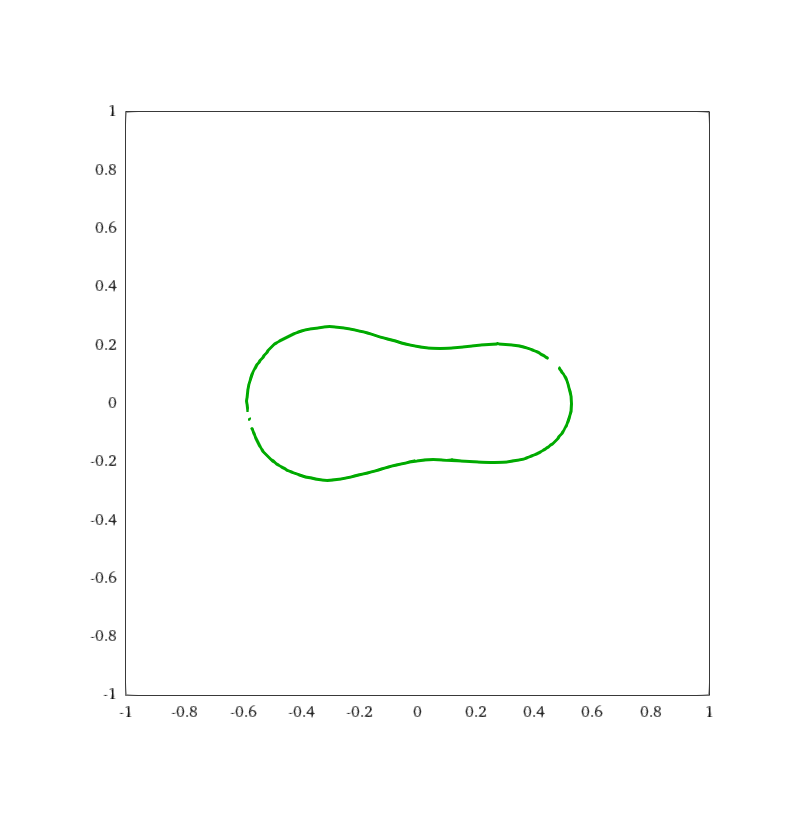}
        }
        \subfloat[$t=0.3$]{
            \includegraphics[trim={2cm 2cm 2cm 2cm},clip,width=0.275\textwidth]{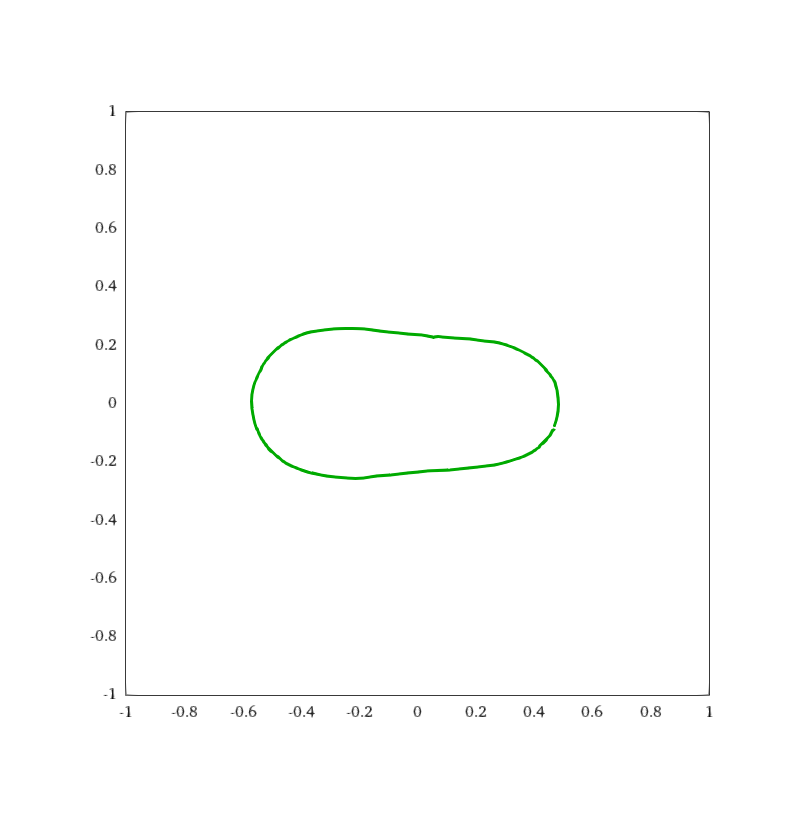}
        }
    \end{center}
        \caption{
            \jsccorrections{
            Test 2 (section~\ref{subsec: test 2}) with the $C^1$-conforming VEM. Screenshots of the zero-level sets at times $t=0,0.004,0.016,0.048,0.144,0.3$ with $\eps=3/100$ are displayed on the $25 \times 25$ Voronoi polygonal grid.}
            }
    \label{fig: zero level sets C1conf version}
    \end{figure}
}
{
    \begin{figure}
        \begin{center}        
        \subfloat[$t=0$]{
            \includegraphics[width=0.275\textwidth]{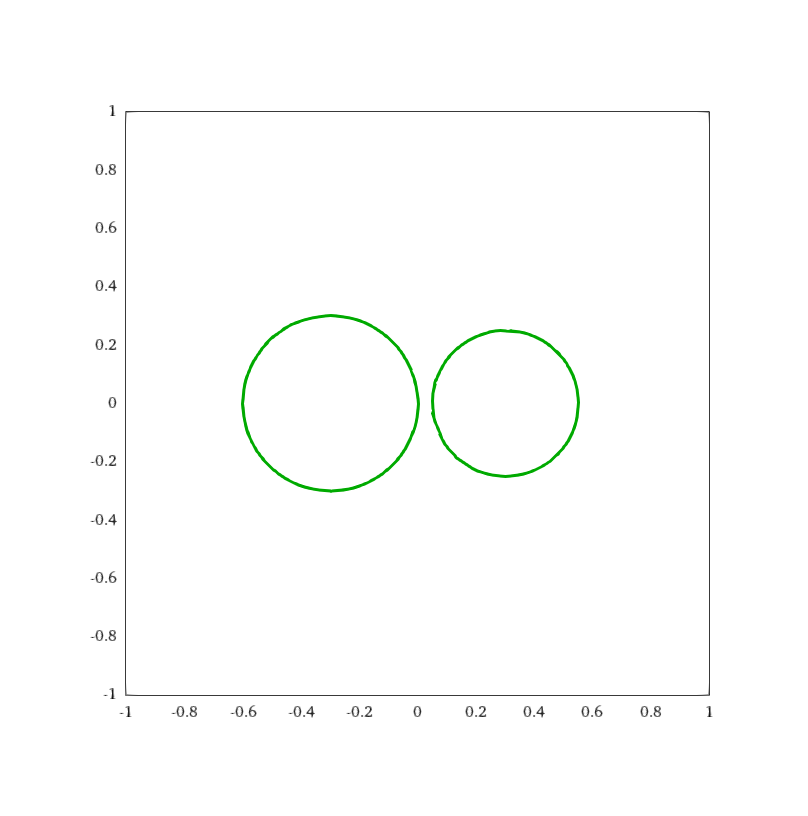}
        }
        \subfloat[$t=0.004$]{
            \includegraphics[width=0.275\textwidth]{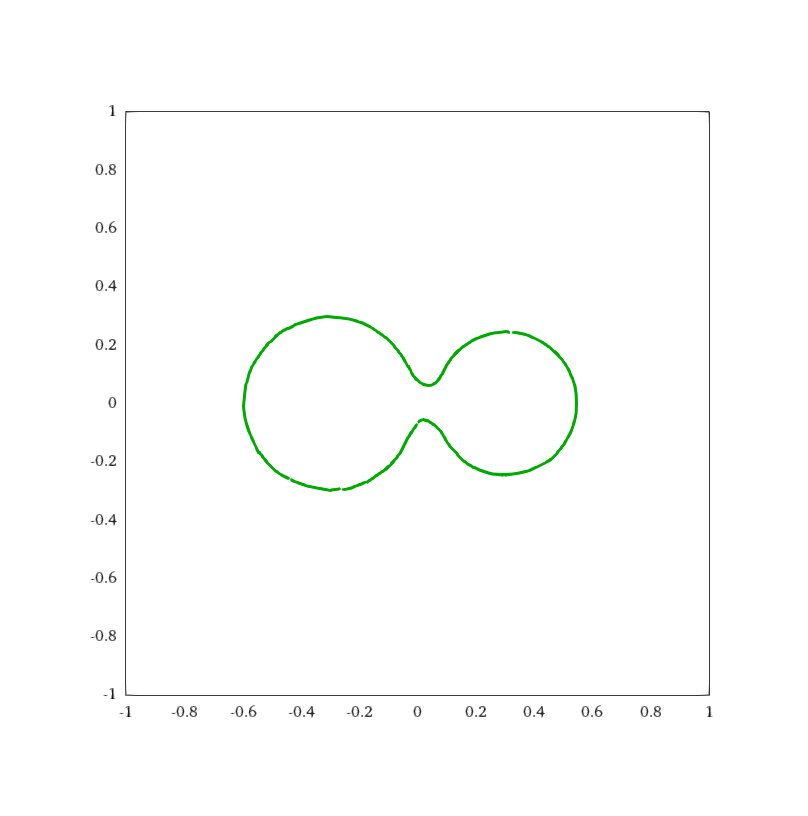}
        }
        \subfloat[$t=0.016$]{
            \includegraphics[width=0.275\textwidth]{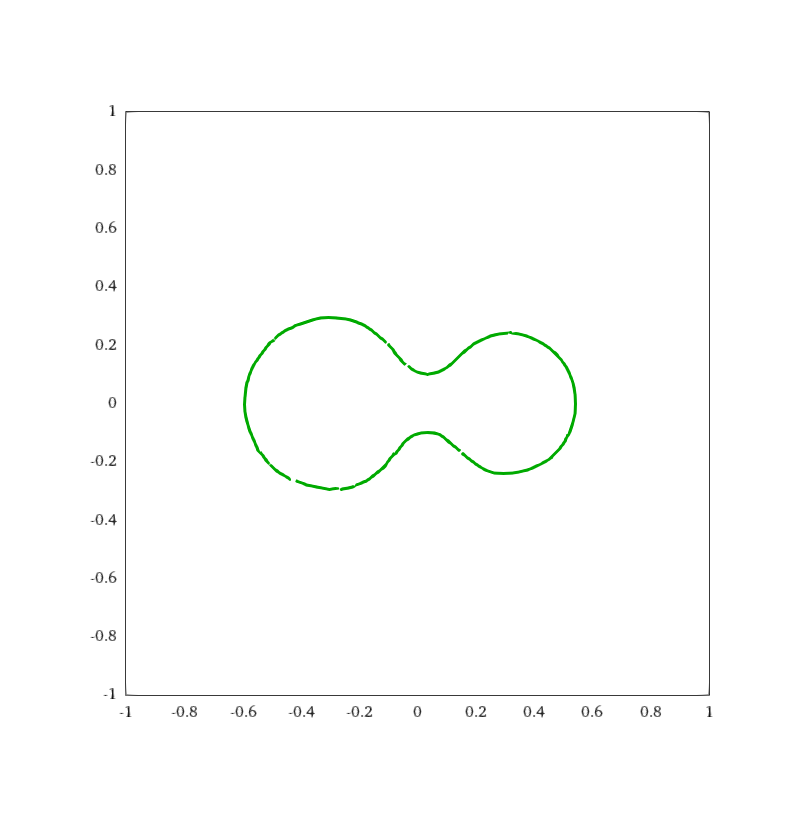}
        }
        \\
        \subfloat[$t=0.048$]{
            \includegraphics[width=0.275\textwidth]{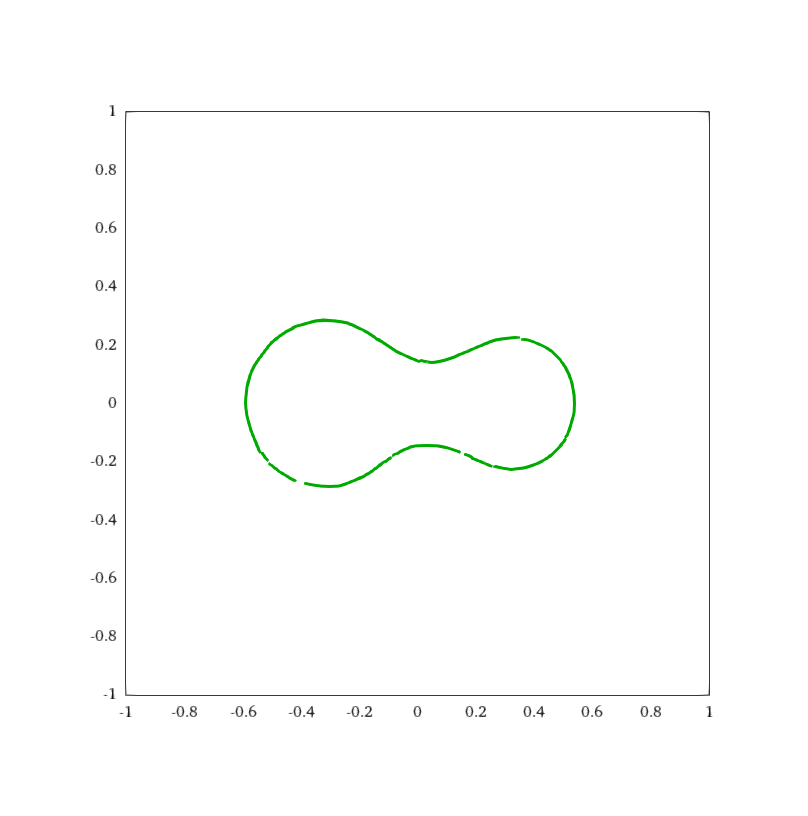}
        }
        \subfloat[$t=0.144$]{
            \includegraphics[width=0.275\textwidth]{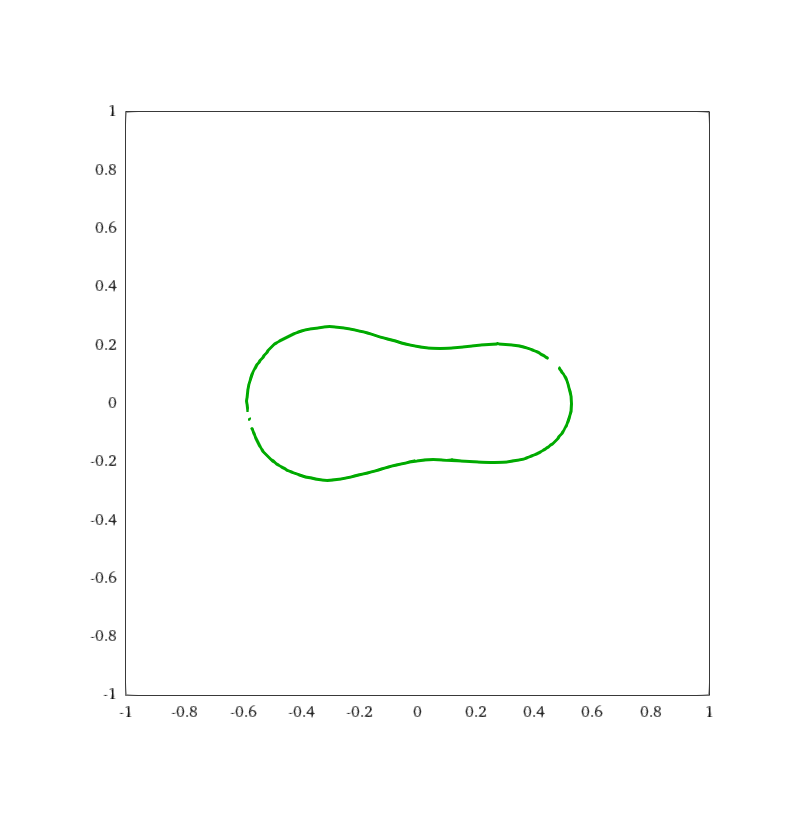}
        }
        \subfloat[$t=0.3$]{
            \includegraphics[width=0.275\textwidth]{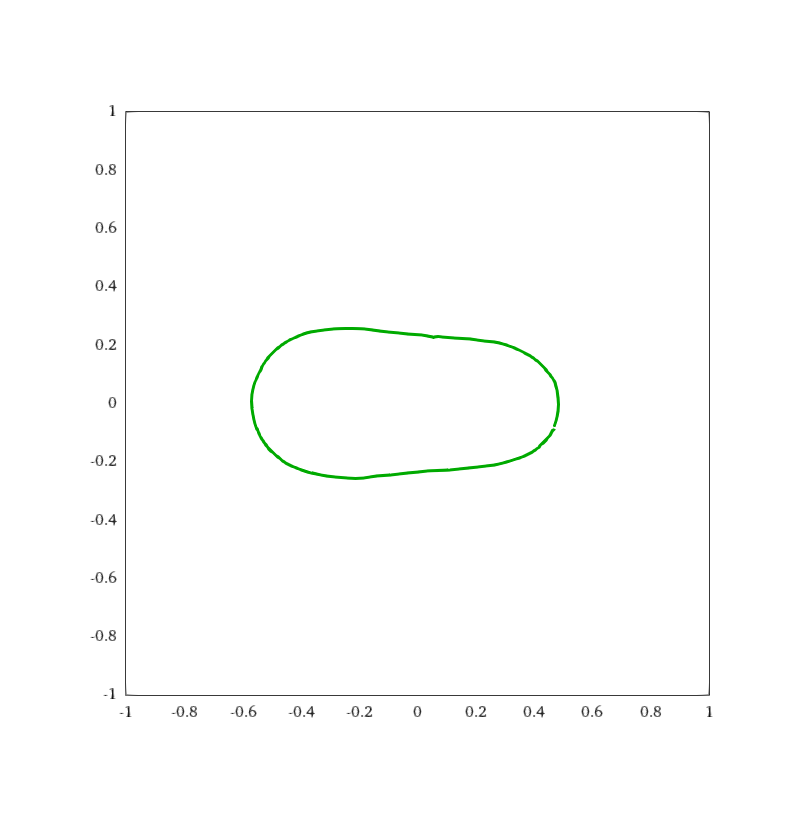}
        }
    \end{center}
        \caption{
            \jsccorrections{
            Test 2 (section~\ref{subsec: test 2}) with the $C^1$-conforming VEM. Screenshots of the zero-level sets at times $t=0,0.004,0.016,0.048,0.144,0.3$ with $\eps=3/100$ are displayed on the $25 \times 25$ Voronoi polygonal grid.}
            }
    \label{fig: zero level sets C1conf version}
    \end{figure}
}

\FloatBarrier
\section{Conclusion}\label{sec: conclusion}
In this paper we have developed a fully nonconforming virtual element method of arbitrary approximation order for the discretization of the two dimensional Cahn-Hilliard equation. 
We applied the projection approach taken in \cite{10.1093/imanum/drab003} to the fourth-order nonlinear Cahn-Hilliard problem and were able to define the discrete forms directly without using an averaging technique for the nonlinear term, as seen in \cite{antonietti_$c^1$_2016}.
This approach enabled us to prove optimal order error estimates in \corrections{the} $L^2$\corrections{-norm} for the semidiscrete (continuous-in-time) scheme. 
We coupled the VEM spatial discretization with a convex splitting Runge-Kutta (CSRK) method to create a fully discrete scheme and the behaviour was investigated with numerical experiments. 
The theoretical convergence result was verified numerically and standard benchmark tests from the literature were carried out, \jsccorrections{as well as a numerical comparison with a conforming VEM scheme.}
%\cleardoublepage
\appendix
\section{Proof of technical results}\label{sec: appendix}
We dedicate this section to the proof of some technical lemmas necessary for the error analysis. Firstly, we give the proof of Lemma \ref{lemma: term 2}.
\begin{proof}[Proof of Lemma \ref{lemma: term 2}]
    We begin by introducing the interpolation $\globalInterpolation z$ of $z$ into our nonconforming VEM space $\vemSpace$,
\begin{align*}
    \left| b(u-\ellipticProj u,z) \right| = \left| b(u-\ellipticProj u,z-\globalInterpolation z) + b(u-\ellipticProj u,\globalInterpolation z) \right|.
\end{align*} 
The first term can be bounded easily using the continuity of the bilinear form $b(\cdot,\cdot)$, \eqref{eqn: elliptic proj in 2norm}, interpolation properties of $\globalInterpolation$ \eqref{eqn: approximation properties of interpolation operator}, and \eqref{eqn: elliptic regularity}. 
Hence, we can show that 
\begin{align*}
    \left| b( u-\ellipticProj u,z - \globalInterpolation z) \right| 
    \leq 
    \| u - \ellipticProj u \|_{2,h} 
    \| z - \globalInterpolation z \|_{2,h} 
    \lesssim 
    h^{\polOrder} \| u - \ellipticProj u \|_{1,h}.
\end{align*}
% \corrections/{using \eqref{eqn: elliptic proj in 2norm}, interpolation properties of $z$, and .} 

We rewrite the remaining term \corrections{introducing} $z_h := \globalInterpolation z - z$, and see that 
\begin{align*}
    b(u-\ellipticProj u,\globalInterpolation z) = b(u,z_h) + b(u,z) - b(\ellipticProj u,\globalInterpolation z). 
\end{align*}

We note that the continuous solution $u$ satisfies
\begin{align*}
    b(u,z) = \eps^2 a(u,z) + r(u;u,z) + \alpha (u,z) = (\eps^2 \Delta^2 u - \nabla \cdot(\phi^{\prime}(u)\nabla u )+ \alpha u,z).
\end{align*}
Therefore, using the definition of the elliptic projection in \eqref{eqn: elliptic proj}, it follows that 
\begin{align*}
    b(u,z_h)+ b(u,z) 
    = \ & \eps^2 a(u,z_h) + r(u;u,z_h) + \alpha (u,z_h) 
    + 
    b(u,z)
    \\
    = \ & 
    ( \eps^2 \Delta^2 u - \nabla \cdot (\phi^{\prime}(u) \nabla u ) + \alpha u, z_h + z) 
    +
    \nonconformity(u,z_h)
    \\
    = \ &
    b_h(\ellipticProj u,\globalInterpolation z) + \nonconformity (u, z_h),
\end{align*}
where $\nonconformity (u,z_h)$ is defined in \eqref{eqn: nonconformity defn}. 
Noting that, by definition,
\begin{align*}
    b_h(\ellipticProj u,\globalInterpolation z) = \ &b(\ellipticProj u,\globalInterpolation z)
    + 
    \eps^2 (a_h(\ellipticProj u,\globalInterpolation z) - a(\ellipticProj u,  \globalInterpolation z )  ) 
    + r_h (u;\ellipticProj u,\globalInterpolation z) - r(u;\ellipticProj u,\globalInterpolation z),
\end{align*}
we observe that 
\begin{align*}
    \left| b(u-\ellipticProj u,\globalInterpolation z) \right|
    = \ &
    \left|
    % b(u,\globalInterpolation z) - b(\ellipticProj u,\globalInterpolation z)
    % \\
    % &=
    % b(u,\globalInterpolation z -z) + b(u,z) - b(\ellipticProj u,\globalInterpolation z)
    % \\
    % &=
    b(u,z_h) + b(u,z) - b(\ellipticProj u,\globalInterpolation z) - b_h(\ellipticProj u,\globalInterpolation z) + b_h(\ellipticProj u,\globalInterpolation z)
    \right|
    \\
    % &=
    % \nonconformity(u,z_h) + b_h(\ellipticProj u,\globalInterpolation z) - b(\ellipticProj u,\globalInterpolation z)
    % \\
    \leq \ &
    \left| \nonconformity (u,z_h) \right|
    +
    \eps^2 \left| a_h(\ellipticProj u,\globalInterpolation z) - a(\ellipticProj u,  \globalInterpolation z )  \right| 
    + \left|  r_h (u;\ellipticProj u,\globalInterpolation z) - r(u;\ellipticProj u,\globalInterpolation z ) \right|
    \\
    = \ &
    T_1 + T_2 + T_3.
\end{align*}

We now look at each term in turn.
For $T_1$ we apply Lemma \ref{lemma: nonconformity error} and \eqref{eqn: elliptic regularity} to see that 
\begin{align}\label{eqn: term T_1}
    T_1 \lesssim h^{\polOrder-1} |z-\globalInterpolation z|_{2,h} \lesssim h^{\polOrder} \| z \|_{3,\Omega} \lesssim h^{\polOrder} \| u-\ellipticProj u\|_{1,h}.
\end{align}

For the second term $T_2$, we introduce the $L^2$ projection of $z$, as well as the $L^2$ projection of $u$, and use polynomial consistency from Lemma \ref{lemma: consistency}. Therefore
    \begin{align}\label{eqn: term T2 bound}
        T_2 &= \eps^2 \left| a_h(\ellipticProj u,\globalInterpolation z) - a(\ellipticProj u,\globalInterpolation z) \right|
        \nonumber
        \\
        &=
        \eps^2 \left| \sum_{K \in \mesh} a_h^K (\globalInterpolation z - \ltwo_{\polOrder} z,\ellipticProj u - \ltwo_{\polOrder} u) 
        - 
        a^K( \globalInterpolation z - \ltwo_{\polOrder} z  , \ellipticProj u - \ltwo_{\polOrder} u )
        \right|
        % \\ 
        % &=  \eps^2 \sum_{K \in \mesh} a_h^K (\globalInterpolation z - \ltwo_{\polOrder} z,\ellipticProj u) + a^K(\ltwo_{\polOrder} z , \ellipticProj u) - (D^2 \globalInterpolation z,D^2 \ellipticProj u)
        % \\
        % &=  \eps^2 \sum_{K \in \mesh} a_h^K (\globalInterpolation z - \ltwo_{\polOrder} z,\ellipticProj u - \ltwo_{\polOrder} u) + a^K(\globalInterpolation z - \ltwo_{\polOrder} z,\ltwo_{\polOrder} u)  + a^K(\ltwo_{\polOrder} z , \ellipticProj u) - (D^2 \globalInterpolation z,D^2 \ellipticProj u)
        % \\
        % &= \eps^2 \sum_{K \in \mesh} a_h^K (\globalInterpolation z - \ltwo_{\polOrder} z,\ellipticProj u - \ltwo_{\polOrder} u) - ( D^2 ( \globalInterpolation z - \ltwo_{\polOrder} z ) , D^2 ( \ellipticProj u - \ltwo_{\polOrder} u ))
        \nonumber
        \\
        &\lesssim
        \eps^2 | \globalInterpolation z - \ltwo_{\polOrder} z |_{2,h} | \ellipticProj u - \ltwo_{\polOrder} u |_{2,h}
        \lesssim
        h \| z \|_{3,\Omega} h^{\polOrder-1}
        \lesssim h^{\polOrder} \| u - \ellipticProj u \|_{1,h}.
    \end{align}
    
    \ifthenelse{\boolean{thesis}}{
    In the last step we have applied stability of the discrete form from Lemma \ref{lemma: stability}. Then, using the triangle inequality, approximation properties of both the interpolation and the $L^2$ projection, it holds that 
    \begin{align*}
        | \globalInterpolation z - \ltwo_{\polOrder} z |_{2,h} 
        &\leq   
        | \globalInterpolation z - z |_{2,h} + | z - \ltwo_{\polOrder} z |_{2,h} 
        \lesssim
        h \| z \|_{3,\Omega} 
        \lesssim
        h \| u - \ellipticProj u \|_{1,h}
        \intertext{and}
        | \ellipticProj u - \ltwo_{\polOrder} u |_{2,h} 
        &\leq
        | \ellipticProj u - u |_{2,h} + |u - \ltwo_{\polOrder} u |_{2,h} 
        \lesssim h^{\polOrder-1},
    \end{align*}
    therefore 
    \begin{align}
        T_2 \lesssim h^{\polOrder} \| u - \ellipticProj u \|_{1,h}.
    \end{align}
    }
    {}

    For the final term, $T_3$, we note that due to the fact that the gradient projection is exact for polynomials, it holds that $r_h(u;\ltwo_{\polOrder}u,\ltwo_{\polOrder}z) =r(u;\ltwo_{\polOrder}u,\ltwo_{\polOrder}z)$. Therefore,
    \begin{align}
        T_3 = \ & \left| r_h(u;\ellipticProj u,\globalInterpolation z) - r(u;\ellipticProj u,\globalInterpolation z) \right|
        % \\
        % = \ & 
        % r_h(u;\ellipticProj u - \ltwo_{\polOrder} u,\globalInterpolation z) + r_h(u;\ltwo_{\polOrder} u,\globalInterpolation z)
        % \\ 
        % &- r(u;\ellipticProj u - \ltwo_{\polOrder}u, \globalInterpolation z) - r(u;\ltwo_{\polOrder} u,\globalInterpolation z)
        \nonumber
        \\
        = \ & 
        | r_h(u;\ellipticProj u - \ltwo_{\polOrder} u,\globalInterpolation z - \ltwo_{\polOrder} z) 
        + r_h(u;\ltwo_{\polOrder} u,\globalInterpolation z) 
        + r_h(u;\ellipticProj u, \ltwo_{\polOrder} z)
        \nonumber
        \\
        &- r_h (u; \ltwo_{\polOrder} u, \ltwo_{\polOrder} z)
        +
        r(u;\ltwo_{\polOrder}u,\ltwo_{\polOrder} z)
        \nonumber
        \\ 
        &- r(u;\ellipticProj u - \ltwo_{\polOrder}u, \globalInterpolation z - \ltwo_{\polOrder} z) - r(u;\ltwo_{\polOrder} u,\globalInterpolation z) - r(u;\ellipticProj u, \ltwo_{\polOrder} z) 
        |
        \nonumber
        \\
        \leq \ &
        \left| r_h(u;\ellipticProj u - \ltwo_{\polOrder} u,\globalInterpolation z - \ltwo_{\polOrder} z)
        - r(u;\ellipticProj u - \ltwo_{\polOrder}u, \globalInterpolation z - \ltwo_{\polOrder} z) \right|
        \label{eqn: first T_3}
        \\
        &+ \left| r_h(u;\ltwo_{\polOrder} u,\globalInterpolation z) - r(u;\ltwo_{\polOrder} u,\globalInterpolation z) \right|
        \label{eqn: second T_3}
        \\
        &
        + \left| r_h(u;\ellipticProj u, \ltwo_{\polOrder} z) - r(u;\ellipticProj u, \ltwo_{\polOrder} z) \right|. 
        \label{eqn: third T_3}
    \end{align}
    The first term \eqref{eqn: first T_3} can be treated in the same way as $T_2$ and we can show that 
    \begin{align}\label{eqn: t3,1}
        | r_h(u;\ellipticProj u &- \ltwo_{\polOrder} u,\globalInterpolation z - \ltwo_{\polOrder} z)
        - 
        r(u;\ellipticProj u - \ltwo_{\polOrder}u, \globalInterpolation z - \ltwo_{\polOrder} z) |
        \nonumber
        \\
        &\lesssim
        \| \phi^{\prime} (u) \|_{L^{\infty}} | \ellipticProj u - \ltwo_{\polOrder} u |_{1,h} |\globalInterpolation z - \ltwo_{\polOrder} z|_{1,h}
        \lesssim
        h^{\polOrder} \| u- \ellipticProj u\|_{1,h}.
    \end{align}

    For the next term \eqref{eqn: second T_3}, since $\ltwo_{\polOrder} u \in \prob_{\polOrder}(K)$, the stabilization part of $r_h$ vanishes and so we have the following 
    \begin{align}\label{eqn: t3,2}
        \left| r_h(u;\ltwo_{\polOrder} u,\globalInterpolation z) \right. - \left. r(u;\ltwo_{\polOrder} u,\globalInterpolation z) 
        \right|
        = \ &
        \left|
        \sum_{K \in \mesh} \int_{K} \phi^{\prime} (u) \, \big( \ltwo_{\polOrder-1} ( \nabla \ltwo_{\polOrder} u) \cdot \ltwo_{\polOrder-1} \nabla \globalInterpolation z 
        - \nabla \ltwo_{\polOrder} u \cdot \nabla \globalInterpolation z 
        \big)
        \, \dx 
        \right|
        \nonumber
        \\
        = \ &
        \left|
        \sum_{K \in \mesh} \int_{K} \phi^{\prime} (u) \,  (\nabla \ltwo_{\polOrder} u) \cdot ( \ltwo_{\polOrder-1} - I ) \nabla \globalInterpolation z 
        \, \dx
        \right|
        \nonumber
        \\  
        = \ &
        \left|
        \sum_{K \in \mesh} \int_{K} ( \ltwo_{\polOrder-1} - I )  (\phi^{\prime} (u) \nabla \ltwo_{\polOrder} u ) \cdot ( \nabla \globalInterpolation z - \ltwo_{0} \nabla  z)
        \, \dx
        \right|
        \nonumber
        \\
        \leq \ &
        \sum_{K \in \mesh} 
        \|  (\ltwo_{\polOrder-1} - I )  \phi^{\prime} (u)  \nabla \ltwo_{\polOrder} u  \|_{0,K} \| \nabla \globalInterpolation z - \ltwo_{0} \nabla z \|_{0,K}
        \nonumber
        \\
        \lesssim \ &
        h^{\polOrder-1} h \| z \|_{3,\Omega}
        \lesssim h^{\polOrder} \| u - \ellipticProj u \|_{1,h}.
    \end{align}

    \ifthenelse{\boolean{thesis}}{
    Note that 
    \begin{align*}
        \| (I- \ltwo_{\polOrder-1}) ( \phi^{\prime}(u) \nabla \ltwo_{\polOrder} u ) \|_{0,K} 
        \leq 
        h^{\polOrder-1} | \phi^{\prime}(u) \nabla \ltwo_{\polOrder} u |_{\polOrder-1,K}
        % \leq 
        % h^{\polOrder-1} \| \phi^{\prime}(u) \|_{W^{\polOrder-1,\infty}} | \nabla u |_{\polOrder-1,K} 
        &\lesssim
        h^{\polOrder-1}
    \end{align*}
    and therefore combined with \eqref{eqn: h from z} and \eqref{eqn: elliptic regularity}, it follows that 
    \begin{align}
        \left|
        r_h(u;\ltwo_{\polOrder} u,\globalInterpolation z) 
        \right. 
        &- \left. r(u;\ltwo_{\polOrder} u,\globalInterpolation z)  \right| 
        \lesssim 
        h^{\polOrder} \| u - \ellipticProj u \|_{1,h}.
    \end{align}
    }
    {}

    For the final term in $T_3$, \eqref{eqn: third T_3}, we again note that since $\ltwo_{\polOrder} z \in \prob_{\polOrder}(K)$, the stabilization part of $r_h$ vanishes, and so 
    \begin{align}\label{eqn: t3 modified}
        | r_h(u;\ellipticProj u, \ltwo_{\polOrder} z) &- r(u;\ellipticProj u, \ltwo_{\polOrder} z) |
        =
        \left|
        \sum_{K \in \mesh} \int_K \phi^{\prime}(u) \left( 
            \ltwo_{\polOrder-1} \nabla \ellipticProj u \cdot \ltwo_{\polOrder-1} \nabla \ltwo_{\polOrder} z 
            - 
            \nabla \ellipticProj u \cdot \ltwo_{\polOrder} z \right)
            \, \dx 
        \right|
        % \\
        % = \ &
        % \sum_{K \in \mesh} \int_K \phi^{\prime}(u) 
        %     ( \ltwo_{\polOrder-1} \nabla \ellipticProj u 
        %     - 
        %     \nabla \ellipticProj u ) \cdot \nabla \ltwo_{\polOrder} z 
        %     \, \dx 
        \nonumber
        \\
        = \ &
        \left|
        \sum_{K \in \mesh} \int_K \phi^{\prime}(u) 
            ( \ltwo_{\polOrder-1} - I) \nabla \ellipticProj u 
            \cdot (\nabla \ltwo_{\polOrder} z )
            \, \dx 
        \right|
        % \\
        % &+ 
        % \int_K \phi^{\prime}(u) 
        %     ( \ltwo_{\polOrder-1} - I) \nabla \ellipticProj u 
        %     \cdot \overline{\nabla z} 
        %     \, \dx 
        \nonumber
        \\
        = \, & 
        \left|
        \sum_{K \in \mesh} \int_K 
        \phi^{\prime}(u) \big( 
            (I - \ltwo_{\polOrder-1} ) ( \nabla u - \nabla \ellipticProj u ) - (I-\ltwo_{\polOrder-1} ) \nabla u 
            \big) \cdot (\nabla \ltwo_{\polOrder} z)
        \, \dx
        \right|
        \nonumber
        % \\
        % \leq \, &
        % \sum_{K \in \mesh} \|  \phi^{\prime}(u) \big( 
        %     (I - \ltwo_{\polOrder-1} ) ( \nabla u - \nabla \ellipticProj u ) \|_{0,K} \| \nabla \ltwo_{\polOrder} z \|_{0,K}
        % \nonumber
        % \\
        % &+ 
        % \sum_{K \in \mesh} \| \phi^{\prime}(u) \big( 
        %     (I - \ltwo_{\polOrder-1} ) \nabla u \|_{0,K} \| \nabla \ltwo_{\polOrder} z \|_{0,K}
        % \nonumber 
        \\
        \leq \, & 
        \sum_{K \in \mesh}  \| \phi^{\prime} (u) \|_{L^{\infty}} h^{1} | \nabla u - \nabla \ellipticProj u |_{1,K}  \| \nabla \ltwo_{\polOrder} z \|_{0,K}
        \nonumber
        \\
        &+
        \sum_{K \in \mesh} 
        \| \phi^{\prime}(u) \|_{L^{\infty}} h^{\polOrder-1+1} | \nabla u |_{\polOrder,K} \| \nabla \ltwo_{\polOrder} z \|_{0,K} 
        \nonumber
        \\
        \lesssim \ &
        h^{\polOrder} \| u - \ellipticProj u \|_{1,h},
    \end{align}
    where we have used \eqref{eqn: elliptic proj in 2norm}, \eqref{eqn: elliptic regularity}, and stability of the $L^2$ projection in the last step. Hence, by combining the estimates from \eqref{eqn: term T_1}, \eqref{eqn: term T2 bound}, \eqref{eqn: t3,1}, \eqref{eqn: t3,2}, and \eqref{eqn: t3 modified}, it holds that
    \begin{align*}
        | b(u-\ellipticProj u,\globalInterpolation z) | \leq T_1 + T_2 + T_3 \lesssim h^{\polOrder} \| u-\ellipticProj u \|_{1,h}.
    \end{align*}
    This concludes the proof.
\end{proof}

Next, we give the proof of Lemma \ref{lemma: extra terms}, necessary for the proof of the estimates \eqref{eqn: elliptic proj time in 2 norm}-\eqref{eqn: elliptic proj time in 1 norm} in Lemma \ref{lemma: elliptic proj bounds}.
\begin{proof}[Proof of Lemma \ref{lemma: extra terms}]
    Using Lemma \ref{eqn: projections and l2 projection result}, since $\ellipticProj u, \eta_h \in \vemSpace$, we can show the following
    \begin{align*}
        | (\phi^{\prime \prime}(u) u_t \, &\gradProj \ellipticProj u,\gradProj \eta_h )_{0,h} 
        - 
        (\phi^{\prime \prime}(u) u_t \nabla u,\nabla \eta_h)_{0,h} |
        \\
        = \ &
        \left|
        \sum_{K \in \mesh} \int_{K} \phi^{\prime \prime} (u) u_t \big( \ltwo_{\polOrder-1} \nabla \ellipticProj u \cdot \ltwo_{\polOrder-1} \nabla \eta_h
        -
        \nabla u \cdot \nabla \eta_h
              \big)    
              \, \dx 
        \right|
        \\
        = \ &
        \left| 
        \sum_{K \in \mesh} \int_{K} 
        \Big( \ltwo_{\polOrder-1} 
        \big( \phi^{\prime \prime} (u) u_t \ltwo_{\polOrder-1} \nabla \ellipticProj u \big) 
        -
        \phi^{\prime \prime} (u) u_t
        \nabla u \Big)
        \cdot \nabla \eta_h 
        \, \dx 
        \right|.
    \intertext{We now introduce the constant projection of the gradient of $\eta_h$, $\ltwo_0 (\nabla \eta_h)$ and see that}
    | (\phi^{\prime \prime}(u) u_t \, &\gradProj \ellipticProj u,\gradProj \eta_h )_{0,h} 
        - 
        (\phi^{\prime \prime}(u) u_t \nabla u,\nabla \eta_h)_{0,h} |
        \\
        \leq \ & 
        \left|
        \sum_{K \in \mesh} \int_{K} 
        \Big( \ltwo_{\polOrder-1} 
        \big\{ \phi^{\prime \prime} (u) u_t \ltwo_{\polOrder-1} \nabla \ellipticProj u 
        -
        \phi^{\prime \prime} (u) u_t
        \ltwo_{\polOrder-1} \nabla u \big\} \Big)
        \cdot \nabla \eta_h 
        \, \dx 
        \right|
        \\
        &+
        \left|
        \sum_{K \in \mesh} \int_{K} 
        \Big( ( \ltwo_{\polOrder-1} - I) \big( \phi^{\prime \prime} (u) u_t
        \ltwo_{\polOrder-1} \nabla u \big) 
        \Big)
        \cdot \big( \nabla \eta_h - \ltwo_0 (\nabla \eta_h) \big)
        \, \dx 
        \right|
        \\
        &+
        \left|
        \sum_{K \in \mesh} \int_{K} 
        \Big( 
        \phi^{\prime \prime}(u) u_t \ltwo_{\polOrder-1} \nabla u
        -
        \phi^{\prime \prime}(u) u_t \nabla u
        \Big)
        \cdot \nabla \eta_h 
        \, \dx 
        \right|
        \\
        \leq \ & 
        \sum_{K \in \mesh} 
        \| \phi^{\prime \prime} (u) u_t 
        \ltwo_{\polOrder-1} \big( \nabla \ellipticProj u 
        -
        \nabla u \big)
        \|_{0,K} \| \nabla \eta_h \|_{0,K}
        \\
        &+ 
        \sum_{K \in \mesh} h^{\polOrder-1} |\phi^{\prime \prime} (u) u_t
        \ltwo_{\polOrder-1} \nabla u|_{\polOrder-1} \| \nabla \eta_h - \ltwo_{0} (\nabla \eta_h) \|_{0,K} 
        \\
        &+ 
        \sum_{K \in \mesh} 
        \| \phi^{\prime \prime} (u) u_t \|_{L^{\infty}} 
        \| \ltwo_{\polOrder-1} \nabla u - \nabla u \|_{0,K} \| \nabla \eta_h \|_{0,K}
        \\
        \lesssim \ & 
        h^{\polOrder} \| \eta_h \|_{2,h},
    \end{align*}
where we have used stability of the $L^2$ projection and \eqref{eqn: elliptic proj in 1norm}.
\end{proof}

Lastly, we give the proof of Lemma \ref{lemma: boundary term bound}, which is necessary for the proof of $L^2$ convergence in Theorem \ref{thm: L2 convergence}.
\begin{proof}[Proof of Lemma \ref{lemma: boundary term bound}]
    Recall that for any $w_h \in \vemSpace \subset \HTwoNCSpace$, and for any edge $e \in \mathcal{E}_h$, the following properties hold,
    \begin{alignat}{3}
        \int_e [ w_h ] p \ \ds &= 0 \quad &&\forall p \in \prob_{\polOrder-3}(e),&&
        \label{eqn: jump w_h}
        \\
        \int_e [ \partial_n w_h ] p \ \ds &= 0 \quad &&\forall p \in \prob_{\polOrder-2}(e).&&
        \label{eqn: jump normal}
    \end{alignat}
    Observe that the following holds
    \begin{align*}
        \left| \sum_{K \in \cT_h} \int_{\partial K} (\partial_n z_h) w_h \, \ds \right|
        &= \left| \sum_{e \in \mathcal{E}_h} \int_e \big( \{ w_h \}[ \partial_n z_h ] + \{\partial_n z_h \} [w_h] \big) \, \ds
        \right|
        \leq
        A_{I} + A_{II}
    \end{align*}
    where we use $\{ \cdot \}$ to denote the average of a function $v$, $ \{ v \} := \half (v|_{K^+} + v|_{K^-})$ for any interior edge $e \subset \partial K^{+} \cap \partial K^{-}$. For a boundary edge $e \in \mathcal{E}_h^{\text{bdry}}$ we define $\{ v \} := v|_e$.

    Using \eqref{eqn: jump normal}, for the first term $A_{I}$, it holds that
    \begin{align*}
        A_{I} = \left|  \sum_{e \in \mathcal{E}_h} \int_{e}  \{ w_h \}  [ \partial_n z_h ]\, \ds \right| 
        &\leq 
        \Big| \sum_{e \in \mathcal{E}_h} \int_{e} ( \{ w_h \} - \cP_0^e \{ w_h \} ) [  \partial_n z_h ] \, \ds \Big|
        \\
        &\leq 
        \Big| \sum_{e \in \mathcal{E}_h } \int_{e} ( \{ w_h \} - \cP_0^e \{ w_h \} ) ( [  \partial_n z_h ] - \cP_0^e [\partial_n z_h] ) \, \ds \Big| 
        \\
        % &\leq 
        % \sum_{e \in \mathcal{E}_h}
        % \| ( I - \cP_0^e ) \{ w_h \} \|_{0,e} \| [  \partial_n z_h  ] - \cP_0^e [  \partial_n z_h  ] \|_{0,e}
        % \\
        &\lesssim
        h^{1-\half} |w_h|_{1,h}
        h^{1-\half} | z_h |_{2,h}
        \\
        &\lesssim
        h|w_h|_{1,h} |z_h|_{2,h}.
    \end{align*}

    For the second term $A_{II}$, we follow the method used in \cite{antonietti_fully_2018}. In view of Remark~\ref{assump: star shaped wrt a ball}, for each edge we define the linear Lagrange interpolant $I^1_{T(e)}$ of $z_h$ on the sub-triangle $T(e)$, made from connecting the interior point $x_K$ to the endpoints of the edge $e$. 
    Since $z_h \in \vemSpace$, it also satisfies $z_h \in H^2(K)$, and so we can build the interpolant by using the values at the vertices of $T(e)$. 
    Further, $z_h$ is continuous at the endpoints of $e$ and so 
    \begin{align*}
        [ I^1_{T(e)} (w_h) ] |_{e} = 0.
    \end{align*}

    Therefore, by using standard interpolation estimates and a trace inequality, we see that
    \begin{align*}
        A_{II} = \left| \sum_{ e \in \mathcal{E}_h} \int_{e} \{ \partial_n z_h \} [ w_h ] \, \ds  
        \right|
        &= 
        \left|
        \sum_{e \in \mathcal{E}_h} \int_{e} \{ \partial_n z_h \} \big( [w_h] - [I^1_{T(e)} (w_h)] \big) \, \ds \right|
        \lesssim h | z_h |_{2,h} |w_h|_{2,h},
    \end{align*}
    hence \eqref{eqn: bdry lemma} holds, as required.
\end{proof}

\section*{Data Availability}
Enquiries about data availability should be directed to the authors.

\section*{Acknowledgements}
The authors would like to acknowledge the University of Warwick Scientific Computing Research Technology Platform for assistance in the research described in this paper.

\section*{Declarations}
The authors have no competing interests to declare that are relevant to the content of this article.

\bibliographystyle{acm}
\bibliography{VEMCahnHilliard}

\end{document}